\documentclass[a4paper,11pt]{article}
\usepackage[utf8]{inputenc}
\usepackage{multirow}%
\usepackage{fullpage}
\usepackage{amsmath,amsthm,amsfonts,amssymb}%
\usepackage[dvips]{graphicx}
\usepackage{enumerate}
\usepackage{dsfont}
\usepackage{array}
\usepackage{bm}
\usepackage{natbib}
\usepackage{color}
\usepackage[titletoc,title]{appendix}

\usepackage{paralist}

\numberwithin{equation}{section}

\usepackage{titletoc}
\usepackage[runin]{abstract}
\usepackage{calc}

\newtheorem{theorem}{Theorem}[section]
\newtheorem{fact}{Fact}

\newtheorem{proposition}[theorem]{Proposition}
\newtheorem{lemma}[theorem]{Lemma}

\theoremstyle{remark}

\newtheorem{example}[theorem]{Example}

\newcount\LongProof
\newcommand{\longproof}[2]{
  \ifnum\LongProof=0{#1}\fi
  \ifnum\LongProof=1{#2}\fi
}

\newcommand\R{\mathbb{R}}
\newcommand{\reals}{\R}

\newcommand\eps{\varepsilon}
\newcommand\prob{\mathbb{P}}    
\newcommand\ind{\mathds{1}}     
\newcommand{\1}{\ind}

\newcommand{\inv}{^{-}}

\newcommand{\point}{\,\cdot\,}

\newcommand{\abs}[1]{\left\lvert{#1}\right\rvert}
\newcommand{\dto}{\rightsquigarrow}

\DeclareMathOperator{\Var}{Var}

\newcommand{\argmin}{\operatornamewithlimits{\arg\min}}

\newcommand{\diff}{\mathrm{d}}

\renewcommand{\abs}[1]{\left\lvert{#1}\right\rvert}

\newcommand{\bu}{\bm{u}}
\newcommand{\bv}{\bm{v}}
\newcommand{\bw}{\bm{w}}

\newcommand{\bY}{\bm{Y}}
\newcommand{\by}{\bm{y}}

\allowdisplaybreaks[1]

\definecolor{auburn}{rgb}{0.43, 0.21, 0.1}
\definecolor{britishracinggreen}{rgb}{0.0, 0.26, 0.15}
\definecolor{burntumber}{rgb}{0.54, 0.2, 0.14}
\definecolor{carmine}{rgb}{0.59, 0.0, 0.09}

\newcommand{\expec}{\operatorname{\mathbb{E}}}

\begin{document}

\title{On the weak convergence of the empirical conditional copula under a simplifying assumption}

\author{Fran\c{c}ois Portier\footnote{LTCI, CNRS, T\'el\'ecom ParisTech, Universit\'e Paris-Saclay. {E-mail:} francois.portier@gmail.com} \and Johan Segers\footnote{Universit\'{e} catholique de Louvain,
	Institut de Statistique, Biostatistique et Sciences Actuarielles,
	Voie du Roman Pays~20,
	B-1348 Louvain-la-Neuve, Belgium. {E-mail:} johan.segers@uclouvain.be}
}
\maketitle

\renewcommand{\abstractname}{Abstract.}
\begin{abstract}
A common assumption in pair-copula constructions is that the copula of the conditional distribution of two random variables given a covariate does not depend on the value of that covariate. Two conflicting intuitions arise about the best possible rate of convergence attainable by nonparametric estimators of that copula. On the one hand, the best possible rates for estimating the marginal conditional distribution functions is slower than the parametric one. On the other hand, the invariance of the conditional copula given the value of the covariate suggests the possibility of parametric convergence rates. The more optimistic intuition is shown to be correct, confirming a conjecture supported by extensive Monte Carlo simulations by I. Hobaek Haff and J. Segers [Computational Statistics and Data Analysis 84:1--13, 2015] and improving upon the nonparametric rate obtained theoretically by I. Gijbels, M. Omelka and N. Veraverbeke [Scandinavian Journal of Statistics 42:1109--1126, 2015]. The novelty of the proposed approach lies in a double smoothing procedure for the estimator of the marginal conditional distribution functions. 
The copula estimator itself is asymptotically equivalent to an oracle empirical copula, as if the marginal conditional distribution functions were known.

\noindent \textbf{Keywords:} Donsker class; Empirical copula process; Local linear estimator; Pair-copula construction; Partial copula; Smoothing; Weak convergence.
\end{abstract}

\section{Introduction}

Let $(Y_1, Y_2)$ be a pair of continuous random variables with joint distribution function $H(y_1, y_2) = \Pr(Y_1 \le y_1, Y_2 \le y_2)$ and marginal distribution functions $F_j(y_j) = \Pr(Y_j \le y_j)$, for $y_j \in \reals$ and $j \in \{1,2\}$. If $H$ is continuous, then $(F_1(Y_1), F_2(Y_2))$ is a pair of uniform $(0, 1)$ random variables. Their joint distribution function, $D(u_1, u_2) = \Pr\{ F_1(Y_1) \le u_1, F_2(Y_2) \le u_2 \}$ for $u_j \in [0, 1]$, is therefore a \emph{copula}. By Sklar's celebrated theorem \citep{sklar:1959}, we have $H(y_1, y_2) = D \{ F_1(y_1), F_2(y_2) \}$. 
The copula $D$ thus captures the dependence between the random variables $Y_1$ and $Y_2$.

Suppose there is a third random variable, $X$, and suppose the joint conditional distribution of $(Y_1, Y_2)$ given $X = x$, with $x \in \reals$, is continuous. Then we can apply Sklar's theorem to the joint conditional distribution function $H(y_1,y_2|x) = \Pr(Y_1 \le y_1, Y_2 \le y_2 \mid X = x)$. Writing $F_j(y_j|x) = \Pr (Y_j \le y_j \mid X = x)$, we have $H(y_1,y_2|x) = C \{ F_1(y_1|x), F_2(y_2|x) \mid x \}$, where $C(u_1, u_2|x) = \Pr\{ F_1(Y_1|x) \le u_1, \, F_2(Y_2|x) \le u_2 \mid X = x \}$ is the copula of the conditional distribution of $(Y_1, Y_2)$ given $X = x$. This conditional copula thus captures the dependence between $Y_1$ and $Y_2$ conditionally on $X = x$. Examples include exchange rates before and after the introduction of the euro \citep{patton:2006}, diastolic versus systolic blood pressure when controlling for cholesterol \citep{lambert:2007}, and life expectancies of males versus females for different categories of countries \citep{veraverbeke+o+g:2011}.

Evidently, we can integrate out the joint and marginal conditional distributions to obtain their unconditional versions: if $X$ has density $f_X$, then $H(y_1,y_2) = \int H(y_1,y_2|x) \, f_X(x) \, \diff x$ and similarly for $F_j(y_j)$. For the copula, however, this relation does not hold: in general, $D(u_1, u_2)$ will be different from $\int C(u_1, u_2|x) \, f_X(x) \, \diff x$.

\begin{example}
Suppose that $Y_1$ and $Y_2$ are conditionally independent given $X$: for all $x$, we have $H(y_1,y_2|x) = F_1(y_1|x) \, F_2(y_2|x)$. Then the copula of the conditional distribution is the independence copula: $C(u_1, u_2|x) = u_1 u_2$. Nevertheless, unconditionally, $Y_1$ and $Y_2$ need not be independent, and the copula, $D$, of the unconditional distribution of $(Y_1, Y_2)$ can be different from the independence copula.
\end{example}

In the previous example, the copula of the conditional distribution of $(Y_1, Y_2)$ given $X = x$ does not depend on the value of $x$. This invariance property is called the \emph{simplifying assumption}. The property does not hold in general but it is satisfied, for instance, for trivariate Gaussian distributions. For further examples and counterexamples of distributions which do or do not satisfy the simplifying assumption, see for instance~\cite{hobaekhaff+a+f:2010} and \cite{stober+j+c:2013}.

\begin{example}
\label{ex:Gaussian}
Let $(Y_1, Y_2, X)$ be trivariate Gaussian with means $\mu_1, \mu_2, \mu_X$, standard deviations $\sigma_1, \sigma_2, \sigma_X > 0$, and correlations $\rho_{12}, \rho_{1X}, \rho_{2X} \in (-1, 1)$. The conditional distribution of $(Y_1, Y_2)$ given $X = x$ is bivariate Gaussian with means $\mu_j - \rho_{jX} x$ and standard deviations $\sigma_j (1 - \rho_{jX}^2)^{1/2}$ for $j \in \{1, 2\}$, while the correlation is given by the partial correlation of $(Y_1, Y_2)$ given $X$, i.e.,
$
	\rho_{12|X}
	=
	(\rho_{12} - \rho_{1X}\rho_{2X}) /
	\{ (1 - \rho_{1X}^2)(1 - \rho_{2X}^2) \}^{1/2}$.
%
As a consequence, the conditional copula of $(Y_1, Y_2)$ given $X = x$ is the so-called Gaussian copula with correlation parameter $\rho_{12|X}$, whatever the value of $x$, while the unconditional copula of $(Y_1, Y_2)$ is the Gaussian copula with correlation parameter $\rho_{12}$.
\end{example}

Given an iid sample $(X_i, Y_{i1}, Y_{i2})$, $i = 1, \ldots, n$, we seek to make nonparametric inference on the conditional dependence of $(Y_1, Y_2)$ given $X$ under the simplifying assumption that there exists a single copula, $C$, such that $C(u_1, u_2|x) = C(u_1, u_2)$ for all $x$ and all $(u_1, u_2)$. As noted in the two examples above, $C$ is usually not equal to the copula, $D$, of the unconditional distribution of $(Y_1, Y_2)$. The inference problem arises in so-called pair-copula constructions, where a multivariate copula is broken down into multiple bivariate copulas through iterative conditioning \citep{joe:1996,bedford+c:2002,aas+c+f+b:2009}. If the pair copulas are assumed to belong to parametric copula families, then likelihood-based inference yields $1/\sqrt{n}$-consistent estimators of the copula parameters \citep{hobaekhaff:2013}. Here, we concentrate instead on the nonparametric case. The mathematical analysis is difficult and our treatment is therefore limited to a single conditioning variable.

Suppose we have no further structural information on the joint distribution of $(X, Y_1, Y_2)$ besides smoothness and the simplifying assumption. Then how well can we hope to estimate the conditional copula $C$? The simplifying assumption implies that $C$ is an object that combines both local and global properties of the joint distribution of $(X, Y_1, Y_2)$: conditioning upon $X = x$ and integrating out, we find that $C$ is equal to the partial copula of $(Y_1, Y_2)$ given $X$ \citep{bergsma:2011,gijbels+o+v:2015:ejs},
\[
  C(u_1, u_2) = \Pr\{ F_1(Y_1|X) \le u_1, F_2(Y_2|X) \le u_2 \}.
\]
On the one hand, the necessity to estimate the univariate conditional distribution functions $F_{j}(y_j|x)$ suggests that the best convergence rate that can be achieved for estimating $C$ will be slower than $1/\sqrt{n}$. On the other hand, since the outer probability is an unconditional one, we may hope to achieve the parametric rate, $ 1/\sqrt{n}$.

In \cite{hobaek+s:2015}, evidence from extensive numerical experiments was provided to support the more optimistic conjecture that the parametric rate $ 1/\sqrt{n}$ can be achieved. The estimator proposed was constructed as the empirical copula \citep{deheuvels:1979,ruschendorf:1976} based on the pairs 
\begin{equation}
\label{eq:pseudoobs}
  \bigl( \hat{F}_{n,1}(Y_{i1}|X_i), \hat{F}_{n,2}(Y_{i2}|X_i) \bigr), 
  \qquad i = 1, \ldots, n, 
\end{equation}
where the $\hat{F}_{n,j}$ were nearest-neighbour estimators of the $F_j$. In \cite{omelka+g+v:15}, a high-level theorem was provided saying that if $\hat{F}_{n,j}$ are estimators of the univariate conditional margins that satisfy a number of assumptions, then the empirical copula based on the pairs \eqref{eq:pseudoobs} is consistent and asymptotically normal with convergence rate $1/\sqrt{n}$. In addition, it was proposed to estimate the univariate conditional margins $F_j$ by local linear estimators. However, attempts to prove that any of these estimators of $F_j$ also satisfy the conditions of the cited theorem have failed so far. The only positive results have been obtained under additional structural assumptions on the conditional margins, exploitation of which leads to estimators $\hat{F}_{n,j}$ with the desired properties.

The contribution of our paper is then two-fold:
\begin{enumerate}
\item
We provide an alternative to Theorem~2 in \cite{omelka+g+v:15}, imposing weaker conditions on the estimators $\hat{F}_{n,j}$ of the univariate conditional margins, and concluding that the empirical copula based on the pairs \eqref{eq:pseudoobs} is consistent and asymptotically normal with rate $ 1/\sqrt{n}$. The conclusion of our main theorem is a bit weaker than the one of the cited theorem in that we can only prove weak convergence of stochastic processes on $[\gamma, 1-\gamma]^2$, for arbitrary $0 < \gamma < 1/2$, rather than on $[0, 1]^2$.
\item
We provide nonparametric estimators $\hat{F}_{n,j}$ of the conditional margins $F_j$ that actually satisfy the conditions of our theorem. The estimators are smoothed local linear estimators. In contrast to \cite{omelka+g+v:15}, we also smooth in the $y_j$-direction, ensuring that the trajectories $(y_j, x) \mapsto \hat{F}_{n,j}(y_j|x)$ belong to a Donsker class with high probability. 
\end{enumerate}
To prove that the smoothed local linear estimators $\hat{F}_{n,j}$ satisfy the requirements of our main theorem, we prove a number of asymptotic results on those estimators that may be interesting in their own right.

Although our estimator of $F_j$ and thus of $C$ is different from the ones of \cite{hobaek+s:2015} or \cite{omelka+g+v:15}, we do not claim it to be superior. Our objective is rather to provide the first construction of a nonparametric estimator of $C$ that can be proven to achieve the $ 1/\sqrt{n}$ convergence rate. This is why, in the numerical experiments, we limit ourselves to illustrate the asymptotic theory but do not provide comparisons between estimators. For the same reason, we do not address other important questions of practical interest such as how to choose the bandwidths or how to test the hypothesis that the simplifying assumption holds. For the latter question, we refer to \cite{acar:genest:neslehova:2012} and \cite{derumigny:2016}.

The paper is structured as follows. In Section~\ref{sec:def+high_level}, we state the main theorem giving conditions on the estimators $\hat{F}_{n,j}$ for the empirical copula based on the pairs \eqref{eq:pseudoobs} to be consistent and asymptotically normal with convergence rate $ 1/\sqrt{n}$. In Section~\ref{sec:main:loclin}, we then show that the smoothed local linear estimator satisfies the requirements. The theory is illustrated by numerical experiments. Auxiliary results of independent interest on the smoothed local linear estimator are stated in Section~\ref{sec:loc_lin_results}. The proofs build upon empirical process theory and their details are spelled out in the appendices.

\section{Empirical conditional copula}
\label{sec:def+high_level} 

 The goal of this section is to present the mathematical framework of the paper (Section~\ref{subsec:def}) and to provide general conditions on the estimated margins that ensure the weak convergence of the estimated conditional copula (Section~\ref{subsec:high_level}).

\subsection{Set-up and definitions}
\label{subsec:def}

Let $f_{X, \bY}$ be the density function of the random triple $(X, \bY) = (X, Y_1, Y_2)$. Let $f_X$ and $S_X = \{ x \in \R : f_X(x) > 0 \}$ denote the density and the support of $X$, respectively. 
The conditional cumulative distribution function of $\bY$ given $X=x$ is given by
\begin{align*}
  H(\by \mid x) 
  = \Pr( Y_1 \le y_1, Y_2 \le y_2 \mid X = x ) 
  = 
  \int_{-\infty}^{y_1}\int_{-\infty}^{y_2} 
    \frac{f_{X,\bY}(x,z_1,z_2)}{f_X(x)} \, 
  \diff z_2 \, \diff z_1,
\end{align*}
for $\by=(y_1,y_1)\in \R^2$ and $x\in S_X$. Since $H(\point|x)$ is a continuous bivariate cumulative distribution function, its copula is given by the function
\begin{equation*}
  C( \bu \mid x) =
  \Pr\left\{ F_{1}(Y_1| X)\,\leq\, u_1 , \; F_{2}(Y_2| X)\,\leq\, u_2\mid X=x  \right\},
\end{equation*}
for $\bu = (u_1, u_2) \in [0, 1]^2$ and $x \in S_X$, where $F_1(\point|x)$ and $F_2(\point|x)$ are the margins of $H(\point|x)$. 


We make the simplifying assumption that the copula of $H(\point|x)$ does not depend on $x$, i.e., $C(\point|x) \equiv C(\point)$. This assumption is equivalent to the one that
\begin{align}
\label{assimplifying}
\text{$ \bigl( F_{1}(Y_1|X), F_{2}(Y_2|X) \bigr)$ is independent of $X$.}
\end{align} 
Under \eqref{assimplifying}, the copula of $H(\point|x)$ is given by $C$ for every $x \in S_X$.  
It is worth mentioning that, for any $j\in \{1,2\}$, and even without the simplifying assumption,
\begin{align}\label{factuniformdist}
 \text{$F_{j}(Y_j|X)$ is uniformly distributed on $[0,1]$ and independent of $X$}.
\end{align} 

Let $(X_i, Y_{i1}, Y_{i2})$, for $i \in \{1, \ldots, n\}$, be independent and identically distributed random vectors, with common distribution equal to the one of $(X, Y_1, Y_2)$.
Our aim is to estimate the conditional copula $C$ without any further structural or parametric assumptions on $C$ or on $F_j(\point|x)$. A reasonable procedure is to estimate the conditional margins in some way, producing random functions $\hat{F}_{n,j}(\point|x)$, and then proceed with the pseudo-observations $\hat{F}_{n,j}(Y_{ij}|X_i)$ from $C$. Exploiting the knowledge that $C$ is a copula, we estimate it by  the empirical copula, $\hat{C}_n$, of those pseudo-observations. Formally, let $\hat{G}_{n,j}$, for $j \in \{1,2\}$, be the empirical distribution function of the pseudo-observations $\hat{F}_{n,j}(Y_{ij}|X_i)$, $i \in \{1,\ldots,n\}$, i.e., 
\begin{align*}
\hat G_{n,j}(u_j) = \frac{1}{n} \sum_{i=1}^n \ind _{\{ \hat F_{n,j}(Y_{ij}|X_i)\, \leq\, u_j\}},\qquad u_j\in[0,1].
\end{align*} 
The generalized inverse of a univariate distribution function $F$ is defined as
\begin{equation}
\label{eq:genInv}
  F^-(u) = \inf \{ y \in \reals : F(y) \ge u \}, \qquad u \in [0, 1].
\end{equation}
Let $\hat{G}_{n,j}^-$ be the generalized inverse of $\hat{G}_{n,j}$. The empirical conditional copula, $\hat{C}_n$, is defined by
\begin{align}\label{def:conditionalempiricalcopula}
\hat C_n(\bu )= \frac{1}{n} \sum_{i=1}^n \ind _{\{ \hat F_{n,1}(Y_{i1}|X_i)\,  \leq \, \hat G^{-}_{n,1}(u_1) \}}\ind _{\{ \hat F_{n,2}(Y_{i2}|X_i)\, \leq \, \hat G^{-}_{n,2}(u_2)\}},\qquad \bu \in [0, 1]^2,
\end{align}
the empirical copula  of the pseudo-observations $(\hat F_{n,1}(Y_{i1}|X_i),\hat F_{n,2}(Y_{i2}|X_i))$, for $i \in \{1, \ldots, n\}$.

We introduce an oracle copula estimator, defined as the empirical copula based on the unobservable random pairs $(F_{1}(Y_{i1}|X_i),F_{1}(Y_{i2}|X_i))$, $i \in \{1,\ldots,n\}$. Let $\hat{G}_{n,j}^{(or)}$ be the empirical distribution function of the uniform random variables $F_j(Y_{ij}|X_i)$, $i \in \{1,\ldots,n\}$, i.e.,
\[
  \hat G^{(or)}_{n,j}(u_j) 
  = 
  \frac{1}{n} \sum_{i=1}^n 
  \ind_{\{  F_{j}(Y_{ij}|X_i)\, \leq \,  u_j\}}, \qquad u_j\in [0,1].
\]
Let $\hat{G}^{(or)-}_{n,j}$ be its generalized inverse, as in \eqref{eq:genInv}. The oracle empirical copula is defined as
\begin{align*}
\hat C_n^{(or)}(\bu) = \frac{1}{n} \sum_{i=1}^n \ind _{\{  {F}_{1}(Y_{i1}|X_i)\, \leq\,  \hat{{G}}^{(or)-}_{n,1} (u_1)\}}\ind _{\{  {F}_{2}(Y_{i2}|X_i)\, \leq\,  \hat{{G}}^{(or)-}_{n,2} (u_2)\}} , \qquad \bu \in [0,1]^2.
\end{align*}
The oracle empircal copula is not computable in practice as it requires the knowledge of the marginal conditional distributions $F_1$ and $F_2$.

We rely on the following H\"older regularity class. Let  $d \in \mathbb{N} \setminus \{ 0 \}$, $0<\delta\leq 1$, $k\in \mathbb N$, and $M>0$ be scalars and let $S\subset \R^d$ be non-empty, open and convex.  Let $\mathcal C_{k+\delta,M}(S)$ be the space of functions $S\rightarrow \R$ that are $k$ times differentiable and whose derivatives (including the zero-th derivative, that is, the function itself) are uniformly bounded by $M$ and such that every mixed partial derivative of order $l \leq k$, say $f^{(l)}$, satisfies the H\"{o}lder condition 
\begin{align}
\label{eq:Holder}
 \sup_{z\neq \tilde z}\,  \frac{\abs{ f^{(l)}(z)-f^{(l)}(\tilde z) }}{\abs{z-\tilde z}^{\delta}}   \leq M ,
\end{align}
where $\lvert\point\rvert$ in the denominator denotes the Euclidean norm. In particular, $\mathcal C_{1,M}(\R)$ is the space of Lipschitz  functions $\R\rightarrow \R $ bounded by $M$ and with Lipschitz constant bounded by $M$.

\subsection{Asymptotic normality}
\label{subsec:high_level}

We now give a set of sufficient conditions to ensure that $\hat C_n $ and $\hat C_n^{(or)}$ are asymptotically equivalent, i.e., their difference is $o_{\mathbb P}(n^{-1/2})$. As a consequence, $\hat{C}_n$ is consistent and asymptotically normal with convergence rate $O_{\mathbb P}(n^{-1/2})$. In the same spirit as Theorem~2 in \cite{omelka+g+v:15}, some of the assumptions are ``ground-level conditions'' and concern directly the distribution $P$, while others are ``high-level conditions'' and deal with the estimators $\hat F_{n,j}$ of the conditional margins. Given a choice for $\hat{F}_{n,j}$, the high-level conditions need to be verified, as we will do in Section~\ref{sec:main:loclin} for a specific proposal. Ground-level assumptions are denoted with the letter G and high-level conditions are denoted with the letter H. 

\begin{enumerate}[(\text{G}1)]
\item 
 \label{cond:smoothnessdensity1}
The law $P$ admits a density $f_{X,\bY}$ on $S_X\times \R^2 $ such that $S_X$ is a nonempty, bounded, open interval. For some $M>0$  and $\delta>0$, the functions $F_1(\point|\point)$ and $F_2(\point|\point)$ belong to $ \mathcal C_{3+\delta, M}(\mathbb R\times S_X)$ and the function $f_X$ belongs to $ \mathcal C_{2, M}( S_X)$. There exists $b>0$ such that $f_X(x) \ge b$ for every $x\in S_X$. For any $j\in \{1,2\}$ and any $\gamma\in (0,1/2)$, there exists $b_\gamma>0$ such that, for every  $y_j \in  [ F_{j}^- (\gamma |x) ,F_{j}^- (1-\gamma |x) ]$ and every $x\in S_X$, we have
$
f _j(y_j|x) \geq b_\gamma
$.

\item
\label{cond:copula_smoothness}
Let $\dot{C}_j$ and $\ddot{C}_{jk}$ denote the first and second-order partial derivatives of $C$, where $j,k \in \{1, 2\}$.
The copula $C$ is twice continuously differentiable on the open unit square, $(0, 1)^2$.  There exists $\kappa > 0$ such that, for all $\bu =(u_1,u_2)\in (0, 1)^2$ and all $j, k \in \{1, 2\}$,
\[
  \abs{ \ddot{C}_{jk}( \bu ) } \le \kappa \, \{ u_j(1-u_j) \, u_k(1-u_k) \}^{-1/2}.
\]
\end{enumerate}
\begin{enumerate}[(\text{H}1)]
\item
\label{cond:high_level_inverse_inequality}
For any $j\in \{1,2\}$ and any $\gamma\in (0,1/2)$,  with probability going to $1$, we have that for every $x\in S_X$, the function $y_j\mapsto \hat F_{n,j}(y_j|x)$ is continuous on $\mathbb R$ and strictly increasing on $ [ F_{j}^- (\gamma |x) ,F_{j}^- (1-\gamma |x) ]$.
\item For any $j\in\{1,2\}$, we have
\label{cond:high_level_consistency}
\begin{align*}
 \sup_{x\in S_X ,\, y_j\in \mathbb R} \abs{ \hat F_{n,j}(y_j | x) - F_j(y_j|x) } =o_{\mathbb P} (n^{-1/4})  .
\end{align*}
\item\label{cond:high_level_donsker}
For any $j\in\{1,2\}$ and any $\gamma\in(0,1/2)$, there exist positive numbers $(\delta_1,M_1)$ such that, with probability going to $1$,
\begin{align*}
&\{x \mapsto  \hat F_{n,j}^-(u_j|x)  \, :\, u_j\in [\gamma, 1-\gamma]  \}\subset \mathcal C_{1+\delta_1,M_1}(S_{X}).
\end{align*}
\end{enumerate}
The support $S_X$ is assumed to be open so that the derivatives of functions on $S_X$ are defined as usual. Condition~(G\ref{cond:copula_smoothness}) also appears as equation~(9) in \cite{omelka+g+v:09}, where it is verified for some popular copula models (Gaussian, Gumbel, Clayton, Student t).

Let $\ell^\infty(T)$ denote the space of bounded real functions on the set $T$, the space being equipped with the supremum distance, and let ``$\leadsto$'' denote weak convergence in this space \citep{wellner1996}.  Let $\mathbb P$ denote the probability measure on the underlying probability space associated to the whole sequence $(X_i,\bY_i)_{i = 1,2,\ldots}$.

\begin{theorem}
\label{theorem:weakcv_highlevel}
Assume that (G\ref{cond:smoothnessdensity1}), (G\ref{cond:copula_smoothness}), (H\ref{cond:high_level_inverse_inequality}), (H\ref{cond:high_level_consistency}) and (H\ref{cond:high_level_donsker}) hold. If the simplifying assumption~\eqref{assimplifying} holds, then for any $\gamma\in (0,1/2)$, we have
\begin{equation}
\label{eq:oP1}
  \sup_{\bu\in [\gamma,1-\gamma]^2 } 
  \left\lvert 
    n^{1/2} \bigl\{\hat C_n(\bu)  -\hat C_n^{(or)}(\bu) \bigr\}
  \right\rvert 
  = o_{\mathbb P}(1), \qquad n \to \infty.
\end{equation}
Moreover, $n^{1/2} \bigl\{\hat C_n(\bu)  - C (\bu)  \bigr\}  \leadsto \mathbb C$  in $\ell^\infty([\gamma,1-\gamma]^2)$,  the limiting process  $\mathbb C$  having the same law as the random process
\begin{align*}
  \mathbb B(\bu ) 
  - \dot{C}_1(\bu) \, \mathbb B(u_1,1) 
  - \dot{C}_2(\bu)\, \mathbb B(1,u_2),
  \qquad \bu \in [\gamma, 1-\gamma]^2,
\end{align*}
where  $\mathbb B$ is a $C$-Brownian bridge, i.e., a centered Gaussian process with continuous trajectories and with covarance function given by
\begin{equation}
\label{eq:cov:B}
  \operatorname{cov} \bigl\{ \mathbb B(\bu), \, \mathbb B(\bv) \bigr\} 
  = 
  C(u_1\wedge v_1,u_2\wedge v_2)
  -
  C(\bu)\, C(\bv), 
  \qquad (\bu, \bv) \in \bigl([\gamma,1-\gamma]^2\bigr)^2.
\end{equation}
\end{theorem}

The proof of Theorem~\ref{theorem:weakcv_highlevel} is given in Appendix~\ref{app:proof:weakcv_highlevel}.
The theorem bears resemblance with Theorem~2 in \cite{omelka+g+v:15}. Our approach is more specific 
because we consider a smoothness approach through the spaces $\mathcal C_{1+\delta,M}(S_X)$ to obtain their Donsker property. Exploiting this context, we formulate in (H\ref{cond:high_level_consistency}) a condition on the rate of convergence of the estimator $\hat F_{n,j}$ that does not involve the inverse $\hat F_{n,j}^-$, as expressed in their condition (Yn).

\section{Smoothed local linear estimators of the conditional margins}
\label{sec:main:loclin}

The objective in this section is to provide estimators, $\hat{F}_{n,j}$, $j \in \{1,2\}$, of the conditional marginal distribution functions $F_j(y_j|x) = \Pr(Y_j \le y_j \mid X = x)$  that satisfy conditions (H\ref{cond:high_level_inverse_inequality}), (H\ref{cond:high_level_consistency}) and (H\ref{cond:high_level_donsker}) of Theorem~\ref{theorem:weakcv_highlevel} (Section~\ref{subsec:loc_lin_def}). As a consequence, the empirical conditional copula based on these estimators is a nonparametric estimator of the conditional copula in the simplifying assumption which, by Theorem~\ref{theorem:weakcv_highlevel}, is consistent and asymptotically normal with convergence rate $1/\sqrt{n}$ (Section~\ref{subsec:weakconvempcondcop}). Numerical experiments confirm that the asymptotic theory provides a good approximation to the sampling distribution, at least for large samples (Section~\ref{subsec:numerical}).

\subsection{Definition of the smoothed local linear estimator}
\label{subsec:loc_lin_def}

Let $K:\R\rightarrow [0, \infty) $ and $L:\R\rightarrow [0, \infty)$ be two kernel functions, i.e., nonnegative, symmetric functions  integrating to unity. Let $(h_{n,1})_{n\geq 1}$ and $(h_{n,2})_{n\geq 1}$ two bandwidth sequences that tend to $0$ as $n \to \infty$. For $(y, Y) \in \reals^2$ and $h > 0$, put
\begin{align}
  L_h(y) &= h^{-1} \, L(h^{-1} y), &
\label{def:varphi}
  \varphi_h(y, Y ) &= \int_{-\infty}^{y} L_{h}(t-Y) \, \diff t.
\end{align}
For $j\in\{1,2\}$, we introduce the smoothed local linear estimator of $F_j(y_j | x)$ defined by 
\begin{align}\label{def:ll_estimator}
\hat F_{n,j}(y_j|x) = \hat a_{n,j} ,
\end{align}
where $\hat a_{n,j}$ is the first component of the random pair
\begin{align}\label{eq:app:optimisationLL}
  (\hat a_{n,j}, \hat b_{n,j}) 
  = 
  \argmin_{(a,b)\in\R^2}\, 
  \sum_{i=1}^n 
    \bigl\{
      \varphi_{h_{n,2}}(y_j,Y_{ij} )  - a- b(X_i-x) 
    \bigr\}^2 \, 
    K \left(\frac{x-X_i}{h_{n,1}} \right) ,
\end{align}
where $ \varphi_h$  in \eqref{def:varphi} serves to smooth the indicator function $y \mapsto \1_{\{Y \le y\}}$. The kernels $K$ and $L$ do not have the same role: $L$ is concerned with ``smoothing'' over  $Y_1$ and $Y_2$ whereas $K$ ``localises'' the variable $X$ at $x\in S_X$. For this reason, we purposefully use two different bandwidth sequences $(h_{n,1})_{n\geq 1}$ and $(h_{n,2})_{n\geq 1}$.  We shall see that the conditions on the bandwidth $h_{n,2}$ for the $y$-directions are weaker than the ones for the bandwidth $h_{n,1}$ for the $x$-direction. The assumptions related to the two kernels and bandwidth sequences are stated in (G\ref{cond:kernel}) and (G\ref{cond:bandwidth}) below. 

In the classical regression context, local linear estimators have been introduced in \cite{stone:1977} and are further studied for instance in \cite{fan1996}. In \cite{omelka+g+v:15}, local linear estimators for $F_j(y|x)$ are considered too, but without smoothing in the $y$-variable, so that condition~(H\ref{cond:high_level_donsker}) does not hold; see Section~\ref{sec:furthercomments} below.

\subsection{Weak convergence of the empirical conditional copula} 
\label{subsec:weakconvempcondcop}

We derive the limit distribution of the empirical conditional copula $\hat{C}_n$ in \eqref{def:conditionalempiricalcopula} when the marginal conditional distribution functions are estimated via the smoothed local linear estimators $\hat{F}_{n,j}(y_j|x)$ in \eqref{def:ll_estimator}. We need the following additional conditions. 

\begin{enumerate}[(\text{G}1)]\setcounter{enumi}{2}
\item
\label{cond:kernel}
The kernels $K$ and $L$ are bounded, nonnegative, symmetric functions on $\reals$, supported on  $(-1, 1)$, and such that $\int L(u) \, \diff u = \int K(u) \, \diff u = 1$. The function $L$ is continuously differentiable on $\R$ and its derivative is a bounded real function of bounded variation. The function $K$ is twice continuously differentiable on $\R$ and its second-order derivative is a bounded real function of bounded variation.
\item
\label{cond:bandwidth2}
There exists $\alpha > 0$ such that the bandwidth sequences $h_{n,1} > 0$ and $h_{n,2} > 0$ satisfy, as $n \to \infty$,
\begin{align*}
&nh_{n,1}^8\rightarrow 0,\qquad nh_{n,2}^8\rightarrow 0,\qquad   h_{n,1}^{-1-\alpha/2}h_{n,2}^2 \to 0, \\
&\frac{nh_{n,1}^{3+\alpha}}{\abs{\log h_{n,1}}}   \rightarrow \infty,\qquad \frac{nh_{n,1}^{1+\alpha}h_{n,2}}{\abs{\log h_{n,1}h_{n,2}}}\to \infty .
\end{align*}  
\end{enumerate}
Condition (G\ref{cond:bandwidth2}) implies in particular that the bandwidth sequences are small enough to ensure that the bias associated to the estimation of $F_1$ and $F_2$ does not affect the asymptotic distribution of the empirical conditional copula. An interesting situation occurs when the bandwidths satisfy $h_{n,2} / h_{n,1}^{1/2+\alpha/4}\rightarrow 0$ and $h_{n,1}^2 / h_{n,2}\rightarrow 0$. Then the above condition becomes
\begin{align*}
&nh_{n,1}^8\rightarrow 0,\quad nh_{n,2}^8\rightarrow 0,\quad  \frac{nh_{n,1}^{3+\alpha}}{\abs{\log h_{n,1}}} \rightarrow \infty.
\end{align*}  
Hence the conditions on the bandwidth $h_{1,n}$ are more restrictive than the ones on $h_{2,n}$. Typically, $h_{n,2}$ might be chosen smaller than $h_{n,1}$ as it does not need to satisfy ${nh_{n,1}^{3+\alpha}}/{\abs{\log h_{n,1}}} \rightarrow \infty$. This might result in a smaller bias. By way of comparison, in the location-scale model, no smoothing in the $y$-direction is required, but \cite{omelka+g+v:15} still require the stronger condition that $nh_{n,1}^{5}\rightarrow 0$ and $nh_{n,1}^{3+\alpha} / \log n \to \infty$.

\begin{theorem}
\label{theorem:weakcv}
Assume that (G\ref{cond:smoothnessdensity1}), (G\ref{cond:copula_smoothness}), (G\ref{cond:kernel}) and (G\ref{cond:bandwidth2}) hold. Then (H\ref{cond:high_level_inverse_inequality}), (H\ref{cond:high_level_consistency}) and (H\ref{cond:high_level_donsker}) are valid. If the simplifying assumption~\eqref{assimplifying} also holds, then for any $\gamma\in (0,1/2)$, equation~\eqref{eq:oP1} is satisfied and $n^{1/2} (\hat C_n  - C)  \leadsto \mathbb C$  in $\ell^\infty([\gamma,1-\gamma]^2)$, where the limiting process  $\mathbb C$ is defined in the statement of Theorem \ref{theorem:weakcv_highlevel}.
\end{theorem}

The proof of Theorem~\ref{theorem:weakcv} is given in Appendix~\ref{app:proof_th_weak_cv} and relies on results on the smoothed local linear estimator presented in Section~\ref{sec:loc_lin_results}.

Distinguishing between the bandwidth sequence $h_{n,1}$ for the $x$-direction on the one hand and the bandwidth sequence $h_{n,2}$ for the $y_1$ and $y_2$-directions on the other hand allows for a weaker assumption than if both sequences would have been required to be the same. In practice, one could even consider, for each $j \in \{1,2\}$, the smoothed local linear estimator $\hat{F}_{n,j}$ based on a bandwidth sequence $h_{n,1}^{(j)}$ for the $x$-direction and a bandwidth sequence $h_{n,2}^{(j)}$ for the $y_j$-direction, yielding four bandwidth sequences in total. However, this would not really lead to weaker assumptions in Theorem~\ref{theorem:weakcv}, since (G\ref{cond:bandwidth2}) would then be required to hold for each pair $(h_{n,1}^{(j)}, h_{n,2}^{(j)})$. The same remark also applies to the kernels $K$ and $L$, which could be chosen differently for each margin $j \in \{1,2\}$. 
The required modification of the formulation of Theorem~\ref{theorem:weakcv} is obvious.

\subsection{Numerical illustrations}
\label{subsec:numerical}

The assertion in \eqref{eq:oP1} that the empirical conditional copula $\hat{C}_n$ is only $o_{\mathbb P}(n^{-1/2})$ away from the oracle empirical copula $\hat{C}_n^{(or)}$ is perhaps surprising, since the estimators of the conditional margins converge at a rate \emph{slower}, rather than faster, than $O_{\mathbb P}(n^{-1/2})$. To support the claim, we performed a number of numerical experiments based on independent random samples from the copula of the trivariate Gaussian distribution (Example~\ref{ex:Gaussian}) with correlations given by $\rho_{1X} = 0.4$, $\rho_{2X} = -0.2$, and $\rho_{12} = 0.3689989$. The conditional copula of $(Y_1, Y_2)$ given $X$ is then equal to the bivariate Gaussian copula with correlation parameter $\rho_{12|X} = 0.5$. Estimation target was the value of the copula at $\bu = (0.5, 0.7)$.

To estimate the conditional margins, we used the smoothed local linear estimator \eqref{def:ll_estimator}--\eqref{eq:app:optimisationLL} based on the triweight kernel $K(x) = (35/32) (1 - x^2)^3 \1_{[-1, 1]}(x)$ and the biweight kernel $L(y) = (15/16) (1 - y^2)^2 \1_{[-1, 1]}(y)$, in accordance to assumption~(G\ref{cond:kernel}). The bandwidths where chosen as $h_{n,1} = h_{n,2} = 0.5 n^{-1/5}$, with $n$ the sample size. The experiments were performed within the statistical software environment \textsf{R} \citep{Rlanguage}. The algorithms for the smoothed local linear estimator were implemented in \textsf{C} for greater speed.

Figure~\ref{fig:oracle} illustrates the proximity between the estimator $\hat{C}_n(\bu)$ and the oracle $\hat{C}_n^{(or)}(\bu)$. The left-hand and middle panels show scatterplots of $1\,000$ independent realizations of the pairs $(n^{1/2}\{ \hat{C}_n(\bu) - C(\bu) \}, n^{1/2}\{ \hat{C}_n^{(or)}(\bu) - C(\bu) \})$ based on samples of size $n = 500$ and $n = 2\,000$. As $n$ increases, the points are concentrated more strongly along the diagonal. The linear correlations between the estimator and the oracle for sample sizes between $n = 100$ and $n = 10\,000$ are plotted in the right-hand panel, each point being based on $5\,000$ samples.

The proximity of the sampling distribution of the estimation error $n^{1/2} \{ \hat{C}_n(\bu) - C(\bu) \}$ and the limiting normal distribution is illustrated in Figure~\ref{fig:norm}. According to Theorems~\ref{theorem:weakcv_highlevel} and~\ref{theorem:weakcv}, the limiting distribution of the empirical conditional copula process is $\mathcal{N}(0, \sigma^2(\bu))$. The asymptotic variance is $\sigma^2(\bu) = \operatorname{var} \{ \mathbb{C}(\bu) \} = \operatorname{var} \{ \mathbb B(\bu ) - \dot{C}_1(\bu) \, \mathbb B(u_1,1) - \dot{C}_2(\bu)\, \mathbb B(1,u_2) \}$ and can be computed from \eqref{eq:cov:B} and the knowledge that $C$ is the bivariate Gaussian copula with correlation parameter $\rho_{12|X}$. As the sample size increases, the sampling distribution moves closer to the limiting normal distribution as confirmed in the QQ-plot and density plot in Figure~\ref{fig:norm}, both based on $1\,000$ samples of size $n = 500$. The realized standard deviations of the empirical conditional copula process are given in the right-hand panel, again for sample sizes between $n = 100$ and $n = 10\,000$, each point being based on $5\,000$ samples of size $n$. The limit value $\sigma(\bu) = 0.2080$ is represented as a horizontal line. 



\begin{figure}
\begin{center}
\begin{tabular}{@{}ccc}
\includegraphics[width=0.30\textwidth]{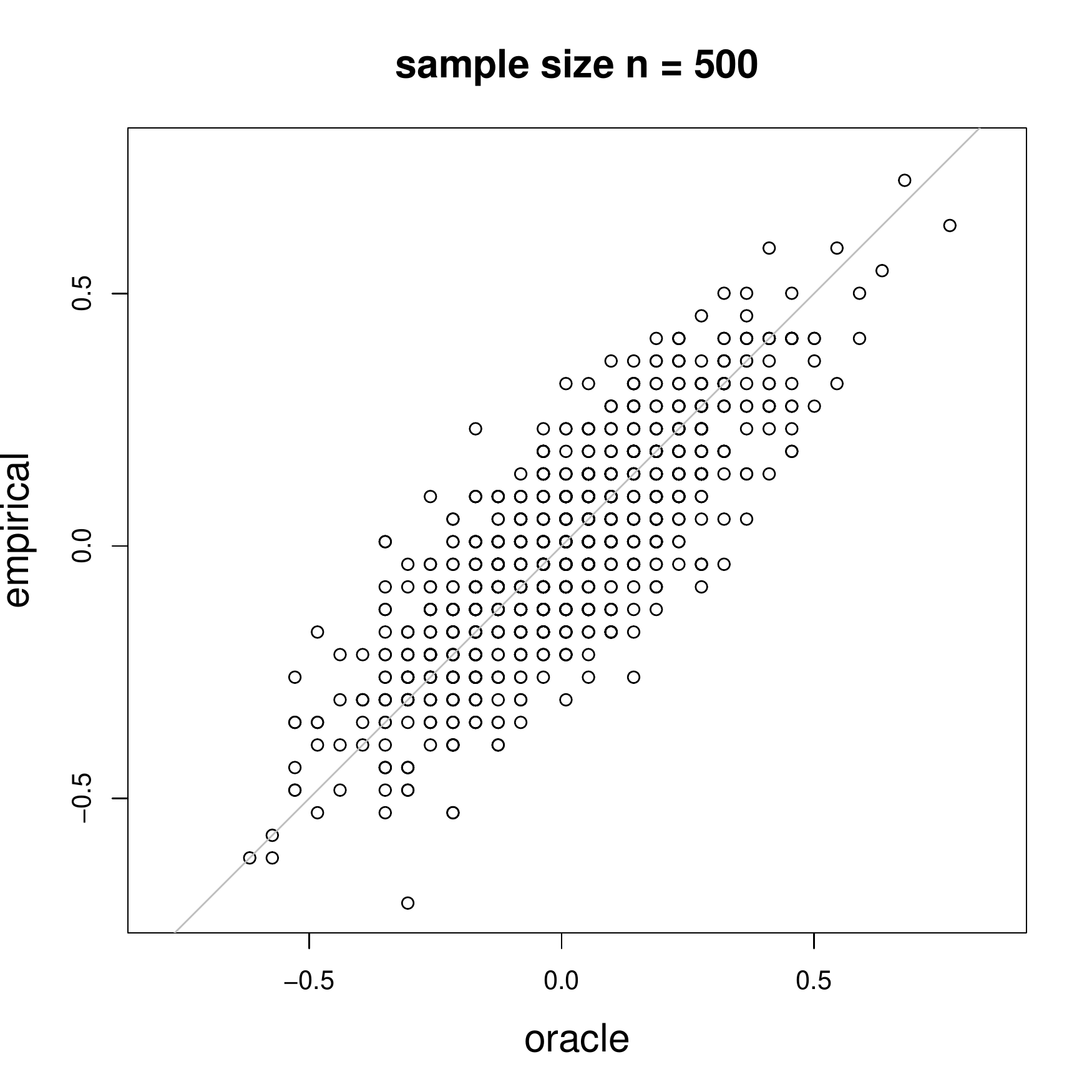}&
\includegraphics[width=0.30\textwidth]{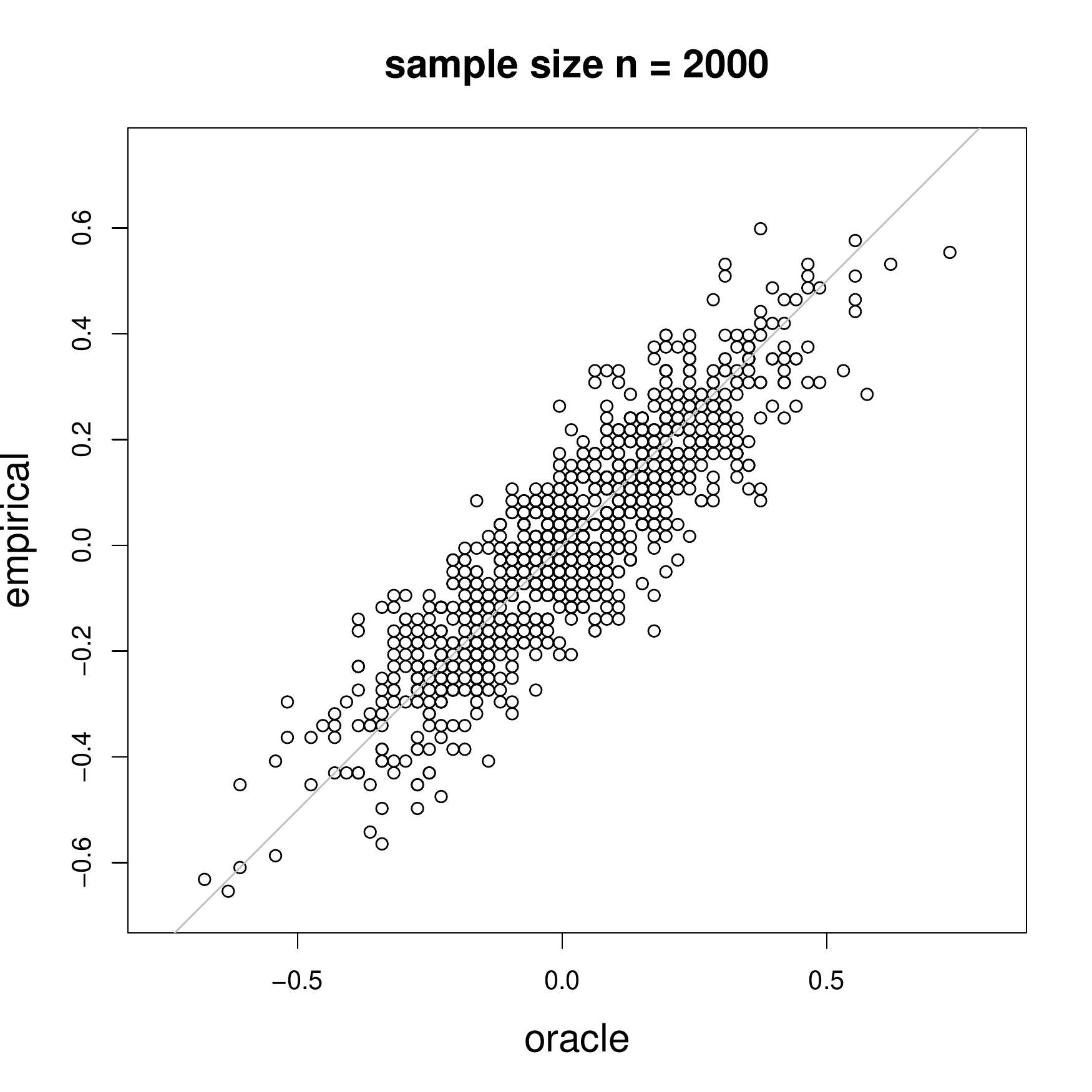}&
\includegraphics[width=0.30\textwidth]{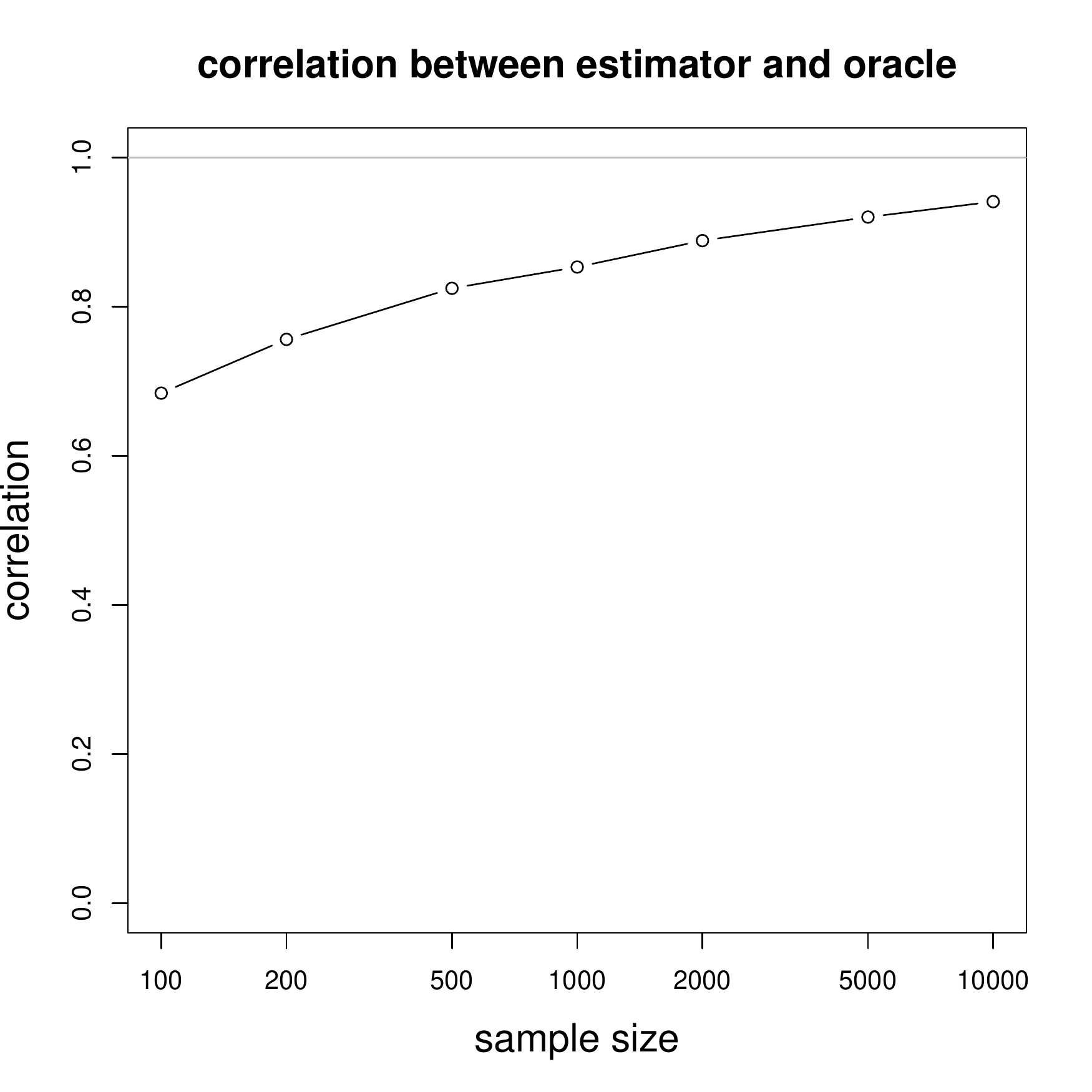}
\end{tabular}
\end{center}
\caption{\label{fig:oracle} Scatterplots of $1\,000$ independent realizations of the pairs $(n^{1/2}\{ \hat{C}_n(\bu) - C(\bu) \}, n^{1/2}\{ \hat{C}_n^{(or)}(\bu) - C(\bu) \})$ based on the trivariate Gaussian copula for sample sizes $n = 500$ (left) and $n = 2\,000$ (middle). Right: correlation between the empirical conditional copula and the oracle empirical copula as a function of the sample size $n$ between $100$ and $10\,000$, each point being based on $5\,000$ samples of size $n$.}
\end{figure}

\begin{figure}
\begin{center}
\begin{tabular}{@{}ccc}
\includegraphics[width=0.30\textwidth]{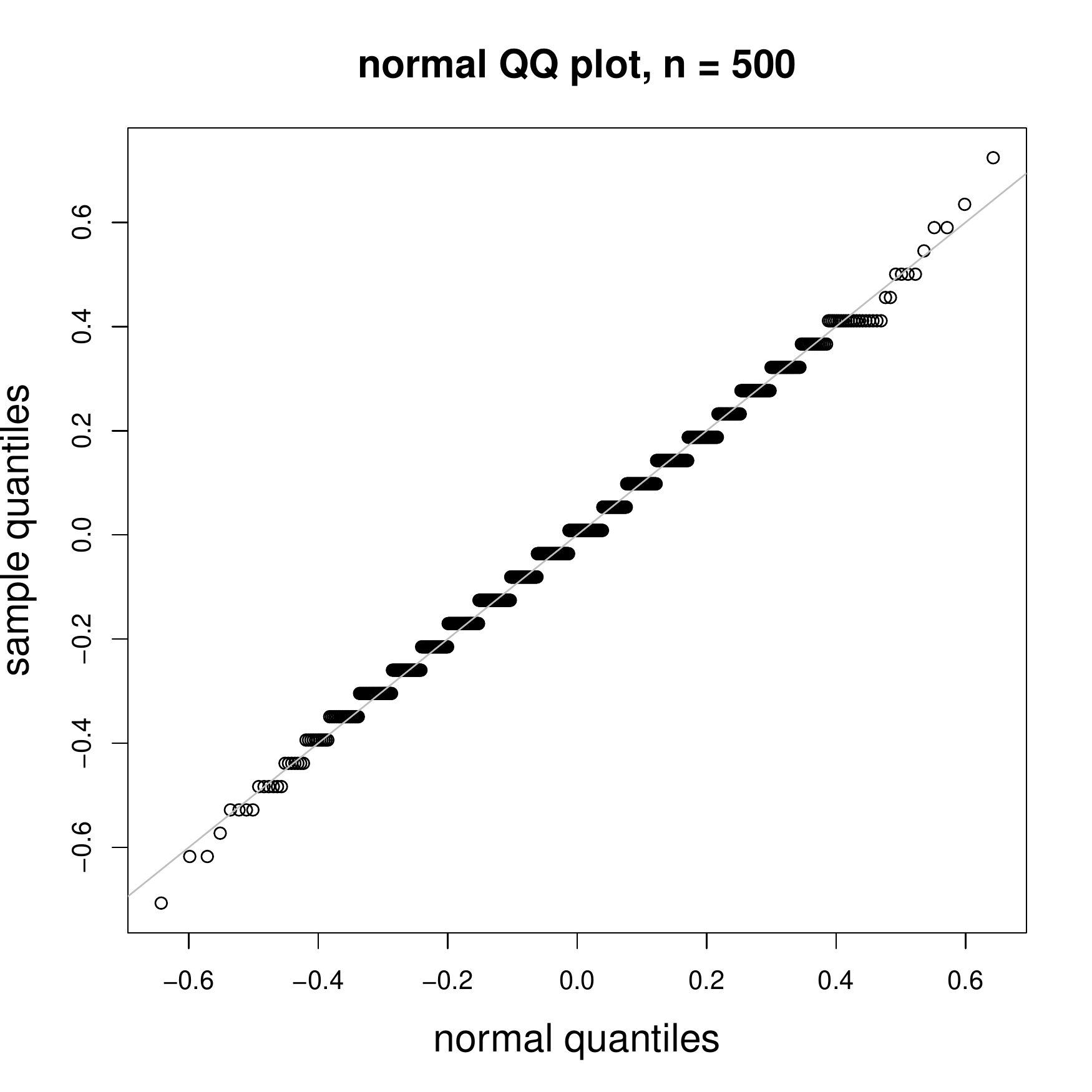}&
\includegraphics[width=0.30\textwidth]{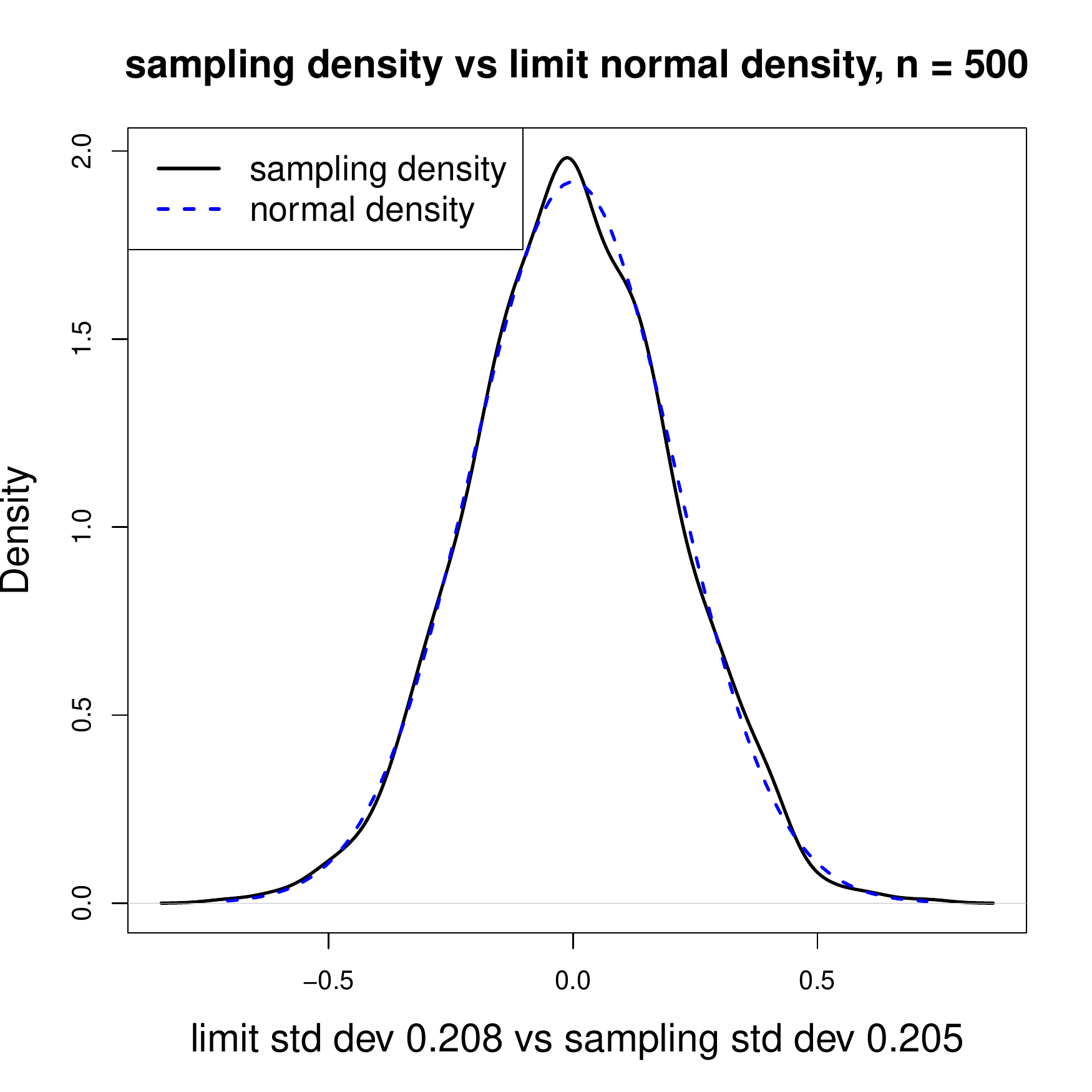}&
\includegraphics[width=0.30\textwidth]{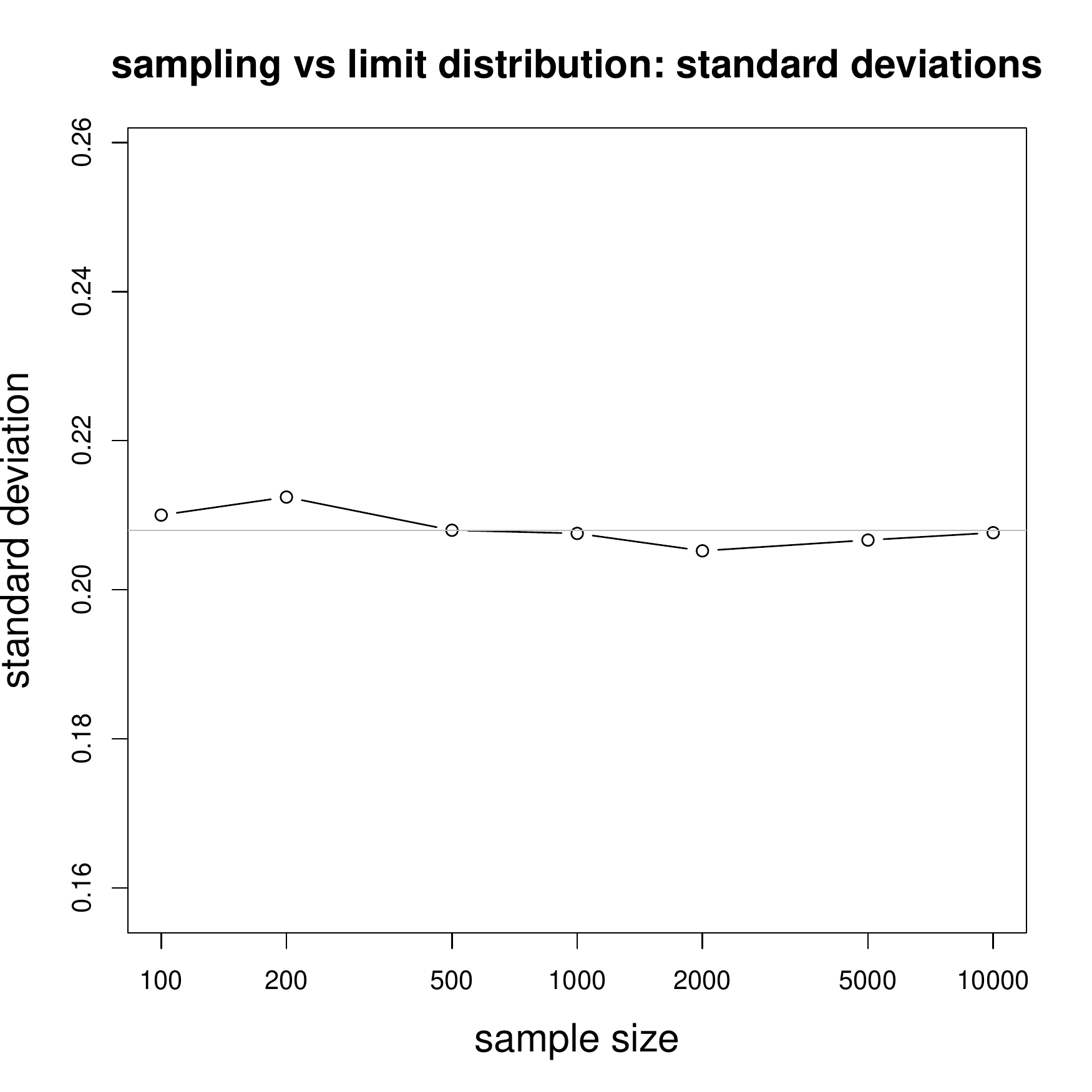}
\end{tabular}
\end{center}
\caption{\label{fig:norm} QQ-plot (left) and density plot (middle) based on $1\,000$ independent realizations of $n^{1/2}\{ \hat{C}_n(\bu) - C(\bu) \}$ for the trivariate Gaussian copula for sample size $n = 500$. Right: standard deviation of $n^{1/2}\{ \hat{C}_n(\bu) - C(\bu) \}$ as a function of $n$ between $100$ and $10\,000$ versus the limit value $\sigma(\bu) = 0.2080$, each point being based on $5\,000$ samples of size $n$.}
\end{figure}

\subsection{Further comments}
\label{sec:furthercomments}


\paragraph{Smoothing in two directions rather than one.}
In \cite{omelka+g+v:15}, local linear estimators that involve smoothing in $x$ but not in $y_j$ are employed to estimate the conditional margins ${F}_{j}(y_j|x)$. 
Interestingly, the approach taken in our paper does not extend to those estimators. Indeed, smoothing in the $y$-direction is crucial to obtain that the random functions $x \mapsto \hat{F}_{n,j}^-(u|x)$, when $u$ ranges over $[\gamma,1-\gamma]$, are included in a H\"{o}lder regularity class (Proposition~\ref{prop:regularity} and its proof). 
Our approach is based on the result that the class of functions formed by the indicators $(y,x)\mapsto \ind _{\{y\,\leq\, q(x)\}} $, where $q: S_X \rightarrow \R$ belongs to $\mathcal C _{1+\delta,M}(S_X) $ with $\delta>0$, $M>0$ and $S_X$ convex, is Donsker  \citep{dudley1999,wellner1996}.
When the estimators of the margins are not smooth with respect to $y_j$, 
the difficulty consists in finding appropriate restrictions on $q$ so that the class $(y,x)\mapsto \ind _{\{y\,\leq\, q(x)\}} $ is Donsker and still contains the functions corresponding to $q(x) = \hat F^-_{n,j}(u|x)$, $u\in[\gamma,1-\gamma]$. Negative results are available and they shed some light on the difficulties that arise. For instance, when $q$ is constrained to be nonincreasing, the class of indicators is equivalent to the class of \textit{lower layers} in $\mathbb R^2$ studied by \cite{dudley1999}, section~12.4, that fails to be Donsker. When $q$ is restricted to $\mathcal C _{1,M}(S_X)$ rather than $\mathcal{C}_{1+\delta,M}(S_X)$, the corresponding class is not Donsker either \citep[section 12.4]{dudley1999}. 

\paragraph{Local linear rather than Nadaraya--Watson smoothing.}
Rather than a local linear smoother in the $x$-direction, one may be tempted to consider a simpler Nadaray--Watson smoother
$
\tilde F_{n,j} (y_j|x) = \int_{-\infty} ^{y_j} \tilde f_{n,j}(y|x) \, \diff y,
$
where 
\begin{align*}
\tilde f_{n,j}(y_j|x) = \frac{\sum_{i=1}^n K_{h_{n,1}}(x-X_i) \, L_{h_{n,2}}(y_j-Y_{ij})}{\sum_{i=1}^n K_{h_{n,1}}(x-X_i)}.
\end{align*}
However, the Nadaraya--Watson density estimator $\tilde{f}_{n,j}$ behaves poorly at the boundary points of the support $S_X$, since the true marginal density $f_X$ of $X$ is not continuous at those boundary points in view of the assumption $\inf\{ f_X(x) : x \in S_X \} \ge b > 0$ \citep{fan:1992}.
In particular, the uniform rates given in Proposition~\ref{prop:uniformCVderivatives} are no longer available, making our approach incompatible with such a Nadaraya--Watson type estimator. Some techniques to bound the metric entropy of such estimators are investigated in \cite{portier:2016}.

\paragraph{Extensions.}
The extension of Theorem \ref{theorem:weakcv} to the whole square unit $[0,1]^2$ might be obtained by verifying ``high-level'' conditions given in Theorem 2 and Corollary 1 in \cite{omelka+g+v:15}. We believe this is not straightforward and needs further research.

The extension to multivariate covariates $\bm X$ is technical and perhaps even impossible as some conflicts between the bias and the variance might arise when choosing the bandwidths $h_{n,1}$ and $h_{n,2}$ in (G\ref{cond:bandwidth2}). As noted by a Referee, using higher-order local polynomial estimators would result in lower bias (under appropriate smoothness conditions) and would (in theory) allow to extend the results presented here to higher-dimensional covariates. Still, completely nonparametric estimation would remain infeasible in high dimensions due to the curse of dimensionality in nonparametric smoothing.

The extension to vectors $\bY$ of arbitrary dimension is probably feasible, but we did not pursue this in view of the motivation from pair-copula constructions.

As suggested by a Referee, an alternative approach to estimating the conditional marginal quantiles could be to use (nonparametric) quantile regression directly instead of inverting conditional distribution functions.

\section{Auxiliary results on the smoothed local linear estimator of the conditional distribution function}
\label{sec:loc_lin_results}

This section contains asymptotic results on the smoothed local linear estimator of the conditional distribution function introduced in Section~\ref{subsec:loc_lin_def}.

The presentation is independent from the copula set-up introduced before, so that the index $j\in\{1,2\}$ will be omitted in this section. Moreover, when possible, we provide weaker assumptions on the bandwidth sequences than the ones introduced in Section~\ref{sec:main:loclin}. The proofs of the results are given in Appendix~\ref{sec:proofs}.

The main difficulty is to show that the estimator $(x,y)\mapsto \hat F_n(y|x)$ introduced in (\ref{def:ll_estimator}) satisfies (H\ref{cond:high_level_donsker}). Our approach relies on bounds on the uniform convergence rates of the estimator and its partial derivatives. Exact rates of uniform strong consistency for Nadaraya--Watson estimators of the conditional distribution function are given in \cite{einmahl2000} and \cite{hardle+j+s:1988}, among others. Strong consistency of derivatives of the Nadaraya--Watson and Gasser--M\"uller estimators of the regression function are studied in \cite{akritas+vk:2001} and \cite{gasser+m:1984}, respectively. References on strong uniform rates for local linear estimators include \cite{masry:1996} and \cite{dony:2006}.

Assume that $(X_1,Y_1), \ldots, (X_n, Y_n)$ are independent and identically distributed random vectors with common distribution equal to the law, $P$, of the random vector $(X, Y)$ valued in $\mathbb R^2$. Assume that $P$ has a density $f_{X,Y}$ with respect to the Lebesgue measure. As before, let $f_X$ and $S_X = \{x \in \R : f_X(x) > 0 \}$ denote the density and the support of $X$, respectively.  The conditional distribution function of $Y$ given $X = x$ is given by
\[
  F(y| x)= \int_{-\infty}^{y} \frac{f_{X,Y}(x,z)}{f_X(x)} \, \diff z ,\qquad y\in \R,\ x\in S_X.
\]

\begin{enumerate}[(\text{G}1')]
\item
The law $P$ of $(X, Y)$ verifies (G\ref{cond:smoothnessdensity1}), i.e., the function $F(\point|\point)$ replaces $F_j(\point|\point)$.
\end{enumerate}
\begin{enumerate}[(\text{G}1')]\setcounter{enumi}{3}
\item
\label{cond:bandwidth}
The bandwidth sequences $h_{n,1} > 0$ and $h_{n,2} > 0$ satisfy, as $n \to \infty$,
\begin{align*}
  h_{n,1}\rightarrow 0, \qquad 
  h_{n,2}\rightarrow 0, \qquad
  \frac{nh_{n,1}^{3}}{\abs{\log h_{n,1}}}  \rightarrow \infty, \qquad
  \frac{nh_{n,1}h_{n,2}}{\abs{\log h_{n,1} h_{n,2} }}  \rightarrow \infty. 
\end{align*} 
\end{enumerate}

The quantity $\hat F_n(y|x) $  can be found as the solution of a linear system of equations derived from \eqref{eq:app:optimisationLL}.  With probability going to $1$, a closed formula is available for $\hat F_n(y|x) $.  Similar expressions are available for the local linear estimator of the regression function \citep[page 55, equation (3.5)]{fan1996}.  For every  $k \in\mathbb N$, define
\begin{align*}
  \hat p_{n,k} (x) &= n^{-1}\sum_{i=1}^n w_{k,h_{n,1}}(x-X_i),\\
  \hat Q_{n,k}(y,x) &= n^{-1}\sum_{i=1}^n  \varphi_{h_{n,2}}(y,Y_{i} ) \, w_{k,h_{n,1}}(x-X_i),
\end{align*}
where $w_{k,h}(\point) = h^{-1}w_k(\point/h)$ and $w_{k} (u)= u ^k K( u )$. In what follows, the function $w_k$ plays the role of a kernel  for smoothing in the $x$-direction.  Let $\mathbb P$ denote the probability measure on the probability space carrying the sequence of random pairs $(X_1, Y_1), (X_2, Y_2), \ldots$.

\begin{lemma}\label{lemma:positivedenominator}
Let $(X_1, Y_1), (X_2, Y_2), \ldots$ be independent random vectors with common law $P$. Assume that (G\ref{cond:smoothnessdensity1}'), (G\ref{cond:kernel}) and (G\ref{cond:bandwidth}')  hold. For some $c > 0$ depending on $K$, we have
\begin{equation}
\label{eq:Pto1}
\lim_{n \to \infty}
\mathbb P\left\{  \hat p_{n,0}(x) \, \hat p_{n,2}(x)-\hat p_{n,1}(x)^2 \geq {b^2c} \  \text{ for all } \ x\in S_X\right\} 
= 1.
\end{equation}
 Consequently, with probability going to $1$, we have
\begin{align}
\label{formula:ll}
  \hat F_n(y|x) 
  = 
  \frac{\hat Q_{n,0}(y,x) \, \hat p_{n,2}(x) - \hat Q_{n,1}(y,x) \, \hat p_{n,1}(x)}%
  {\hat p_{n,0}(x) \, \hat p_{n,2}(x)-\hat p_{n,1}(x)^2},
  \qquad y \in \R, \ x \in S_X.
\end{align}
\end{lemma}

Differentiating $\hat{F}_n(y|x)$ in \eqref{formula:ll} with respect to $y$, we  obtain $\hat f_n(y|x)$, the estimated conditional density of $Y$ given $X  = x$. It is given by
\begin{align}\label{formula:llderivatives}
\hat f_n(y|x) = \frac{\hat q_{n,0}(y,x)\, \hat p_{n,2}(x) - \hat q_{n,1}(y,x)\,\hat p_{n,1}(x)}{\hat p_{n,0}(x)\, \hat p_{n,2}(x)-\hat p_{n,1}(x)^2},
\end{align}
where, for every  $k \in\mathbb N$,
\begin{align*}
&\hat q_{n,k} (y,x) = n^{-1}\sum_{i=1}^n L_{h_{n,1}}(y-Y_{i})\, w_{k,h_{n,2}}(x-X_i).
\end{align*}
To derive the uniform rates of convergence of the smoothed local linear estimator $\hat F_n(y | x)$ and its derivatives, formula~\eqref{formula:ll} permits to work with the quantities $\hat{Q}_{n,k}(y, x)$, $k \in \{0, 1\}$, and $\hat{p}_{n,k}(x)$, $k \in \{0, 1, 2\}$. The asymptotic behavior of the latter quantities is handled using empirical process theory in Proposition~\ref{prop:concentration:kernel}.  For ease of writing, we abbreviate the partial differential operator $\partial^l / \partial^l x$ to $\partial^l_x$.

\begin{proposition}\label{prop:uniformCVderivatives}
Let $(X_1, Y_1), (X_2, Y_2), \ldots$ be independent random vectors with common law $P$.  Assume that (G\ref{cond:smoothnessdensity1}'), (G\ref{cond:kernel}) and (G\ref{cond:bandwidth}')  hold. We have, as $n \to \infty$,
\begin{align*}
&\sup_{x\in S_X, \, y\in \R} \left | \hat F_n (y|x) - F(y|x) \right|= O_{\mathbb P}\left( \sqrt{\frac{\abs{\log h_{n,1}}}{n {h_{n,1}} }} +{h_{n,1}^{2}}+{h_{n,2}^{2}} \right),\\
&\sup_{x\in S_X, \, y\in \R} \left| \partial_x\{ \hat F_n (y|x) -  F(y|x)\} \right|= O_{\mathbb P}\left( \sqrt{\frac{\abs{\log h_{n,1}}}{n {h_{n,1}^{3}} }} +h_{n,1}+h_{n,1}^{-1}{h_{n,2}^{2}} \right),\\
&\sup_{x\in S_X, \, y\in \R} \left| \partial_{x}^2 \{ \hat F_n (y|x) -  F(y|x)\} \right|= O_{\mathbb P}\left( \sqrt{\frac{\abs{\log h_{n,1}} }{n {h_{n,1}^{5}} }} +1 +h_{n,1}^{-2}{h_{n,2}^{2}} \right),\\
&\sup_{x\in S_X, \, y\in \R} \left| \partial_y\{ \hat F_n (y|x) -  F(y|x)\} \right|=  O_{\prob}  \left( \sqrt {\frac{\abs{\log h_{n,1}h_{n,2}}}{n h_{n,1}h_{n,2} }} +h_{n,1}^{2}+h_{n,2}^{2} \right),\\
&\sup_{x\in S_X,\, y\in \R} \left| \partial_{y}^2\{ \hat F_n (y|x) -  F(y|x)\} \right|=  O_{\prob}  \left( \sqrt {\frac{\abs{\log h_{n,1}h_{n,2}}}{n h_{n,1}h_{n,2}^3 }} +h_{n,1}^{1+\delta} + h_{n,2}^{1+\delta}\right),\\
&\sup_{x\in S_X,\, y\in \R} \left|\partial_{y}\partial_{x} \{ \hat F_n(y|x) -  F(y|x)\}\right |= O_{\prob}  \left( \sqrt {\frac{\abs{\log h_{n,1}h_{n,2}}}{n h_{n,1}^3h_{n,2} }} +h_{n,2}^2h_{n,1}^{-1}+h_{n,1} \right).
\end{align*}
\end{proposition}

For the sake of brevity, the previous proposition is stated under (G\ref{cond:bandwidth}') although the first assertion remains true under the weaker condition $ h_{n,1} + h_{n,2} \to 0$ and ${nh_{n,1}}/{\abs{\log h_{n,1}}}\to \infty$.

Recall the generalized inverse in \eqref{eq:genInv}. Uniform convergence rates for the estimated conditional quantile function $F_n^-(u|x)$ are provided in the following proposition.

 \begin{proposition}\label{prop:quantiletransformunif}
Let $(X_1, Y_1), (X_2, Y_2), \ldots $ be independent random vectors with common law $P$. Assume that (G\ref{cond:smoothnessdensity1}'), (G\ref{cond:kernel}), (G\ref{cond:bandwidth}') hold. For any $\gamma\in (0,1/2)$, we have, as $n \to \infty$,
\begin{align*}
&\sup_{x\in S_X,\, u\in [\gamma,1-\gamma]} \left| F\big(\hat  F_{n} ^- ( u|x )|x\big)  - u \right|=O_{\mathbb P}\left( \sqrt{\frac{\abs{\log h_{n,1}}}{n {h_{n,1}} }} +{h_{n,1}^{2}}+{h_{n,2}^{2}} \right),\\
&\sup_{x\in S_X,\, u\in [\gamma,1-\gamma]} \left|\hat  F_{n} ^- ( u|x )  -  F^-(u|x) \right| =O_{\mathbb P}\left( \sqrt{\frac{\abs{\log h_{n,1}}}{n {h_{n,1}} }} +{h_{n,1}^{2}}+{h_{n,2}^{2}} \right).
 \end{align*}
\end{proposition}

The convergence rates in Propositions~\ref{prop:uniformCVderivatives} and \ref{prop:quantiletransformunif} serve to show that 
the estimated quantile functions $x\mapsto \hat F_n^- (u|x) $, as $u$ varies in $[\gamma,1-\gamma]$, belong to a certain H\"{o}lder regularity class. The bandwidths $h_{n,1}$ and $h_{n,2}$ are required to be large enough.

 \begin{enumerate}[(\text{G}1'')]\setcounter{enumi}{3}
\item \label{cond:bandwidth1bis} There exists $\alpha > 0$ such that the bandwidth sequences $h_{n,1} > 0$ and $h_{n,2} > 0$ satisfy, as $n \to \infty$,
\begin{align*}
&h_{n,1}\rightarrow 0,\qquad h_{n,2}\rightarrow 0,\qquad   h_{n,1}^{-1-\alpha/2}h_{n,2}^2 \to 0, \\
&\frac{nh_{n,1}^{3+\alpha}}{\abs{\log h_{n,1}}}   \rightarrow \infty,\qquad \frac{nh_{n,1}^{1+\alpha}h_{n,2}}{\abs{\log h_{n,1}h_{n,2}}}\to \infty .
\end{align*}  
\end{enumerate}

Note that in the case that $h_{n,1}=h_{n,2} $, the previous condition becomes $h_{n,1}\rightarrow 0$ and $  {nh_{n,1}^{3+\alpha}}/\abs{\log h_{n,1}} \to \infty$. Recall the function class $\mathcal{C}_{k+\delta,M}(S)$ defined via the H\"older condition~\eqref{eq:Holder}.

\begin{proposition} \label{prop:regularity}
Let $(X_1, Y_1), (X_2, Y_2), \ldots$ be independent random vectors with common law $P$. Assume that (G\ref{cond:smoothnessdensity1}'), (G\ref{cond:kernel}) and (G\ref{cond:bandwidth1bis}'')  hold. For any $\gamma\in (0,1/2)$, we have
\begin{align*}
  \lim_{n \to \infty}
  \mathbb P \left[ 
    \big\{x\mapsto \hat F_n^-( u |x)\ :\ u \in [\gamma,1-\gamma] \big\} 
    \subset
    \mathcal C_{1+\delta_1,M_1}(S_{X}) 
  \right] 
  = 1,
\end{align*}
where $\delta_1 = \min(\alpha/2, \delta)$ and where $M_1>0$ depends only on $b_\gamma$ and $M$. 
\end{proposition}

   \begin{appendices}

\section{Proof of Theorem \ref{theorem:weakcv_highlevel}}
\label{app:proof:weakcv_highlevel}

Condition~(G\ref{cond:copula_smoothness}) implies the existence and continuity of $\dot{C}_j$ on $\{ \bu \in [0, 1]^2 : 0 < u_j < 1 \}$, for $j \in \{1, 2\}$. As a consequence, the oracle empirical process $n^{1/2} (\hat C_n^{(or)}- C )$  converges weakly in $\ell^{\infty}([0,1]^2)$ to the tight Gaussian process  $\mathbb C$ \citep{segers:2012}. The proof of Theorem~\ref{theorem:weakcv_highlevel} therefore consists of showing equation~\eqref{eq:oP1}.

We use  notation from empirical process theory. Let $P_n=n^{-1}\sum_{i=1}^n\delta_{(X_i,\bY_i)}$ denote the empirical measure. For a  function $f$ and a probability measure $Q$, write $Qf=\int f\, \diff Q$. The empirical process is
\begin{align*}
\mathbb G_n =n^{1/2}  (P_n -P).
\end{align*}
For any pair of cumulative distribution  functions $F_1$ and $F_2$  on $\reals$, put $\bm{F} (\by)= (F_{1}(y_1)  , F_{2}(y_2)) $ for  $\by = (y_1,y_2)\in \R^2$ and $\bm{F}^- (\bu)= (F_{1}^-(u_1)  , F_{2}^-(u_2)) $ for  $\bu=(u_1,u_2)\in [0,1]^2$, the generalized inverse being defined in \eqref{eq:genInv}.

Our proof follows from an application of Theorem~2.1 stated in \cite{wellner2007} and reported below; for a proof see for instance \citet[Lemma~3.3.5]{wellner1996}, noting that the conclusion of their proof is in fact stronger than what is claimed in their statement. Let $\xi_1, \xi_2, \ldots$ be independent and identically distributed random elements of a measurable space $(\mathcal{X}, \mathcal{A})$ and with common distribution equal to $P$. Let $\mathbb P$ denote the probability measure on the probability space on which the sequence $\xi_1, \xi_2, \ldots$ is defined. Let $\mathbb{G}_{\xi,n}$ be the empirical process associated to the sample $\xi_1, \ldots, \xi_n$. Let $\mathcal{E}$  and $\mathcal{U}$ be sets and let $\{ m_{u,\eta} \;:\; u \in \mathcal{U}, \, \eta \in \mathcal{E} \}$ be a collection of real-valued, measurable functions on $\mathcal{X}$.  
\begin{theorem}[Theorem~2.1 in \cite{wellner2007}]
\label{lemmawellner}
Let $\hat{\eta}_n$ be random elements in $\mathcal{E}$. Suppose there exist $\eta_0 \in \mathcal{E}$ and $\mathcal{E}_0 \subset \mathcal{E}$ such that the following three conditions hold: 
\begin{enumerate}[(i)]
\item \label{point1thwellner} $\sup_{u \in \mathcal{U} } P\big(m_{u,\hat \eta_n}-m_{u, \eta_0}\big)^2 = o_{\mathbb{P}}(1)$ as $n \to \infty$;
\item \label{point2thwellner} $\mathbb P(\hat \eta_n\in \mathcal E_0) \to 1$ as $n \to \infty$;
\item \label{point3thwellner} $\{m_{u,\eta} - m_{u,\eta_0} \ :\ u\in \mathcal{U},\ \eta \in \mathcal E_0\}$ is $P$-Donsker.
\end{enumerate}
Then it holds that
\begin{align*}
  \sup_{u \in \mathcal{U} } 
  \abs{ \mathbb{G}_{\xi,n} \left( m_{u,\hat \eta_n}-m_{u, \eta_0}\right) } 
  = o_{\mathbb{P}}(1), 
  \qquad n \to \infty.
\end{align*}
\end{theorem}

The empirical process notation allows us to write
\begin{align*}
  \hat{C}_n(\bu) 
  &= 
  P_n \ind_{\{  \hat {\bm{F}}_{n}\, \leq\, \hat{\bm{G}}^-_{n} (\bu) \}}, &
  \hat C^{(or)}_n(\bu)
  &=
  P_n \ind_{\{   \bm{F}\,\leq\, \hat{\bm{G}}^{(or)-}_{n} (\bu) \}}.
\end{align*}
To establish (\ref{eq:oP1}), we rely on the decomposition
\begin{align*}
\nonumber
&n^{1/2}  \bigl\{  \hat C_n(\bu) - \hat C^{(or)}_n(\bu)\bigr\}\\
\nonumber
&= \mathbb G_n\left\{ \ind_{\{  \hat {\bm{F}}_{n}\, \leq\, \hat{\bm{G}}^-_{n} (\bu) \}} -\ind_{\{  \bm{F}\, \leq \,\hat{\bm{G}}^{(or)-}_{n} (\bu) \}} \right\}+
n^{1/2}  P\left\{ \ind_{\{   \hat{\bm{F}}_{n}\, \leq\, \hat{\bm{G}}^-_{n} (\bu) \}} -\ind_{\{   \bm{F}\,\leq\, \hat{\bm{G}}^{(or)-}_{n} (\bu) \}} \right\} \\
&=\hat A_{n,1}(\bu)+\hat A_{n,2}(\bu).
\end{align*}
Let $\gamma\in (0,1/2)$. The proof consists in showing that the empirical process term $\hat A_{n,1}(\bu)$ goes to zero, uniformly over $\bu\in [\gamma,1-\gamma]^2$, in probability (first step) and that the bias term $\hat A_{n,2}(\bu)$ goes to zero, uniformly over $\bu\in [\gamma,1-\gamma]^2$, in probability (second step). The simplifying assumption~\eqref{assimplifying} is crucial for treating the bias term in the second step. But before executing this program, it is useful to obtain some results on $\hat F_{n,j}$ and $\hat G_{n,j}$, $j=1,2$ (preliminary steps).

\paragraph{Preliminary step:} We establish some preliminary results on $\hat F_{n,j}$ and $\hat G_{n,j}$, $j=1,2$, that we list as facts in the following.

\begin{fact}\label{prel:consistency_inverse}
For any $j = 1,2$,
\begin{align*}
\sup_{u_j\in [\gamma,1-\gamma],\, x\in S_X }   \abs{ F_j\big(\hat  F_{n,j} ^- ( u_j |x )|x\big)  - u_j } = o_{\mathbb P} (n^{-1/4}).
\end{align*}
\end{fact}
\begin{proof}
Because of (H\ref{cond:high_level_inverse_inequality}), invoking Lemma \ref{lemma:FoF^-}, it holds that with probability going to one, 
\begin{align*}
u_j= \hat  F_{nj}\big(\hat  F_{n,j}^- ( u_j|x)|x\big),
\end{align*} 
for each $x\in S_X$ and $u_j\in [\gamma,1-\gamma]$. It follows that
\begin{align*}
 \abs{ F_j\big(\hat  F_{n,j} ^- ( u_j|x )|x\big)  - u_j }
& = \abs{ F_j\big(\hat  F_{n,j} ^- (u_j|x )|x\big)  - \hat F_{n,j}\big(\hat  F_{n,j} ^- ( u_j|x )|x\big)  }\\
&\leq  \sup_{x\in S_X, \, y_j\in \R}  \abs{ F_j(y_j|x)  - \hat F_{n,j}(y_j|x)  },
\end{align*} 
which is $o_{\mathbb P} (n^{-1/4})$ by (H\ref{cond:high_level_consistency}).
\end{proof}

A consequence of Fact \ref{prel:consistency_inverse} is that
\begin{align*}
&\inf_{x\in S_X,\, u_j\in [\gamma,1-\gamma]} F_j\big(\hat  F_{n,j} ^- ( u|x )|x\big) =o_{\mathbb P}(1) +\gamma,\\
&\sup_{x\in S_X,\, u_j \in [\gamma,1-\gamma]} F_j\big(\hat  F_{n,j} ^- ( u|x )|x\big) =o_{\mathbb P}(1) +1-\gamma.
\end{align*}
Hence, the sequence of events 
\begin{align*}
E_{n,j,1} = \{\forall x\in S_X,\, \forall u_j\in [\gamma,1-\gamma]\quad : \quad  \gamma /2\leq F_j\big(\hat  F_{n,j} ^- ( u_j|x )|x\big)\leq 1-\gamma/2  \},
\end{align*}
has probability going to one. From (H\ref{cond:high_level_consistency}) we have that, with probability going to $1$, for any $x\in S_X$ and $y_j\geq F_j^{-}(1-\gamma/2|x) $, it holds that $\hat F_{n,j}(y_j|x) \geq 1-\gamma$. Hence the sequence of events
\begin{align*}
E_{n,j,2} = \{ \forall x\in S_X\quad :\quad   \inf_ {    y_j\geq F_j^{-}(1-\gamma/2|x) } \hat F_{n,j}(y_j|x) \geq 1-\gamma  \},
\end{align*}
has probability going to one as well.

\begin{fact}\label{prel:equivalence_inverse}
On a sequence of events whose probabilities tend to one, it holds that for every $u_j\in [\gamma,1-\gamma]$ and every $(y_j,x)\in \R\times S_X$,
\begin{align*}
\hat F_{n,j}(y_j|x) \,\leq\, u_j\
\Leftrightarrow \ 
y_j \,\leq\, \hat F_{n,j}^- (u_j |x).
\end{align*}
\end{fact}
\begin{proof}
The sense $\Leftarrow$ is an easy consequence of (H\ref{cond:high_level_inverse_inequality}): because of the continuity of $y_j\mapsto \hat F_{n,j} (y_j|x) $ we can apply Lemma \ref{lemma:FoF^-} to obtain the implication. The converse direction ``$\Rightarrow$'' requires the fact that  $\hat{F}_{n,j}(\,\cdot\,|x)$ is strictly increasing on $[F_{j}^{-}(\gamma|x),F_{j}^{-}(1-\gamma|x)]$, which is given by (H\ref{cond:high_level_inverse_inequality}). Assume  that $u_j\in [\gamma,1-\gamma]$ and $(y_j,x)\in \R\times S_X$ are such that $\hat F_{n,j}^-( u_j|x)<y _j $. Consider two cases, according to whether $y_j$ is smaller than $F_{j}^-(1-\gamma/2|x)$ or not.
\begin{compactitem}
\item On the one hand, suppose that  $ y_j<F_{j}^-(1-\gamma/2|x)   $. We have, under $E_{n,j,1}$, that $ F_j^{-}(\gamma /2|x) \leq \hat  F_{n,j} ^- ( u_j|x )$. Hence we get the chain of inequalities
\begin{align*}
F_{j}^-(\gamma/2|x) \leq \hat F_{n,j}^-( u_j|x)<y _j<F_{j}^-(1-\gamma/2|x).
\end{align*}
Using the monotonicity of $\hat F_{n,j}(\cdot|x)$ on $[F_{j}^{-}(\gamma/2|x),F_{j}^{-}(1-\gamma/2|x)]$ and Lemma~\ref{lemma:FoF^-}, we find that $u_j< \hat F_{n,j}(y_j|x)$.
\item On the other hand, suppose that  $F_{j}^-(1-\gamma/2|x) \leq  y _j $, or equivalently that $ 1-\gamma/2 \leq F_{j}(y_j|x)$. Under $E_{n,j,2}$, it then holds that $u_j\leq  1-\gamma < \hat F_{n,j}(y_j|x)$.
\end{compactitem}
\end{proof}

\begin{fact} \label{prel:unif_conv_G}
We have
\begin{align*}
\sup_{u_j\in[\gamma,1-\gamma] } \abs{\hat G_{n,j} (u_j) - \hat G_{n,j}^{(or)} (u_j)}  = o_{\mathbb P} (n^{-1/4}).
\end{align*}
\end{fact}
\begin{proof}
From Fact \ref{prel:equivalence_inverse}, it holds that, on a sequence of events whose probabilities tend to one, with a slight abuse of notation,
\begin{align*}
&\hat{G}_{n,j} (u_j)-\hat{G}_{n,j}^{(or)} (u_j) \\
&= n^{-1/2} \mathbb G_n\left\{ \ind_{\{ Y \,\leq\, \hat  F_{n,j}^-(u_j|X) \}} -\ind_{\{   Y \,\leq\,  F_{j}^{-} ( u_j|X) \}} \right\}+
P\left \{ \ind_{\{ \hat  F_{n,j} \,\leq\, u_j\}} -\ind_{\{   F_j \,\leq\, u_j \}}\right\} .
\end{align*}
We apply Theorem~\ref{lemmawellner} with $\xi_i = ( Y_{i,j}, X_i )$, $\mathcal{X} = \reals \times S_X$, $\mathcal{U} = [\gamma, 1-\gamma]$ and $\mathcal{E} $  the space of measurable functions valued in $\R$ and defined on $[\gamma,1-\gamma]  \times S_X$. Moreover, the quantities $\eta_0$, $\hat \eta_n$, and the map $m_{u_j,\eta}$ are given by, for every $u_j \in [\gamma,1-\gamma]$ and $x\in S_X$,
\begin{align*}
  \eta_0(u_j,x)& = {F}_j^-(u_j|x), \\
  \hat \eta_n(u_j,x)&=\hat F_{n,j}^-(u_j|x),\\
  m_{u_j,\eta}(y,x ) &=  \ind_{\{y\,\leq\, {\eta}(u_j  ,x)\}}  .
 \end{align*}  
Finally, the space $\mathcal{E}_0$ is the collection of those elements  $\eta$ in $\mathcal{E}$  such that
\begin{align*}
  \{ x\mapsto \eta(u_j,x) \, :\, u_j \in [\gamma,1-\gamma] \} 
  &\subset
   \mathcal C_{1+\delta_1,M_1}(S_{X}) .
\end{align*}
The verification of the three assumptions in Theorem~\ref{lemmawellner} is as follows:
\begin{compactitem}
\item
First, we show point (i). Recall that if  the random variable $U$ is uniformly distributed on $(0, 1)$, then $\operatorname{E}(\ind_{\{U\leq u_1 \}}-\ind_{\{U\leq u_2 \}})^2 = \abs{u_1-u_2} $. We have
\begin{align*}
&  \int
    \abs{
      \ind_{\{ y \,\leq\, \hat  F_{n,j}^-(u_j|x) \}} 
      - 
      \ind_{\{ y \,\leq\, F_j^{-} ( u_j|x) \}}
    }^2
  f_{X,Y}(x, y) \, \diff (x, y)
\\
  &= 
  \int  \abs{ F\bigl(\hat F_{n,j}^-(u_j|x)|x\bigr) - u_j } \, f_X(x)\,  \diff x \\
  &\leq 
  \sup_{u_j\in[\gamma,1-\gamma],\, x\in S_X } 
  \abs{ F\big(\hat F_{n,j}^-(u_j|x)|x\big) - u_j },
\end{align*} 
which, by Fact \ref{prel:consistency_inverse}, tends to zero in probability.
\item
Second, point (ii) is directly obtained invoking (H\ref{cond:high_level_donsker}).
\item
Third, point (iii) follows from the existence of $\delta_2>0 $ and  $M_2>0$ such that 
\begin{align}\label{inclusin:smouthness}
\{
    x \mapsto F_{j}^-( u_j|x)\ :\ 
    u_j \in [\gamma,1-\gamma] 
 \}\subset 
 \mathcal C_{1+\delta_2,M_2}(S_{X}).
\end{align}
The inclusion is indeed implied by the formula
\begin{align*}
  \partial_x  F^-_{j} (u_j  |x) 
  = -\left.\frac{\partial _x  F_{j} (y_j |x) }{ f_{j}  (y_j|x) }\right|_{y =  F_{j}^- (u_j  |x) },
\end{align*}
which, by (G\ref{cond:smoothnessdensity1}), is bounded by $M/b_\gamma$. Then, based on (G\ref{cond:smoothnessdensity1}), we easily obtain that the function $x \mapsto \partial_x F_{j}^-(u_j|x)$ is $\delta$-H\"older with H\"older constant depending only on $b_\gamma$ and $M$ (for more details, the reader is invited to read the proof of Proposition \ref{prop:regularity}). It remains to note that under (G\ref{cond:smoothnessdensity1}), the set $S_X$ is bounded and convex, implying that for any $\delta>0$ and $M>0$, the class of subgraphs of $\mathcal C_{1+\delta,M}(S_{X})$ is Donsker \citep[Corollary 2.7.5]{wellner1996}. As the difference of two Donsker classes remains Donsker \citep[Example~2.10.7]{wellner1996}, we obtain point (iii).
\end{compactitem}
Consequently, we have shown that 
\begin{align*}
\hat{G}_{n,j} (u_j)-\hat{G}_{n,j}^{(or)} (u_j) =
 \int \left\{ F\big( \hat  F_{n,j}^- ( u_j|x)|x\big) -  u_j\right \}\,  f_X(x)\, \diff x  +o_{\mathbb P}(n^{-1/2}).
\end{align*}
Conclude invoking Fact \ref{prel:consistency_inverse}. 
\end{proof}

\begin{fact}\label{prel:unif_conv_G^-1}
We have
\begin{align*}
\sup_{u_j\in[\gamma,1-\gamma] } \abs{\hat G_{n,j}^- (u_j) - \hat G_{n,j}^{(or)-} (u_j)} = o_{\mathbb P} (n^{-1/4}).
\end{align*}
\end{fact}
\begin{proof}
By Fact \ref{prel:unif_conv_G}, the supremum distance $\epsilon_n = \sup_{u_j \in [\gamma,1-\gamma]} \lvert \hat{G}_{n,j}(u_j)-\hat{G}_{n,j}^{(or)} (u_j) \rvert$ converges to zero in probability. We work on the event $\{ \epsilon_n < \gamma \}$, the probability of which tends to $1$. By Lemma~\ref{lemma:bracketingquantile}, we have, for $u_j \in [\gamma, 1-\gamma]$,
\begin{multline*}
  \abs{ \hat{G}_{n,j}^{-} (u_j)-   \hat{G}_{n,j}^{(or)-}(u_j)}\\
  \leq 
  \abs{ \hat{G}_{n,j}^{(or)-} ((u_j-\epsilon_n)\vee 0) -\hat{G}_{n,j}^{(or)-} (u_j)}
  \vee  
  \abs{ \hat{G}_{n,j}^{(or)-} ((u_j+\epsilon_n)\wedge 1) -\hat{G}_{n,j}^{(or)-} (u_j)},
\end{multline*}
with $a \vee b = \max(a, b)$.  In terms of $\mathbb W_n (u_j)= \sqrt n \{ \hat{G}_{n,j}^{(or)-} (u_j) - u_j \}$,  we have
\begin{multline*}
  \abs{ \hat{G}_{n,j}^{(or)-} ((u_j-\epsilon_n)\vee 0) -\hat{G}_{n,j}^{(or)-} (u_j) }
  \vee  
  \abs{ \hat{G}_{n,j}^{(or)-} ((u_j+\epsilon_n)\wedge 1) -\hat{G}_{n,j}^{(or)-} (u_j) } 
  \\
  \leq \epsilon_n + 2\sup_{u_j\in [0,1] } |\mathbb W_n (u_j)|.
\end{multline*}
From Fact \ref{prel:unif_conv_G} and Lemma~\ref{lem:Vervaat:random}, we get the desired rate $o_{\mathbb P}( n^{-1/4})$. 
\end{proof}

For $u_j\in [\gamma,1-\gamma]$, $x \in S_X$, and $j \in \{1, 2\}$, define
\begin{align}\label{def:deltaj}
\hat \Delta_{n,j}(u_j|x)= {F}_{j} \bigl(\hat  {{F}}_{n,j}^-\bigl( \hat{G}^-_{n,j} (u_j)|\, x\bigr)|\, x\bigr)-\hat{G}^{(or)-}_{n,j} (u_j).
\end{align}

\begin{fact}\label{prel:unif_conv_Delta}
We have
\begin{align*}
\sup_{u_j\in [\gamma, 1-\gamma],\, x\in S_X } |\hat \Delta_{n,j}(u_j|x)| & = o_{\mathbb P} (n^{-1/4}).
\end{align*}
\end{fact}

\begin{proof}
Write
\begin{align*}
\hat \Delta_n(u_j|x) 
  =  
  \left| 
    F\bigl(\hat  F_n ^-\bigl( \hat{G}_{n,j}^- (u_j)|\,x \bigr)|\,x\bigr)  - \hat{G}_{n,j}^{-} (u_j) 
  \right|
  +
  \left|
    \hat{G}_{n,j}^{-} (u_j)- \hat{G}_{n,j}^{(or)-} (u_j)
  \right|.
\end{align*}
By Fact \ref{prel:unif_conv_G^-1}, we only need to treat the first term on the right-hand side. Using the fact that $ \sup_{u_j\in [\gamma,1-\gamma]} \lvert \hat{G}_{n,j}^{-} (u_j)-   \hat{G}_{n,j}^{(or)-}(u_j) \rvert  = o_{\mathbb P}(1)$ and the fact that $ \sup_{u_j\in (0,1]} \lvert   \hat{G}_{n,j}^{(or)-}(u_j) - u_j \rvert =o_{\mathbb P}( 1)$, which is a consequence of Lemma \ref{lem:Vervaat:random}, we  know that $ \hat{G}_{n,j}^{-} (u_j)$ takes values in $[\gamma/2,1-\gamma/2]$ with probability going to $1$. Then we use Fact \ref{prel:consistency_inverse} to conclude that $\hat \Delta_n(u_j|x)  = o_{\mathbb P}(n^{-1/4})$.
\end{proof}

\paragraph{First step:} We show that
\begin{align*}
  \sup_{{\bu}\in [\gamma,1-\gamma]^2} \abs{\hat A_{n,1}({\bu})} 
  = o_{\mathbb P}(1), \qquad n \to \infty.
\end{align*}
By Fact \ref{prel:equivalence_inverse}, it holds that (with a slight abuse of notation)
\begin{align*}
  \hat A_{n,1}(\bu)
  &= 
  \mathbb G_n \left\{ 
    \ind_{\{ \bY\, \leq\,  \hat {\bm F}_{n}^-(\hat{\bm{G}}^-_{n} (\bu)|X) \}}
    -
    \ind_{\{ \bY\, \leq\,  \bm{F}^-(\hat{\bm{G}}^{(or)-}_{n} (\bu)|X) \}}
  \right\}.
\end{align*}
Therefore we apply Theorem~\ref{lemmawellner} with $\xi_i = (X_i, \bY_i )$, $\mathcal{X} =S_X \times \reals^2  $, $\mathcal{U} = [\gamma, 1-\gamma]^2$ and $\mathcal{E} $  the space of measurable functions valued in $\R^4$ and defined on $S_X  \times [\gamma,1-\gamma]^2   $. Moreover, the quantities $\eta_0$ and $\hat \eta_n$ are 
 given by, for every $\bu \in [\gamma,1-\gamma]^2$ and $x\in S_X$,
\begin{align*}
& \eta_0(\bu,x)= \left(\bm{F}^-(\bu|x),\bm{F}^-(\bu|x)\right), \\
& \hat \eta_n(\bu,x)=\left(\hat{\bm{F}}_{n}^-\big(\hat{\bm{G}}^-_{n} (\bu)|x\big),\bm{F}^-\big(\hat{\bm{G}}^{(or)-} _{n}(\bu)|x\big)\right).
\end{align*}
Identifying $u \in \mathcal{U}$ with $\bm{u} \in [\gamma, 1-\gamma]^2$ and $\eta \in \mathcal{E}$ with $({\bm{\eta}}_1,{\bm{\eta}}_2)$, where $\bm{\eta}_j$, $j\in\{1,2\}$, are valued in $\R^2$, the map $m_{u,\eta} : \reals^2 \times S_X \to \reals$ is given by
\begin{align*}
m_{u,\eta}(\by,x ) =  \ind_{\{\by\, \leq \, \bm{\eta}_1(\bu  ,x)\}}-\ind_{\{\by\, \leq  \, \bm{\eta}_2(\bu  ,x)\}} , 
\end{align*}
Finally, the space $\mathcal{E}_0$ is the collection of those elements  $\eta = ({\bm\eta}_1,{\bm\eta}_2)$ in $\mathcal{E}$  such that
\begin{align*}
  \{ x \mapsto \bm{\eta}_1(\bu, x) \, :\, \bu \in [\gamma,1-\gamma]^2 \} 
  &\subset
  \bigl( \mathcal C_{1+\delta_1,M_1}(S_{X}) \bigr)^2,\\
  \{ x\mapsto \bm{\eta}_2(\bu,x) \, :\, \bu\in [\gamma,1-\gamma]^2 \} 
  &\subset
  \bigl( \mathcal C_{1+\delta,M_2}(S_{X}) \bigr)^2,
\end{align*}
where $M_2$ depends only on $b_\gamma$ and $M$. In the following we check each condition of Theorem~\ref{lemmawellner}.

\textit{Verification of Condition~(\ref{point1thwellner}) in Theorem~\ref{lemmawellner}.} 
Because the indicator function is bounded by $1$, we have
\begin{multline*}
  \int 
    \abs{
      \ind_{\{\by\,\leq\, \hat{\bm{F}}_{n}^-( \hat{\bm{G}}^- _{n}(\bu) |x)\}}
      -
      \ind_{\{\by \, \leq\,   \bm{F}^-( \hat{\bm{G}}^{(or)-}_{n} (\bu) |x)\}}
    }^2
  f_{X, \bY}(x, \by) \, \diff (x, \by) \\
  \le
  \sum_{j=1}^2 \sup_{u_j \in [\gamma, 1-\gamma]}
  \int
    \abs{
      \ind_{\{y_j\,\leq\, \hat F_{n,j}^-( \hat{G}^-_{n,j} (u_j) |x)\}}
      -
      \ind_{\{y_j\,\leq\,  F_{j}^-( \hat{G}^{(or)-}_{n,j} (u_j) |x)\}}
    }^2
  f_{X, Y_j}(x, y_j) \, \diff (x, y_j),
\end{multline*}
so that we can focus on each margin separately. Recall that if  the random variable $U$ is uniformly distributed on $(0, 1)$, then $\operatorname{E}(\ind_{\{U\leq u_1 \}}-\ind_{\{U\leq u_2 \}})^2 = \abs{u_1-u_2} $. Writing $\hat a_{n,x}(u_j)= \hat F_{n,j}^-\big( \hat{G}^-_{n,j} (u_j) |x\big) $, and using~\eqref{assimplifying} and~\eqref{factuniformdist}, we have
\begin{align*}
  \lefteqn{
  \int
    \abs{
      \ind_{\{y_j\,\leq\, \hat a_{n,x}(u_j) \}}
      -
      \ind_{\{y_j\,\leq\,  F_{j}^-( \hat{G}^{(or)-}_{n,j} (u_j) |x)\}}
    }^2
  f_{X, Y_j}(x, y_j) \, \diff (x, y_j)
  } \\
  &=
  \int
    \abs{ 
      \ind_{\{ F_{j}(y_j|x)\,\leq\,   F_{j}(\hat a_{n,x}(u_j)|x)\}}
      -
      \ind_{\{ F_{j}(y_j|x)\,\leq\,    \hat{G}^{(or)-}_{n,j} (u_j)\}}
    }^2
  f_{X, Y_j}(x, y_j) \, \diff (x, y_j)
  \\
  &=
  \int_{S_X}
    \abs{ F_{j}( \hat a_{n,x}(u_j) | x)-\hat{G}^{(or)-}_{n,j} (u_j) }
  f_X(x) \, \diff x
  \\
  &= 
  \int_{S_X}  \abs{ \hat \Delta_{n,j}(u_j |x ) } \, f_X(x) \, \diff x,
\end{align*}
where $\hat \Delta_{n,j} $ has been defined in \eqref{def:deltaj}. Fact \ref{prel:unif_conv_Delta}, demonstrated during the preliminary step, permits to conclude.

\textit{Verification of Condition~(\ref{point2thwellner}) in Theorem~\ref{lemmawellner}.} We establish that, for each $j=1,2$,
\begin{align*}
\left\{x\mapsto \hat F_{n,j}^-\big( \hat{G}^-_{n,j} (u_j)|x\big)\ :\ u_j\in [\gamma,1-\gamma]\right\} &\subset \mathcal C_{1+\delta_1,M_1}(S_{X}),\\
\left\{x\mapsto F_{j}^-\big(\hat{G}^{(or)-} _{n,j}(u_j)|x\big)\ :\ u_j\in [\gamma,1-\gamma] \right\} &\subset \mathcal C_{1+\delta,M_2}(S_{X}),
\end{align*}
with probability going to $1$.
\begin{compactitem}
\item
For the  first inclusion, we have already shown in the proof of Fact \ref{prel:unif_conv_Delta} that $ \hat{G}_{n,j}^-  (u _j)\in [\gamma/2,1-\gamma/2] $, for every $u_j\in [\gamma,1-\gamma] $, with probability going to $1$. On this set, we have 
\[
  \bigl\{ 
    x \mapsto \hat F_{n,j}^-\big( \hat{G}_{n,j}^-  (u_j ) |x\big)\ :\ 
    u_j \in [\gamma,1-\gamma]
  \bigr\}  
  \subset 
  \bigl\{
    x \mapsto \hat F_{n,j}^-(u_j  |x)\ :\ 
    u_j \in [\gamma/2,1-\gamma/2]
  \bigr\},
\]
which is included in $\mathcal C_{1+\delta_1,M_1}(S_{X})$ by (H\ref{cond:high_level_donsker}).
\item
For the  second inclusion, by Lemma~\ref{lem:Vervaat:random}, for every $u_j\in [\gamma,1-\gamma]$, it holds that $\hat{G}^{(or)-}_{n,j} (u_j)\in [\gamma/2,1-\gamma/2]$ with  probability going to one. It follows that 
\[
  \bigl\{ 
    x \mapsto F_{j}^-\big(\hat{G}^{(or)-} _{n,j}(u_j)|x\big)\ :\ 
    u_j\in [\gamma, 1-\gamma]
  \bigr\}
  \subset 
  \bigl\{
    x \mapsto F_{j}^-( u_j|x)\ :\ 
    u_j\in [\gamma/2,1-\gamma/2] 
  \bigr\},
\] 
which is included in $\mathcal C_{1+\delta,M_2}(S_{X})$ by (\ref{inclusin:smouthness}), for some $M_2>0$.
\end{compactitem}

\textit{Verification of Condition~(\ref{point3thwellner}) in Theorem~\ref{lemmawellner}.} It is enough to show that the class of functions 
\begin{align*}
  \left\{
    \ind_{\{\by\,\leq\,  \bm {g_1}(x)\}} - \ind_{\{\by\,\leq\,  \bm{g_2}(x)\}}
    \quad  :\quad 
    (\bm{g_1},\bm{g_2}) \in 
      \bigl( \mathcal C_{1+\delta_1,M_1}(S_{X}) \bigr)^2 
      \times 
      \bigl( \mathcal C_{1+\delta,M_2}(S_{X}) \bigr)^2  
  \right\}
\end{align*}
is $P$-Donsker. Since the sum and the product of two bounded Donsker classes is Donsker \citep[Example~2.10.8]{wellner1996}, we can focus on the class
\begin{align*}
\big\{\ind_{\{y\,\leq\,  g(x)\}}\quad  :\quad g\in \mathcal C_{1+\delta,M}(S_{X})\big \}.
\end{align*}
For any $\delta>0$ and $M>0$, the latter is Donsker since the class of subgraphs of $\mathcal C_{1+\delta,M}(S_{X})$, under (G\ref{cond:smoothnessdensity1}), has a sufficiently small entropy \citep[Corollary~2.7.5]{wellner1996}.

\paragraph{Second step:} We show that
\begin{align*}
  \sup_{{\bu}\in [\gamma,1-\gamma]^2} \abs{ \hat A_{n,2}({\bu}) }
  = o_{\mathbb P}(1), \qquad n \to \infty.
\end{align*}
Because of the simplifying assumption \eqref{assimplifying}, we have,  for every $\bu\in [0,1]^2$, the formula
\begin{align}
\label{formulaoperator}
  \int \ind_{\{\by\,\leq\, \hat {\bm{F}} ^-_{n}(\hat{\bm{G}}^-_{n} (\bu)|x)\}} \, f_{X, \bY}(x, \by) \, \diff(x, \by)
  = 
  \int 
    C \big( 
      \bm{F} \big(
	\hat{\bm{F}}^-_{n}\big(\hat{\bm{G}}^-_{n} (\bu)|x \big) \mid x
      \big) 
    \big)\, 
    f_X(x) \, 
  \diff x.
\end{align}
Next we  use (G\ref{cond:copula_smoothness}) to expand $C$ in the previous expression around $ \hat{\bm{G}}^{(or)-}_{n} (\bu)$ and then we conclude by using the rates for the quantities $\hat \Delta_{n,j}(u_j|x)$, $ j= 1,2$, established in Fact \ref{prel:unif_conv_Delta}.

In the light of the first step, we have
\begin{align*}
  \sup_{\bu\in [\gamma,1-\gamma]^2} 
  \left| 
    n^{1/2}  \{  \hat C_n(\bu) - \hat C_n^{(or)}(\bu) \} 
    -
    n^{1/2} P 
    \left\{ 
      \ind_{\{ \hat{\bm{F}}_{n}\,\leq\, \hat{\bm{G}}_{n}^- (\bu)\}} 
      -
      \ind_{\{   \bm{F}\,\leq\, \hat{\bm{G}}_{n}^{(or)-} (\bu) \}}
    \right\} 
  \right| 
  = o_{\mathbb P}(1).
\end{align*}
Moreover, because $\hat C_n(u_1,1) = \hat G_{n,1}\big(\hat G_{n,1}^-(u_1)\big )$ and $\hat C_n^{(or)}(u_1,1) = \hat G_{n,1}^{(or)}\big(\hat G_{n,1}^{(or)-}(u_1) \big)$, we find 
\begin{align*}
  \sup_{u_1\in [\gamma,1-\gamma] } \abs{  \hat C_n(u_1,1) - u_1} = O(n^{-1})
  \qquad \text{and} \qquad 
  \sup_{u_1\in [\gamma,1-\gamma] } \abs{ \hat C_n^{(or)}(u_1,1) -u_1} =O(n^{-1}),
\end{align*}  
which  implies, by the triangle inequality, that
\begin{align*}
  \sup_{u_1\in [\gamma,1-\gamma] } 
  \abs{n^{1/2} \{ \hat C_n(u_1,1) - \hat C_n^{(or)}(u_1,1) \}}
  = o(1).
\end{align*}
Similarly, we find
\begin{align*}
  \sup_{u_2\in [\gamma,1-\gamma] } 
  \abs{n^{1/2} \{ \hat C_n(1,u_2) - \hat C_n^{(or)}(1,u_2) \} } 
  = o(1) .
\end{align*}
Bringing all these facts together yields, for any $j=1,2$,
\begin{align*}
& \sup_{u_j\in [\gamma,1-\gamma] }   \left| n^{1/2} P\left\{ \ind_{\{ \hat  F_{n,j}\,\leq\, \hat{G}^-_{n,j} (u_j) \}} -\ind_{\{  F_{j}\,\leq\, \hat{G}^{(or)-}_{n,j} (u_j) \}}\right\}\right| = o_{\mathbb P}(1).
\end{align*}
Hence, from the definition of  $\hat \Delta_{n,j}(u_j|x)$ given in equation~\eqref{def:deltaj}, it  follows that, for any $j=1,2$,
\begin{align}\label{eq:neglectabilitymargins}
  \sup_{u_j\in [\gamma,1-\gamma] }  
  \abs{ n^{1/2} \int \hat \Delta_{n,j}(u_j|x)\,  f_X(x)\, \diff x } 
  = o_{\mathbb P}(1),
\end{align}
Next define 
\begin{align*}
  \hat W_n(\bu|x) 
  &= 
  \hat \Delta_{n,1}(u_1|x) \, \dot{C}_1(\bm{u})
  +  
  \hat \Delta_{n,2}(u_2|x) \, \dot{C}_2(\bm{u}),
  \\
  \hat W_n(\bu) 
  &= \int \hat W_n(\bu|x) \, f_X(x)\, \diff x.
\end{align*}
 By~\eqref{eq:neglectabilitymargins}, it holds that $ \sup_{\bu \in [\gamma,1-\gamma]^2 } \lvert \sqrt n  \, \hat W_n(\bu) \rvert = o_{\mathbb P}(1)$. Then because of~\eqref{formulaoperator} and the simplifying assumption~\eqref{assimplifying}, we have
\begin{align*}
  \lefteqn{
    \hat A_{n,2}(\bu)- \hat W_n(\bu) 
  } \\
  &=
  n^{1/2} \int_{S_X} 
  \left\{
    H \bigl( \hat{\bm{F}} _{n}^-\bigl(\hat{\bm{G}}_{n}^- (\bu) | x \bigr) | x \bigr) -
    H\big(  \bm{F}^-\bigl(\hat{\bm{G}}_{n}^{(or)-} (\bu) | x \bigr) | x \bigr) -
    \hat W_n(\bu|X)
  \right\}
  f_X(x) \, \diff x \\
  &=
  n^{1/2} \int_{S_X}
  \left\{
    C\left(\bm{F}\bigl( \hat{\bm{F}}_{n}^-\bigl(\hat{\bm{G}}_{n}^- (\bu) |x\bigr)|x\bigr) \right) -C\left(\hat{\bm{G}}_{n}^{(or)-}(\bu)\right) -
    \hat W_n(\bu|x)
  \right\}
  f_X(x) \, \diff x.
\end{align*}
The  term inside the second integral equals
\begin{align*}
  C\left(\hat{\bm{G}}_{n}^{(or)-}(\bu) + \hat{\bm{\Delta}}_n(\bu|x) \right)   
  -
  C\left(\hat{\bm{G}}_{n}^{(or)-}(\bu)\right)
  -
  \hat W_n(\bu|x),
\end{align*}
with $\hat {\bm{\Delta}}_n(\bu|x)= (\hat \Delta_{n,1}(u_1|x), \, \hat \Delta_{n,2}(u_2|x))$. Then, by Lemma~\ref{lem:ddotC} [which we can apply since $\hat{\bm{G}}_{n}^{(or)-}(\bu) $ lies in the interior of the unit square], we have, for all $\bu \in [\gamma,1-\gamma]^2$ and all $x \in S_X$,
\[
  \abs{ C\left(\hat{\bm{G}}_{n}^{(or)-}(\bu) +\hat {\bm{\Delta}}_n(\bu|x) \right)   -C\left(\hat{\bm{G}}_{n}^{(or)-}(\bu)\right) -\hat W_n(\bu|x) }
  \le \frac{4\kappa}{\gamma} \, \abs{ \hat {\bm{\Delta}}_n( \bu | x ) }^2.
\]
Integrating out over $x \in S_X$ yields
\begin{align*}
  \abs{ \hat {A}_{n,2}( \bu ) - \hat {W}_n( \bu ) }
  &\le {n}^{1/2} \, \frac{4\kappa}{\gamma} \, \int_{S_X} \abs{ \hat {\bm{\Delta}}_n( \bu | x ) }^2 \, f_X(x) \, \diff x\\
  &\le {n}^{1/2} \, \frac{4\kappa}{\gamma} \, \sup_{\bu\in [\gamma,1-\gamma]^2,\, x\in S_X} \abs{ \hat {\bm{\Delta}}_n( \bu | x ) }^2 ,
\end{align*}
which is $o_{\mathbb P} (1)$ by Fact \ref{prel:unif_conv_Delta}.
\qed

\section{Proof of Theorem \ref{theorem:weakcv}}
\label{app:proof_th_weak_cv}

The proof follows from an application of Theorem \ref{theorem:weakcv_highlevel}. We need to show that (H\ref{cond:high_level_inverse_inequality}), (H\ref{cond:high_level_consistency}) and (H\ref{cond:high_level_donsker}) are valid. The last two assertions are direct consequences of Proposition \ref{prop:uniformCVderivatives} and Proposition \ref{prop:regularity}, respectively. Those propositions are stated in section \ref{sec:loc_lin_results}.

Hence, we only need to show that (H\ref{cond:high_level_inverse_inequality}) holds. By Lemma~\ref{lemma:positivedenominator}, there exists $c > 0$  depending on $K$ such that the  events
\begin{align*}
  E_{n,3} = \left\{  
    \inf_{x \in S_X}
    \left\{ \hat p_{n,0}(x)\,\hat p_{n,2}(x)-\hat p_{n,1}(x)^2 \right\} \geq {b^2c}{} 
  \right\} ,
\end{align*}
 have probability going to $1$. Under  $E_{n,3}$, equations \eqref{formula:ll} and \eqref{formula:llderivatives} hold, i.e.,
\begin{align*}
&\hat F_{n,j}(y_j|x) = \frac{\hat Q_{n,j,0}(y_j,x)\, \hat p_{n,2}(x) - \hat Q_{n,j,1}(y_j,x)\,\hat p_{n,1}(x)}{\hat p_{n,0}(x)\,\hat p_{n,2}(x)-\hat p_{n,1}(x)^2},\\
&\hat f_{n,j}(y_j|x) = \frac{\hat q_{n,j,0}(y_j,x)\, \hat p_{n,2}(x) - \hat q_{n,j,1}(y_j,x)\, \hat p_{n,1}(x)}{\hat p_{n,0}(x)\, \hat p_{n,2}(x)-\hat p_{n,1}(x)^2},
\end{align*}
with $\hat f_{n,j}(y_j|x) = \partial _{y_j}\hat F_{n,j}(y_j|x) $, and, for every $k \in\mathbb N$,
\begin{align*}
&\hat q_{n,j,k} (y_j,x) = n^{-1}\sum_{i=1}^n L_{h_{n,2}}(y_j-Y_{ij})\, w_{k,h_{n,1}}(x-X_i),\\
&\hat Q_{n,j,k}(y_j,x) =\int_{-\infty}^{y_j} \hat q_{n,j,k}(t,x)\, \diff t .
\end{align*}
Hence, on the event $E_{n,3}$, $y_j\mapsto \hat F_{n,j}(y_j|x)$ is continuous and differentiable on $\mathbb R$. By the uniform convergence of $\hat f_{n,j}$, stated in Proposition~\ref{prop:uniformCVderivatives}, and using (G\ref{cond:smoothnessdensity1}), we get for all $x\in S_X$ and all $y_j\in [F_{j}^{-}(\gamma|x),F_{j}^{-}(1-\gamma|x)]$ that 
\begin{align*}
\hat f_{n,j} ( y_j|x) &\geq   f_{j} ( y_j|x)  - \sup_{x\in S_X,\, y_j\in \R} \abs{\hat f_{n,j} ( y_j|x)-  f_{j} ( y_j|x)}\\
& \geq b-o_{\mathbb P}(1).
\end{align*}
Hence the event
\begin{align*}
  E_{n,4} = \left\{
    \min_{j \in \{1, 2\}} \inf_{ x\in S_X } \inf _{y_j\in [F_{j}^{-}(\gamma|x),F_{j}^{-}(1-\gamma|x)] } 
    \hat f_{n,j} ( y_j|x)
    \geq b/2
  \right\}.
\end{align*}
has probability going to $1$. Conclude by noting that under  $E_{n,4}$, for every $x\in S_X$, the function $F_{n,j}(\point |x)$ is strictly increasing on $[F_{j}^{-}(\gamma|x),F_{j}^{-}(1-\gamma|x)]$.
\qed

\section{Proofs of the auxiliary results on the smoothed local linear estimator}

\subsection{Uniform convergence rates for kernel estimators}

Our analysis of the smoothed local linear estimator relies on a result on the concentration around their expectations of certain kernel estimators. The result follows from empirical process theory and notably by the use of some version of Talagrand's inequality \citep{talagrand1994}, formulated in \cite{gine+g:02}. For stronger results including exact almost-sure convergence rates, see \cite{einmahl2000}. 

\begin{proposition}\label{prop:concentration:kernel}
Let $(X_i,Y_i)$, $i=1,2,\ldots,$ be independent copies of a bivariate random vector $(X,Y)$. Assume that $(X,Y)$ has a bounded density $f_{X,Y}$ and that $K$ and $L$ are bounded real functions of bounded variation that vanish outside $[-1,1]$. For any sequences $h_{n,1}\rightarrow 0$ and $h_{n,2}\rightarrow 0$ such that $nh_{n,1} / \abs{\log h_{n,1}} \to \infty $ and $ {n h_{n,1}h_{n,2} } /\abs{\log h_{n,1}h_{n,2}}\to \infty  $, we have, as $n \to \infty$, 
\begin{align*}
&\sup_{x\in \R}\left | n^{-1} \sum_{i=1}^n \left\{ K_{h_{n,1}}(x-X_i) - E[ K_{h_{n,1}}(x-X)  ]\right\} \right |= O_{\mathbb P}\left(\sqrt {\frac{\abs{\log h_{n,1}}}{n {h_{n,1}} }} \right),\\
&\sup_{(x,y)\in \R^2}\left | n^{-1} \sum_{i=1}^n \left\{L_{h_{n,2}}(y-Y_i)\, K_{h_{n,1}}(x-X_i) - E[L_{h_{n,2}}(y-Y)\, K_{h_{n,1}}(x-X)  ]\right\} \right |= O_{\mathbb P}\left(\sqrt {\frac{\abs{\log h_{n,1}h_{n,2}}}{n h_{n,1}h_{n,2} }} \right),\\
&\sup_{(x,y)\in \R^2}\left | n^{-1} \sum_{i=1}^n \left \{\big(  L_{h_{n,2}}(y-Y_{i} )  - E [L_{h_{n,2}}(y-Y_{i} )\mid X_i ]\big)\,   K_{h_{n,1}}(x-X_i)\right\}  \right |= O_{\mathbb P}\left(\sqrt {\frac{\abs{\log h_{n,1}h_{n,2}}}{n h_{n,1}h_{n,2} }} \right),\\
&\sup_{(x,y)\in \R^2}\left | n^{-1} \sum_{i=1}^n \left\{ \varphi_{h_{n,2}}(y,Y_{i} ) \, K_{h_{n,1}}(x-X_i) - E[\varphi_{h_{n,2}}(y,Y )\,  K_{h_{n,1}}(x-X)  ]\right\} \right |= O_{\mathbb P}\left(\sqrt {\frac{\abs{\log h_{n,1}} }{n {h_{n,1}} }} \right),\\
&\sup_{(x,y)\in \R^2}\left | n^{-1} \sum_{i=1}^n \left \{\big(   \varphi_{h_{n,2}}(y,Y_{i} )  - E [\varphi_{h_{n,2}}(y,Y_{i} )\mid X_i ]\big)\,   K_{h_{n,1}}(x-X_i)\right\}  \right | = O_{\mathbb P}\left(\sqrt {\frac{\abs{\log h_{n,1}} }{n {h_{n,1}} }} \right),
\end{align*}
where $K_h(\,\cdot\,)=h^{-1}K(\,\cdot\,/h)$, $L_h(\,\cdot\,) = h^{-1} L(\,\cdot\,/h)$, and $ \varphi_h(y,Y_{i} ) =\int _{-\infty} ^{y} L_{h}(t-Y_{i}) \, \diff t $.
\end{proposition}

\begin{proof}
We use the definition of bounded measurable VC classes given in \cite{gine+g:02}. That is, call a class $\mathcal F$ of measurable functions a bounded measurable VC class if $\mathcal F$ is separable or image admissible Suslin \citep[Section 5.3]{dudley1999} and if there exist $A>0$ and $v>0$ such that, for every probability measure $Q$, and every $0<\epsilon<1$,
\begin{align*}
\mathcal N \left(\mathcal F, L_2(Q), \epsilon\|F\|_{L_2(Q)}\right)\leq \left(\frac{A}{\epsilon}\right)^{v}, 
\end{align*}
where $F$ is an envelope for $\mathcal F$ and $\mathcal N(T,d,\epsilon)$ denotes the $\epsilon$-covering number of the metric space $(T,d)$ \citep{wellner1996}. [Other terminologies associated to the same concepts are \textit{Euclidean classes} \citep{nolan+p:1987} and \textit{Uniform entropy numbers} \citep{wellner1996}.] Using \cite{nolan+p:1987}, Lemma~22, Assertion~(ii), the classes  $\{z\mapsto L(h^{-1}(y-z))\ : \ y\in \R,\ h>0\}$ and $\{z\mapsto K(h^{-1}(x-z))\ : \ x\in \R,\ h>0\}$ are bounded measurable VC [see \cite{gine+g:02} for remarks and references on measurability issues associated to the previous classes]. 

The first statement is a direct application of Theorem~2.1, Equation~(2.2), in \cite{gine+g:02}, with, in their notation,
\begin{align*}
&\mathcal F_n = \left\{z\mapsto K\big(h_{n,1}^{-1}(x-z)\big)\ : \ x\in \R\right\},\\
&\sigma_n^2 = h_{n,1}\,\|f_X\|_\infty \int K(u)^2\, \diff u,\\
& U = \sup_{u\in [-1,1]}|K(u)|,
\end{align*}
where $f_X$ denotes the density of $X$ and $\|\cdot\|_\infty $ is the supremum norm. The class $\mathcal F_n$ is included in one of the bounded measurable VC classes given above. Moreover, we have
\begin{align*}
\Var\left(K\big(h_{n,1}^{-1}(x-X)\big)\right)&\leq \int K\big(h_{n,1}^{-1}(x-z)\big)^2\, f_X(z) \, \diff z\\
&=h_{n,1} \int K(u)^2\, f_X(x-h_{n,1}u) \, \diff u\\
&\leq \sigma_n^2.
\end{align*}
The cited equation yields that the expectation of the supremum over $x$ of the absolute value of $n^{-1} \sum_{i=1}^n \big\{ K_{h_{n,1}}(x-X_i) - E[ K_{h_{n,1}}(x-X)  ]\big\}$ is bounded by some constant times
\begin{align*}
(nh_{n,1})^{-1} \abs{\log (h_{n,1})} +  \left\{ {(nh_{n,1})^{-1} \abs{\log (h_{n,1})}} \right\}^{1/2} =O\left( \left\{ {(nh_{n,1})^{-1} \abs{\log (h_{n,1})}} \right\}^{1/2}  \right).
\end{align*}
Hence we have obtained the first assertion.

Preservation  properties for VC classes \citep[Lemma~2.6.18]{wellner1996} imply that the product class 
\begin{align*}
  \left\{
    (z_1,z_2) \mapsto L\bigl(h_2^{-1}(y-z_2)\bigr)\, K\bigl(h_1^{-1}(x-z_1)\bigr) \ : \ 
    x \in \R,\ y \in \R,\ h_1>0,\ h_2>0
  \right\}
\end{align*}
is bounded measurable VC. Then we can apply Theorem~2.1, Equation~(2.2), in \cite{gine+g:02} in a similar fashion as before. The main difference  lies in the variance term, which  follows from
 \begin{align*}
\lefteqn{\Var\left( L\big(h_{n,2}^{-1}({y-Y})\big)\, K_{}\big(h_{n,1}^{-1}({x-X})\big)\right)} \\
 &\leq  \int  L\big(h_{n,2}^{-1}({y-z_2})\big)^2\, K_{}\big(h_{n,1}^{-1}({x-z_1 })\big)^2\, f_{X,Y}(z_1, z_2)\, \diff (z_1, z_2)\\ 
 &= h_{n,1}h_{n,2}  \int L(u_2)^2\, K(u_1)^2 \, f_{X,Y}\big(x-h_{n,1}u_1, y-h_{n,2}u_2\big) \, \diff (u_1,u_2)\\
&\leq h_{n,1}h_{n,2} \, \|f_{X,Y}\|_\infty \int L(u_2)^2 K(u_1)^2\, \diff (u_1,u_2) =\sigma_n ^2   .
 \end{align*} 
 Computing the bound leads to the second assertion.

Using the second statement and the triangle inequality, the third assertion is obtained whenever
\begin{multline*}
  \sup_{(x,y)\in \R^2} 
  \abs{ 
    n^{-1} \sum_{i=1}^n 
    \left\{  
      E [L_{h_{n,2}}(y-Y_i) \mid X_i ]\, K_{h_{n,1}}(x-X_i) -
      E [L_{h_{n,2}}(y-Y)\, K_{h_{n,1}}(x-X)] 
    \right\} 
  } \\
= O_{\mathbb P}\left(\sqrt {\frac{\abs{\log h_{n,1}h_{n,2}}}{n h_{n,1}h_{n,2} }} \right).
\end{multline*}
If the class $\{z\mapsto E[L(h^{-1}({y-Y}))\mid X=z]\ :\ y\in \R,\ h>0\}$ is a bounded measurable VC class of functions, then the class $\{z\mapsto E [L(h_2^{-1}({y-Y})) \mid X=z]\, K(h_1^{-1}({x-z}))\ :\ x\in \R,\ y\in \R, \ h_1>0,\, h_2>0\}$ is still a bounded measurable VC class of functions \citep[Lemma 2.6.18]{wellner1996}. Consequently, we can apply Theorem 2.1, Equation (2.2) in \cite{gine+g:02}, with the same $\sigma_n^2$ as before, because, by Jensen's inequality,
\begin{align*}
  \Var \left(
    E \bigl[L\bigl(h_{n,2}^{-1}({y-Y})\bigr) \mid X \bigr] \,
    K\bigl(h_{n,1}^{-1}({x-X})\bigr)
  \right) 
  &\leq 
  E\left[
    E \bigl[ L\big(h_{n,2}^{-1}({y-Y})\big) \mid X \bigr]^2 \, 
    K\bigl(h_{n,1}^{-1}({x-X})\bigr)^2
  \right] \\
  &\leq 
  E \left[
    L\bigl(h_{n,2}^{-1}({y-Y})\bigr)^2\, K\bigl(h_{n,1}^{-1}({x-X})\bigr)^2 
  \right] \\
  &\leq \sigma_n^2.
\end{align*}
Now we show that $\mathcal L = \{z\mapsto E [L(h^{-1}({y-Y}))\mid X=z ]\ :\ y\in \R,\ h>0\}$ is a bounded measurable VC class of functions. Let $Q$ be a probability measure on $S_X$. Define $\tilde  Q$ as the probability measure given by
\begin{align*}
  \diff \tilde Q (y) = \int f(y|x)\, \diff Q(x) \, \diff y.
\end{align*}
Let $f_1,\ldots, f _N$ denote the  centers of an $\epsilon$-covering of the class $\mathcal L' = \{z\mapsto L(h^{-1}({y-z}))\ :\ y\in \R,\ h>0 \}$ with respect to the metric $L_2(\tilde Q) $. For a function $x\mapsto E[ f(Y) \mid X=x ]$, element of the space $ \mathcal F$, there exists $k\in \{1,\ldots,N\} $ such that, by Jensen's inequality and Fubini's theorem,
\begin{align*}
\int \left\{E\bigl[f (Y)\mid X=x\bigr]- E\bigl[f_k(Y)\mid X=x\bigr] \right\}^2 \, \diff Q(x)&\leq \int  E \bigl[ \bigl\{ f (Y)- f_k(Y)\bigr\} ^2\mid X=x\bigr] \, \diff Q(x)\\
&=\int \int \big \{ f(y)- f_k(y)\big\} ^2\, f(y|x) \, \diff y \, \diff Q(x)\\
&=  \int \big\{f(y) - f_k(y)\big \}^2 \, \diff \tilde Q(y)\leq \epsilon^2 .
\end{align*}
Consequently, 
\begin{align*}
  \mathcal N \bigl( \mathcal L, L_2(Q), \epsilon \bigr) 
  \leq 
  \mathcal N \bigl( \mathcal L', L_2(\tilde Q), \epsilon \bigr).
\end{align*}
Since the kernel $L$ is bounded, there exists a positive constant $c_\infty$ that is an envelope for both classes $\mathcal L$ and $\mathcal L'$. Note that $\|c_\infty \|_{L_2(Q)}=c_\infty$ for any probability measure $Q$. Using the fact that $\mathcal L'$ is bounded measurable VC (see the proof of the first assertion), it follows that
\begin{align*}
  \mathcal N \bigl( \mathcal L, L_2(Q), \epsilon c_\infty \bigr)
  \leq 
  \mathcal N \bigl( \mathcal L', L_2(\tilde Q), \epsilon c_\infty \bigr)
  \leq 
  \left(\frac{A}{\epsilon}\right)^{v},
\end{align*} 
i.e., the class $\mathcal L$ is bounded measurable VC. 
 
In the same spirit, the fourth assertion holds true whenever the class $ \{z\mapsto\varphi_h(y,z )\ :\ y\in \R,\ h>0\}$ is a bounded measurable VC class of functions. This is indeed true in virtue of Lemma~22, Assertion~(ii), in \cite{nolan+p:1987}, because each function $z\mapsto \varphi_h(y,z )$ of the class can be written as $z\mapsto \overline L( h^{-1}(y-z) ) $, where
\begin{align*}
\overline L(t)= \int_{-\infty}^t L(u)\, \diff u= \int_{-\infty}^t L^+(u)\, du+\int_{-\infty}^t L^-(u)\, \diff u,
\end{align*}
with $L=L^++L^- $, $L^+(u)\geq 0$, $L^-(u)<0$, which is indeed a bounded real function of bounded variation (as the sum of an increasing function and a decreasing function is of bounded variation). Applying Theorem~2.1, Equation~(2.2) in \cite{gine+g:02}, with $\sigma_n$ given by the upper bound in the inequality chain
\begin{align*}
\Var\left(\varphi_{h_{n,2}}(y,Y )\, K\big(h_{n,1}^{-1}(x-X)\big)\right) &\leq \int \varphi_{h_{n,2}}(y,z_2 )^2\, K\big(h_{n,1}^{-1}(x-z_1)\big)^2\, f_{X,Y}(z_1,z_2)\, \diff (z_1,z_2) \\
& \leq h_{n,1} \sup_{u\in [-1,1]} \{\overline L(u)^2\}\, \|f_X\|_\infty \int K(u)^2\, \diff u= \sigma_n^2 ,
\end{align*}
leads to the same rate as in the first assertion. 

Using the fourth statement and the triangle inequality, the fifth assertion is obtained whenever
\begin{multline*}
\sup_{(x,y)\in \R^2}\abs{ n^{-1} \sum_{i=1}^n \left\{  E [\varphi_{h_{n,2}}(y,Y_i )\mid X_i ] \,  K_{h_{n,1}}(x-X_i) -E [\varphi_{h_{n,2}}(y,Y ) \, K_{h_{n,1}}(x-X)] \right\}}
\\= O_{\mathbb P}\left(\sqrt {\frac{\abs{\log h_{n,1}} }{n {h_{n,1}} }} \right).
\end{multline*}
Following exactly the same lines as in the proof of the third statement [replacing $L(h^{-1}({y-Y}))$ by $\varphi_h(y,Y)$] we show that, since $ \{z\mapsto\varphi_h(y,z )\ :\ y\in \R,\ h>0\}$ is a bounded measurable VC class of functions (obtained in the proof of the fourth statement), the class $\{z\mapsto E[\varphi_h(y,Y)\mid X=z]\ :\ y\in \R,\ h>0 \}$ is a bounded measurable VC class of functions. Then the class $\{z\mapsto E [\varphi_{h_2}(y,Y ) \mid X=z]\, K(h_1^{-1}({x-z}))\ :\ x\in \R,\ y\in \R, \ h_1>0,\, h_2>0 \}$ is still a bounded measurable VC class of functions \citep[Lemma 2.6.18]{wellner1996}. We conclude by applying Theorem 2.1, Equation (2.2) in \cite{gine+g:02}, with the same $\sigma_n^2$ as before; note that Jensen's inequality gives  $E [\varphi_{h_{n,2}}(y,Y )\mid X=x ]^2\leq  E[\varphi_{h_{n,2}}(y,Y)^2\mid X=x]$. 
\end{proof}

\subsection{Proofs for Section~\ref{sec:loc_lin_results}}
\label{sec:proofs}

\subsubsection{Proof of Lemma \ref{lemma:positivedenominator}}

Recall that $\hat{p}_{n,k}(x) = n^{-1} \sum_{i=1}^n w_{k,h_{n,1}}(x - X_i)$ with $w_{k,h}(\,\cdot\,) = h^{-1} \, w_k(\,\cdot\,/h)$ and $w_{k} (u)= u^k \, K( u )$. We have
\[
  \expec\{\hat p_{n,k} (x) \} = \int f_X(x-h_{n,1}u) \, w_{k}(u)  \,\diff u.
\]
By Condition (G\ref{cond:kernel}), the function $K$ is a bounded function of bounded variation that vanishes outside $(-1,1)$. Applying Proposition~\ref{prop:concentration:kernel}, first statement, with kernel $u\mapsto w_{k} (u)$ and using Condition~(G\ref{cond:bandwidth}),  we have
\begin{equation*}
  \sup_{x\in \R} \abs{\hat p_{n,k}(x)-\mathbb   E\{\hat  p_{n,k}(x)\}}
  = o_{\mathbb P}(1), \qquad n \to \infty.
\end{equation*}
To obtain \eqref{eq:Pto1}, we can rely on the expectations of all  terms involved: it suffices to show that, for some $c>0$ and $n$ sufficiently large,  for all $x \in S_X$,
\begin{align}\label{eq:expec}
\expec \{ \hat p_{n,0}(x)\} \expec \{\hat p_{n,2}(x)\} - [\expec \{\hat p_{n,1}(x)\}]^2 \geq 2b^2c.
\end{align}
We have
\begin{align*}
  \expec \{ \hat p_{n,0}(x)\} \expec \{\hat p_{n,2}(x)\} - \expec \{\hat p_{n,1}(x)\}^2 
  &= \expec \{\hat  p_{n,0}(x)\} \int  f_X(x-h_{n,1}u)\, \bigl\{ u- a_n(x) \bigr\}^2\, K(u)\, \diff u,
\end{align*}
with $a_n(x) = \expec\{\hat  p_{n,1}(x)\} / \expec \{\hat  p_{n,0}(x)\}$. On the one hand, by (G\ref{cond:kernel}),  as $n$ is sufficiently large, we have, for all $x \in S_X$,
\begin{align*}
\expec\{\hat p_{n,0}(x)\} &= \int  f_X(x-h_{n,1}u)\, K(u)\, \diff u\\
&\geq b  \int _{x-h_{n,1}u\, \in S_X} K(u)\, \diff u\\
&\geq b\min\left( \int_{-1}^0 K(u) \, \diff u, \int_{0}^1 K(u) \, \diff u\right) 
= b/2,
\end{align*}
where the last equality is due to the symmetry of the function $K$. On the other hand, for $n$ large enough,  for all $x \in S_X$,
\begin{align*}
  \lefteqn{
    \int f_X(x-h_{n,1}u) \, \bigl\{ u- a_n(x) \bigr\}^2\, K(u) \, \diff u 
  } \\
  &\geq 
  b \int _{x-h_{n,1}u\, \in S_X} \bigl\{ u-  a_n(x) \bigr\}^2 \, K(u)\, \diff u \\
  &\geq 
  \min \left[ 
    \int_{-1}^0 \bigl\{ u-  a_n(x) \bigr\}^2 \, K(u) \, \diff u, \;
    \int_{0}^1 \bigl\{ u-  a_n(x) \bigr\}^2 \, K(u) \, \diff u
  \right] \\
  &=
  \min \left[
    \int_{0}^1 \bigl\{ u +  a_n(x) \bigr\}^2 \, K(u) \, \diff u, \;
    \int_{0}^1 \bigl\{ u -  a_n(x) \bigr\}^2 \, K(u) \, \diff u
  \right] \\
  &\geq  
  b \inf_{a \in \R} \int_0^1 \bigl\{ u -  a \bigr\}^2 \, K(u) \, \diff u 
  = 
  \frac{b\, c_K}{2},
\end{align*}
where $c_K = 2\int_0^1 (u-a_K)^2 \, K(u)\, \diff u $ and $a_K = 2\int_0^1 u \, K(u)\, \diff u$. By taking $c = c_K / 8$ we obtain \eqref{eq:expec}.  We have shown \eqref{eq:Pto1}.

By differentiating, we see that \eqref{eq:app:optimisationLL} is equivalent to a linear system whose determinant is equal to $   \hat p_{n,0}(x) \, \hat p_{n,2}(x)-  \hat p_{n,1}(x)^2  $. Because the latter is strictly positive with probability going to $1$, we can invert the linear system to obtain \eqref{formula:ll}.
\qed

\subsubsection{Proof of Proposition \ref{prop:uniformCVderivatives}}

Using Lemma~\ref{lemma:positivedenominator}, and since we are concerned with convergence in probability, we can restrict attention to the event  that $\hat p_{n,0}(x) \, \hat p_{n,2}(x)-\hat p_{n,1}(x)^2\geq b^2 c$ for every $x\in S_X$. Recall $\hat{Q}_{n,k}(y, x) = n^{-1} \sum_{i=1}^n \varphi_{h_{n,2}}(y, Y_i) \, w_{k,h_{n,1}}(x - X_i)$, where $\varphi_h$ is defined in \eqref{def:varphi} and $w_{k,h}$ is defined right before Lemma~\ref{lemma:positivedenominator}. It follows that
\begin{align*}
\hat F_n(y|x) &= \frac{\hat Q_{n,0}(y,x)\, \hat p_{n,2}(x) - \hat Q_{n,1}(y,x)\, \hat p_{n,1}(x)}{\hat p_{n,0}(x)\, \hat p_{n,2}(x)-\hat p_{n,1}(x)^2}\\
&= {\hat Q_{n,0}(y,x)\, \hat s_{n,2}(x) - \hat Q_{n,1}(y,x)\, \hat s_{n,1}(x)},
\end{align*}
with $\hat s_{n,k}(x)= \big(\hat p_{n,0}(x) \, \hat p_{n,2}(x)-\hat p_{n,1}(x)^2\big)^{-1} \hat p_{n,k}(x)$. Write
\[
  \hat Q_{n,k}(y,x)
  =
  \hat \beta_{n,k}(y,x) +  \hat b_{n,k}(y,x) + \hat \gamma_{n,k}(y,x),
\]
with 
\begin{align*}
&\hat \beta_{n,k}(y,x)= n^{-1} \sum_{i=1}^n\big\{ \varphi_{h_{n,2}}(y,Y_i ) -   E [\varphi_{h_{n,2}}(y,Y_{i} ) \mid X_i ] \big\} \,  w_{k,h_{n,1}}(x-X_i) ,\\
&\hat b_{n,k}(y,x) = n^{-1} \sum_{i=1}^n\big \{    E [\varphi_{h_{n,2}}(y,Y_{i} ) \mid X_i ] - F(y|x) -(X_i-x)\, \partial _x F(y|x) \big\} \,  w_{k,h_{n,1}}(x-X_i),\\
&\hat \gamma_{n,k}(y,x) = F(y|x)\, \hat p_{n,k}(x) - h_{n,1}\, \hat p_{n,k+1}(x)\, \partial_x F(y|x).
\end{align*} 
Elementary algebra yields
\begin{align*}
\hat \gamma_{n,0}(y,x)\, \hat s_{n,2}(x) - \hat \gamma_{n,1}(x)\,\hat s_{n,1}(x)= F(y|x).
\end{align*}
It follows that
\begin{align}
\hat F_n(y|x)  -  F(y|x)
&=  \{\hat \beta_{n,0}(y,x)+\hat b_{n,0}(y,x)\}\, \hat s_{n,2}(x) -\{\hat \beta_{n,1}(y,x)+\hat b_{n,1}(y,x)\}\, \hat s_{n,1}(x). \label{equation:decompuniform}
\end{align}
\longproof%
{%
The remainder of the proof now consists of multiple applications of Proposition~\ref{prop:concentration:kernel} to establish uniform rates of convergence for the derivatives of order  $l \in \{0, 1, 2\}$ of $\hat{\beta}_{n,k}$ and $\hat{b}_{n,k}$, for $k \in \{0, 1\}$, and of $\hat{s}_{n,k}$, for $k \in \{1, 2\}$. To do so, we rely on conditions (G1), (G3) and (G4). We need to distinguish between many situations because differentiating with respect to $x$ or $y$ does not impact the rates of convergence in the same way. In contrast, the index $k$ has no effect. The details are tedious and are given in full length in the arXiv version of the paper \citep{portier:segers:2017}. In view of \eqref{equation:decompuniform}, the rates of convergence for $\hat{F}_n(y|x)$ then follow. \qed
}%
{%
In what follows, we apply Proposition~\ref{prop:concentration:kernel} to establish uniform rates of convergence for the derivatives of order  $l \in \{0, 1, 2\}$ of $\hat{\beta}_{n,k}$ and $\hat{b}_{n,k}$, for $k \in \{0, 1\}$, and of $\hat{s}_{n,k}$, for $k \in \{1, 2\}$. We need to distinguish between many situations because differentiating with respect to $x$ or $y$ does not impact the rates of convergence in the same way. In contrast, the index $k$ has no effect. Consequently, in what follows we fix $k \in \{0, 1, 2\}$. We define 
\begin{align*}
  \hat r^{(l)}_{n,k}(x) &= n^{-1}\sum_{i=1}^n \abs{ w_{k,h_{n,1}}^{(l)} (x-X_i) }, &
  w_{k,h}^{(l)}(\,\cdot\,) &= h^{-1} \, w_{k}^{(l)}(\,\cdot\,/h),
 \end{align*} 
where $w_k^{(l)} : u\mapsto \partial_u^l w_k(u)$ for $l \in \{0,1,2\}$. All asymptotic statements are for $n \to \infty$, which is omitted for brevity.

\paragraph{Rate of $\hat r_{n,k}^{(l)}$ for $l\in \{0,1,2\}$.}
Invoking Condition~(G\ref{cond:kernel}), the functions $K$, $K'$ and $K''$ are bounded real functions of bounded variation that vanish outside $(-1,1)$. Consequently, the functions $w_k^{(l)}$, $l\in \{0,1,2\}$, are bounded functions of bounded variation that vanish outside $(-1,1)$. Note that the absolute value of a function of bounded variation is of bounded variation too, in view of the fact that $\bigl| \abs{a} - \abs{b} \bigr| \le \abs{a - b}$ for all $a, b \in \reals$. By Proposition~\ref{prop:concentration:kernel}, first assertion, with $K$ equal to $u\mapsto |w_k^{(l)}(u)|$, we therefore have
\begin{align*}
\sup_{x\in S_X} \left| \hat r^{(l)}_{n,k}(x) -\expec\{ \hat r^{(l)}_{n,k}(x)\}\right| = O_{\mathbb P}\left(\sqrt {\frac{\abs{\log h_{n,1}}}{n h_{n,1}^{} }} \right) . 
\end{align*}
By (G\ref{cond:smoothnessdensity1}) it holds that
\begin{align*}
\expec\{ \hat r^{(l)}_{n,k}(x)\} = \int f_X(x-h_{n,1}u) \, | w_{k}^{(l)} (u)|\,\diff u \leq M \int  | w_{k}^{(l)} (u)|\,\diff u.
\end{align*}
By the triangle inequality and Condition~(G\ref{cond:bandwidth}), it follows that, for $l\in \{0,1,2\}$,
\begin{align}
\label{rate:pnk2}
  \sup_{x \in S_X} \hat r^{(l)}_{n,k}(x)
  = 
  O_{\mathbb P} \left(
    \sqrt{ \frac{\abs{\log h_{n,1}}}{n h_{n,1}} } + 1 
  \right)
  =
  O_{\mathbb P} ( 1 ).
\end{align}

\paragraph{Rate of $ \partial^{l}_x \hat p_{n,k}$ for $l \in \{0, 1, 2\}$.}
Differentiating under the expectation [apply the dominated convergence theorem invoking (G\ref{cond:kernel})], we get for any $l \in \{0,1,2\}$,
\begin{align}
\label{eq:diffexpectation1}
  \partial_x^{l}\{\expec\hat p_{n,k}(x)\} & = \partial_x^{l}\left\{\int f_X(z)\,  w_{k,h_{n,1}}(x-z)\,\diff z\right\} = h_{n,1}^{-l} \int f_X(z) \, w_{k,h}^{(l)}(x-z)\,\diff z.
\end{align}
Similarly  [apply the dominated convergence theorem invoking~(G\ref{cond:smoothnessdensity1})], we have 
\begin{align}\label{eq:diffexpectation2}
\partial_x^{l}\{\expec\hat p_{n,k}(x)\} &= \partial_x^{l}\left\{\int f_X(x-h_{n,1}u) \, w_{k}(u)\,\diff u\right\} = \int f_X^{(l)}(x-h_{n,1}u)\,  w_{k}(u)\,\diff u.
\end{align}
As a consequence of~\eqref{eq:diffexpectation1}, we have 
\begin{align*}
\partial_x^{l}\{\hat p_{n,k}(x)-\expec\hat p_{n,k}(x) \} &= (nh_{n,1}^l)^{-1}\sum_{i=1}^n \left\{ w_{k,h_{n,1}}^{(l)}(x-X_i) -E[w_{k,h_{n,1}}^{(l)}(x-X)] \right\},
\end{align*}
and in view of Proposition \ref{prop:concentration:kernel}, first assertion, with $K$ equal to $u\mapsto w_k^{(l)}(u)$, it holds that
\begin{align*}
\sup_{x\in S_X} \abs{ \partial_x^{l}\{\hat p_{n,k}(x)-\expec\hat p_{n,k}(x) \}} &=O_{\mathbb P} \left(\sqrt {\frac{\abs{\log h_{n,1}}}{n h_{n,1}^{1+2l} }}\right).
\end{align*}
 Using \eqref{eq:diffexpectation2} and invoking  (G\ref{cond:smoothnessdensity1}), we have
\begin{align*}
 \abs{\partial_x^{l} \expec\{\hat p_{n,k}(x) \}} &\leq M \int \abs{w_{k}(u)} \,\diff u.
\end{align*} 
By the triangle inequality, it follows that, for $l\in \{0,1,2\}$,
 \begin{align}\label{rate:pnk}
\sup_{x\in S_X} \abs{\partial_x^{l}\hat p_{n,k}(x)}=O_{\mathbb P} \left(\sqrt {\frac{\abs{\log h_{n,1}}}{n h_{n,1}^{1+2l} }}+1 \right).
 \end{align}

\paragraph{Rate of $ \partial^{l}_x \hat s_{n,k}$ for $l \in \{0, 1, 2\}$.} 

We use the quotient rule for derivatives to obtain that
\begin{align*}
\sup_{x\in S_X} \abs{\partial_x^{l} \hat s_{n,k}(x)} = O_{\mathbb P}\left(\sqrt {\frac{\abs{\log h_{n,1}}}{n h_{n,1}^{1+2l} }} +1\right).
\end{align*}
For $l=0$, the previous formula follows from equation (\ref{rate:pnk}) because $\hat p_{n,0}(x) \, \hat p_{n,2}(x)-\hat p_{n,1}(x)^2\geq bc$ for every $x\in S_X$. For $l=1$, differentiating $\hat s_{n,k}(x)$ gives a sum of terms which are all of the form
\begin{align*}
\frac{ \partial^{}_x  \{\hat p_{n,k_1}(x)\} \, \hat p_{n,k_2}(x)\,  \hat p_{n,k_3}(x)}{\big(\hat p_{n,0}(x)\, \hat p_{n,2}(x)-\hat p_{n,1}(x)^2\big)^2  },
\end{align*}
where $k_1$, $k_2$, $k_3 \in \{0, 1, 2\}$. By~\eqref{rate:pnk}, each term is of the order  $O_{\mathbb P}(  \sqrt {{\abs{\log h_{n,1}}}/(n h_{n,1}^{3})}+1)$. Finally, when $l=2$, the order is given by differentiating the previous expression. We obtain a sum of different terms. Putting $\hat d_n(x) =\hat p_{n,0}(x)\hat p_{n,2}(x)-\hat p_{n,1}(x)^2$, these terms are of the form
 \begin{align*}
&\frac{ \partial^{l_1}_x  \{\hat p_{n,k_1}(x)\}\, \partial^{l_2}_x \{\hat p_{n,k_2}(x)\}\,  \hat p_{n,k_3}(x)}{\hat d_n(x)^2  }, \quad \text{with $l_1+l_2 = 2$},\\
\text{and}\qquad &\frac{\partial^{}_x\{\hat d_n(x)\}\, \partial^{}_x  \{\hat p_{n,k_1}(x)\} \, \hat p_{n,k_2}(x)\,  \hat p_{n,k_3}(x)   }{\hat d_n(x)^3  }.
\end{align*}
 By \eqref{rate:pnk}, the term with the highest order is 
\[ 
  \frac%
  { \partial^{2}_x  \{\hat p_{n,k_1}(x)\}\, \hat p_{n,k_2}(x)\, \hat p_{n,k_3}(x) }%
  {\hat d_n(x)^2  } 
  = O_{\mathbb P} \left(\sqrt {\frac{\abs{\log h_{n,1}}}{n h_{n,1}^{5}}} + 1 \right). 
\]

\paragraph{Rate of $ \partial^{l}_x \hat \beta_{n,k}$ for $l \in \{0, 1, 2\}$.}

Proposition~\ref{prop:concentration:kernel}, fifth assertion, with $K$ equal to $w_k^{(l)}$,  yields
\begin{align*}
\sup_{x\in S_X,\, y\in \R} \abs{\partial^{l}_x \hat \beta_{n,k}(y,x)} = O_{\mathbb P}\left(\sqrt {\frac{\abs{\log h_{n,1}}}{n h_{n,1}^{1+2l} }} \right) .
\end{align*}

\paragraph{Rate of $ \partial^{l}_x \hat b_{n,k}$ for $l \in \{0, 1, 2\}$.} 

Let $F_{n}(y|x) = \int F(y-h_{n,2}u|x)L(u)\, \diff u $. Fubini's theorem gives that $   E [\varphi_{h_{n,2}}(y,Y_{i} )\mid X_i=x ] =  F_{n}(y|x)$. Hence we consider
 \begin{align*}
\hat b_{n,k}(y,x) = \hat B_{n,1}^{(0)}(y,x) + \hat B_{n,2}^{(0)}(y,x),
\end{align*}
with 
 \begin{align*}
&\hat B_{n,1}^{(0)}(y,x) = n^{-1} \sum_{i=1}^n \{ F_{n}(y|X_i) - F(y|X_i) \}\,  w_{k,h_{n,1}}(x-X_i),\\
 & \hat B_{n,2}^{(0)}(y,x) =  n^{-1} \sum_{i=1}^n \{  F(y|X_i) -   F(y|x) -(X_i-x)\,  \partial _{x} F(y|x)    \} \,  w_{k,h_{n,1}}(x-X_i).
\end{align*}
Differentiating once with respect to $x$, we find
\begin{align*}
\partial _{x} \hat b_{n,k}(y,x) = \hat B_{n,1}^{(1)}(y,x) + \hat B_{n,2}^{(1)}(y,x)+\hat B_{n,3}^{(1)}(y,x) ,
\end{align*}
with 
 \begin{align*}
\hat B_{n,1}^{(1)}(y,x) &= n^{-1} \sum_{i=1}^n \{ F_{n}(y|X_i) - F(y|X_i) \}\,   \partial _x \{w_{k,h_{n,1}}(x-X_i)\},\\
\hat B_{n,2}^{(1)}(y,x) &=  n^{-1} \sum_{i=1}^n \{  F(y|X_i) -   F(y|x) - (X_i-x)\, \partial _{x} F(y|x)   \} \, \partial _x \{ w_{k,h_{n,1}}(x-X_i)\},\\
\hat B_{n,3}^{(1)}(y,x) &= h _{n,1}\, \partial _{x}^{2} F(y|x) \,      \hat p_{n,k+1}(x) .
\end{align*}
Differentiating twice with respect to $x$, we find
\begin{align*}
\partial _{x}^{2} \hat b_{n,k}(y,x) = \hat B_{n,1}^{(2)}(y,x) + \hat B_{n,2}^{(2)}(y,x)+\hat B_{n,3}^{(2)}(y,x)+\hat B_{n,4}^{(2)}(y,x) ,
\end{align*}
with 
\begin{align*}
  \hat B_{n,1}^{(2)}(y,x) 
  &= 
  n^{-1} \sum_{i=1}^n 
    \{ F_{n}(y|X_i) - F(y|X_i) \}\, 
    \partial _x^2 \{w_{k,h_{n,1}}(x-X_i)\}, 
  \\
  \hat B_{n,2}^{(2)}(y,x) 
  &= 
  n^{-1} \sum_{i=1}^n 
    \{  F(y|X_i) -   F(y|x) - (X_i-x)\,\partial _{x} F(y|x) \} \, 
    \partial _x^2 \{ w_{k,h_{n,1}}(x-X_i)\},
  \\
  \hat B_{n,3}^{(2)}(y,x) 
  &= 
  h_{n,1} \, \hat p_{n,k+1}(x) \, \partial _{x}^{3}  F(y|x),
  \\
  \hat B_{n,4}^{(2)}(y,x) 
  &= 
  h_{n,1}\, \partial _{x}^2 F(y|x)\, 
  \left( 
    n^{-1} \sum_{i=1}^n \{    (x-X_i)/h_{n,1}   \} \, \partial _x \{ w_{k,h_{n,1}}(x-X_i)\}
  \right) 
  \\
  & \qquad \mbox{} +  
 h_{n,1} \, \partial _x \{ \hat p_{n,k+1}(x) \} \, \partial _{x}^{2} F(y|x).
\end{align*}
All this results in the formula
\begin{align*}
\partial _{x}^{l} \hat b_{n,k}(y,x) = \hat B_{n,1}^{(l)}(y,x) + \hat B_{n,2}^{(l)}(y,x)+\hat B_{n,3}^{(l)}(y,x)+\hat B_{n,4}^{(l)}(y,x) ,
\end{align*}
with, for $l \in \{0,1,2\}$,
 \begin{align*}
\hat B_{n,1}^{(l)}(y,x) &= n^{-1} \sum_{i=1}^n \{ F_{n}(y|X_i) - F(y|X_i) \}\,   \partial _x^l \{w_{k,h_{n,1}}(x-X_i)\},\\
\hat B_{n,2}^{(l)}(y,x) &=  n^{-1} \sum_{i=1}^n \{  F(y|X_i) -   F(y|x) -(X_i-x)\, \partial _{x} F(y|x)    \} \, \partial _x^l \{ w_{k,h_{n,1}}(x-X_i)\},
\end{align*}
whereas $ \hat B_{n,3}^{(0)}(y,x) = 0$ and $\hat B_{n,3}^{(l)}(y,x) =h_{n,1}  \,    \hat p_{n,k+1}(x)\, \partial _{x}^{l+1} F(y|x)$ for $l \in \{1, 2\}$,  and $
 \hat B_{n,4}^{(0)}(y,x)= \hat B_{n,4}^{(1)}(y,x)  = 0$. Assumption~(G\ref{cond:smoothnessdensity1}) implies that
\begin{align*}
  \abs{ F(y+u|x)-F(y|x) - u\, \partial _y F(y|x) } \leq M \frac {u^{2}}{2}.
\end{align*}
By (G\ref{cond:kernel}) and \eqref{rate:pnk2}, since $\int u \, L(u) \, \diff u=0$, we have
 \begin{align*}
\lefteqn{
\sup_{x\in S_X,\, y\in\R}  \abs{ \hat B_{n,1}^{(l)}(y,x) }
} \\
&\leq  h_{n,1}^{-l}\sup_{x\in S_X} \{ \hat r_{n,k}^{(l)}(x)\} \sup_{x\in S_X,\, y\in \R} \left| F_{n}(y|x) - F(y|x) \right|   \\ 
 &= h_{n,1}^{-l} \sup_{x\in S_X} \{ \hat r_{n,k}^{(l)}(x)\} \sup_{x\in S_X,\, y\in \R} \left|  \int  \{F(y-h_{n,2}u|x) -F(y|x) + h_{n,2} u\, \partial _y F(y|x)\}\,  L(u)\, \diff u \right|  \\
 &\leq    h_{n,1}^{-l}h_{n,2}^2  \sup_{x\in S_X} \{ \hat r_{n,k}^{(l)}(x)\}\,\frac {M}{2} \int u^2L(u)\, \diff u\\
 &= O_{\mathbb P} \left(h_{n,1}^{-l}h_{n,2}^2\sqrt {\frac{\abs{\log  h_{n,1}}}{n  h_{n,1}^{} }}+h_{n,1}^{-l}h_{n,2}^2\right)\\
&= O_{\mathbb P} \left(h_{n,1}^{-l}h_{n,2}^2 \right) .
 \end{align*} 
By~(G\ref{cond:smoothnessdensity1}), one has  
\begin{align*}
 \abs{F(y|x+u)-F(y|x) - u\, \partial _x F(y|x)} \leq M \frac {u^{2}}{2}.
\end{align*}  
Using~\eqref{rate:pnk2}, it follows that
 \begin{align*}
\lefteqn{
\sup_{x\in S_X,\, y\in\R}  \abs{ \hat B_{n,2}^{(l)}(y,x) }
} \\
&\leq h_{n,1}^{-l} \sup_{x\in S_X} \{ \hat r_{n,k}^{(l)}(x)\}   \sup_{ \abs{x-\tilde x}<h_{n,1},\, y\in \R} \abs{ F(y|\tilde x) - F(y|x) -(\tilde x -x)\, \partial _{x} F(y|x) }   \\
 &= O_{\mathbb P} \left(h_{n,1}^{2-l}\sqrt {\frac{\abs{\log h_{n,1}}}{n h_{n,1}^{} }}+h_{n,1}^{2-l} \right)\\
 &= O_{\mathbb P} \left(h_{n,1}^{2-l} \right) .
 \end{align*}
Next, by \eqref{rate:pnk}, we have $\hat B_{n,3}^{(l)}(y,x) = O_{\mathbb P} (h_{n,1} \sqrt{\abs{\log h_{n,1}}/(n h_{n,1})}+h_{n,1}) = O_{\mathbb P} (h_{n,1})$ for $l \in \{1,2\}$. Finally, because $\partial_u\{u^{k+1}K(u)\} = u^k K(u)+ u \, \partial _u \{u^{k} K(u)\} $, we have
 \begin{align*}
 B_{n,4}^{(2)}(y,x) &= h_{n,1} \partial _{x}^2\{ F(y|x)\} \left( n^{-1} \sum_{i=1}^n \{    (x-X_i)/h_{n,1}   \} \ \partial _x \{ w_{k,h_{n,1}}(x-X_i)\}+       \partial _x \{\hat p_{n, k+1}(x)\}  \right)   \\ 
 &= h_{n,1} \partial _{x}^2 \{ F(y|x)\}\, \bigl( 2 \partial _x \{\hat p_{n, k+1}(x)\}   -\hat p_{n, k}(x) \bigr) \\
 &=O_{\mathbb P} \left(h_{n,1}\sqrt {\frac{\abs{\log h_{n,1}}}{n h_{n,1}^{3} }}+h_{n,1} \right)\\
 &= O_{\mathbb P} (h_{n,1}). 
 \end{align*} 
Putting all this together yields 
\begin{align*}
\sup_{x\in S_X,\, y\in \R} \abs{ \partial _x^{l}\hat b_{n,k}(y,x)} = O_{\mathbb P} \left(h_{n,1}^{2-l} +h_{n,1}^{-l}h_{n,2}^2 \right),\qquad l \in \{0, 1, 2\}.
\end{align*}

\paragraph{Rate of $\partial^{l_1}_y \partial^{l_2}_x \hat \beta_{n,k}$ for $(l_1,l_2)\in  \{ (1,0), \, (1,1),\, (2,0) \}$.} 

Start by differentiating under the expectation to obtain
\begin{align*}
\partial^{l_1}_y \partial^{l_2}_x \hat \beta_{n,k} &=  n^{-1} \sum_{i=1}^n \left( \partial^{l_1}_y\{\varphi_{h_{n,2}}(y,Y_i )\} -  E [\partial^{l_1}_y\{\varphi_{h_{n,2}}(y,Y_{i} )\} \mid X_i ] \right) \,  \partial^{l_2}_x\{ w_{k,h_{n,1}}(x-X_i) \}\\
&= (nh_{n,1}^{l_2}h_{n,2}^{l_1-1})^{-1} \sum_{i=1}^n \left( L_{h_{n,2}}^{(l_1-1)}(y-Y_i ) -  E [L_{h_{n,2}}^{(l_1-1)}(y-Y_{i} ) \mid X_i ] \right) \,   w_{k,h_{n,1}}^{(l_2)}(x-X_i) ,
\end{align*}
with $L_{h}^{(l)}(u) = h^{-1}L^{(l)}(u/h) $. By Condition~(G\ref{cond:kernel}), the functions $L$ and $L^{(1)}$ are bounded functions of bounded variation. Applying Proposition~\ref{prop:concentration:kernel}, third assertion, with $L$ equal to $ L^{(l_1-1)} $ and $K$ equal to $w_{k}^{(l_2)}$, for $(l_1,l_2)\in  \{ (1,0), \, (1,1),\, (2,0) \}$, gives
\begin{align*}
\sup_{x\in S_X,\, y\in \R} \abs{ \partial^{l_1}_y \partial^{l_2}_x \hat \beta_{n,k}(y,x) } = O_{\mathbb P}\left(\sqrt {\frac{\abs{\log h_{n,1}h_{n,2}}}{n h_{n,1}^{1+2 l_2 }h_{n,2}^{1+2(l_1-1) } }} \right).
\end{align*}

\paragraph{Rate of $\partial^{l_1}_y \partial^{l_2}_x \hat b_{n,k}$ for $(l_1,l_2)\in  \{ (1,0), \, (1,1),\, (2,0) \}$.} 
We mimic here the approach taken when treating $\partial^{l}_x \hat b_{n,k}$. In the following, terms denoted by $\hat B_{n}^{(l_1,l_2)}(y,x)$ are related to the derivatives of $\hat b_{n,k}$ of order $l_1$ (resp.\ $l_2$) with respect to $y$ (resp.\ $x$).  Differentiating $l_1$ times with respect to $y$ produces
\begin{align*}
 \partial _{y}^{l_1}  \hat b_{n,k}(y,x) = \hat B_{n,1}^{(l_1,0)}(y,x) + \hat B_{n,2}^{(l_1,0)}(y,x),
\end{align*}
with
 \begin{align*}
&\hat B_{n,1}^{(l_1,0)}(y,x) = n^{-1} \sum_{i=1}^n \big\{ \partial _y^{l_1}F_{n}(y|X_i) - \partial _y^{l_1} F(y|X_i) \big\}\, w_{k,h_{n,1}}(x-X_i),\\
 & \hat B_{n,2}^{(l_1,0)}(y,x) =  n^{-1} \sum_{i=1}^n \big\{  \partial _y^{l_1} F(y|X_i) - \partial _y^{l_1}  F(y|x) - (X_i-x)  \, \partial _y^{l_1} \partial _{x} F(y|x)  \big\} \, w_{k,h_{n,1}}(x-X_i).
\end{align*}
Differentiating  with respect to $x$ gives
\begin{align*}
\partial _{x}^{}  \partial _{y}^{}  \hat b_{n,k}(y,x) = \hat B_{n,1}^{(1,1)}(y,x) + \hat B_{n,2}^{(1,1)}(y,x)+ \hat B_{n,3}^{(1,1)}(y,x),
\end{align*}
with
 \begin{align*}
 \hat B_{n,1}^{(1,1)}(y,x) &= n^{-1} \sum_{i=1}^n \big\{ \partial _y^{}F_{n}(y|X_i) - \partial _y^{} F(y|X_i) \big\}\, \partial _{x}^{}  \{ w_{k,h_{n,1}}(x-X_i)\},\\
 \hat B_{n,2}^{(1,1)}(y,x) &=  n^{-1} \sum_{i=1}^n \big\{  \partial _y^{} F(y|X_i) - \partial _y^{}  F(y|x) -(X_i-x)\, \partial _y^{} \partial _{x} F(y|x)    \big\}\,   \partial _{x}^{}  \{  w_{k,h_{n,1}}(x-X_i)\},\\
 \hat B_{n,3}^{(1,1)}(y) &= h_{n,1}  \,  \hat p_{n,k+1}(x) \, \partial _y \partial _{x}^{2} F(y|x).
\end{align*}
By (G\ref{cond:smoothnessdensity1}) we have
\begin{align*}
\abs{\partial_y^{} F(y+u|x)-\partial_y^{}F(y|x) - u\, \partial _y^2 F(y|x)} \leq \frac {Mu^{2}}{2}.
\end{align*}  
By (G\ref{cond:kernel}) and \eqref{rate:pnk2}, since $\int u \, L(u) \, \diff u = 0$, we have
 \begin{align*}
\lefteqn{\sup_{x\in S_X,\, y\in \R} \abs{\hat B_{n,1}^{(1,l_2)}(y,x)}} \\
&\leq h_{n,1}^{-l_2} \sup_{x\in S_X} \{ \hat r_{n,k}^{(l_2)}(x)\} \, \sup_{x\in S_X,\, y\in \R} \left | \partial _y^{}F_{n}(y|x) - \partial _y^{} F(y|x) \right|   \\ 
 &= h_{n,1}^{-l_2}  \sup_{x\in S_X} \{ \hat r_{n,k}^{(l_2)}(x)\}  \sup_{x\in S_X,\, y\in \R} \left|  \int  \{\partial _y^{}F(y-h_{n,2}u|x) -\partial _y^{}F(y|x) + h_{n,2}u \, \partial _y^{2} F(y|x)\}\,  L(u)\, \diff u\right|  \\
 &\leq  h_{n,2}^{2}h_{n,1}^{-l_2}  \sup_{x\in S_X} \{ \hat r_{n,k}^{(l_2)}(x)\}  \, \frac {M}{2} \int u^2L(u)\, \diff u\\
 &=O_{\mathbb P} \left(h_{n,2}^{2}h_{n,1}^{-l_2} \sqrt {\frac{\abs{\log h_{n,1}}}{n h_{n,1}^{} }}+h_{n,2}^{2}h_{n,1}^{-l_2} \right)\\
 &= O_{\mathbb P} \left(h_{n,2}^{2} h_{n,1}^{-l_2}\right).
 \end{align*}
By (G\ref{cond:smoothnessdensity1}) we have
\begin{align*}
&\abs{\partial_y^{2} F(y+u|x)-\partial_y^{2}F(y|x) -u\, \partial_y^{3}F(y|x)} \leq  M \frac{\abs{u}^{1+\delta}}{1+\delta}.
\end{align*}  
Then, by \eqref{rate:pnk2}, it holds that
 \begin{align*}
\lefteqn{\sup_{x\in S_X,\, y\in \R}   \abs{ \hat B_{n,1}^{(2,0)}(y,x) }}\\
&\leq  \sup_{x\in S_X} \{ \hat r_{n,k}^{}(x)\}   \sup_{x\in S_X,\, y\in \R} \left| \partial _y^{2}F_{n}(y|x) - \partial _y^{2} F(y|x) \right|   \\ 
 &=  \sup_{x\in S_X} \{ \hat r_{n,k}^{}(x)\}  \sup_{x\in S_X,\, y\in \R} \left|  \int  \{\partial _y^{2}F(y-h_{n,2}u|x) -\partial _y^{2}F(y|x) +h_{n,2}u\, \partial_y^{3}F(y|x) \} \,  L(u)\, \diff u\right|  \\
 &\leq  h_{n,2}^{1+\delta} \sup_{x\in S_X} \{ \hat r_{n,k}^{}(x)\}  \,  \frac M {1+\delta} \int |u|^{1+\delta} L(u)\, \diff u\\
 &=O_{\mathbb P} \left(h_{n,2}^{1+\delta} \sqrt {\frac{\abs{\log h_{n,1}}}{n h_{n,1}^{} }}+h_{n,2}^{1+\delta} \right)\\
  &=O_{\mathbb P} \left(h_{n,2}^{1+\delta} \right).
 \end{align*}
By (G\ref{cond:smoothnessdensity1}), one has  
\begin{align*}
&\big|\partial _y F(y|x+u)-\partial _y  F(y|x) - u\, \partial _y  \partial _x F(y|x)\big| \leq M \frac {u^{2}}{2}.
\end{align*}  
Using  \eqref{rate:pnk2}, it follows that
 \begin{align*}
\lefteqn{\sup_{x\in S_X,\, y\in\R}  \abs{\hat B_{n,2}^{(1,l_2)}(y,x) }}\\
& \leq h_{n,1}^{-l_2}  \sup_{x\in S_X} \{ \hat r_{n,k}^{(l_2)}(x)\}  \sup_{ |x-\tilde x|<h_{n,1},\, y\in \R} \abs{ \partial _y F(y|\tilde x) - \partial _y F(y|x) -(\tilde x -x)\, \partial _y \partial _{x} F(y|x) } \\
 & = O_{\mathbb P} \left(h_{n,1}^{2-l_2}  \sqrt {\frac{\abs{\log h_{n,1}}}{n h_{n,1}^{} }}+h_{n,1}^{2-l_2} \right)\\
 &= O_{\mathbb P} \left(h_{n,1}^{2-l_2}  \right).
 \end{align*}
By~(G\ref{cond:smoothnessdensity1}), one has  
\begin{align*}
&\abs{\partial _y^2 F(y|x+u)-\partial _y^2  F(y|x) - u\, \partial _y^2  \partial _x F(y|x)} \leq M \frac {|u|^{1+\delta}}{1+\delta}.
\end{align*}  
Using~\eqref{rate:pnk2}, it then follows that
 \begin{align*}
\lefteqn{\sup_{x\in S_X,\, y\in\R}  \abs{\hat B_{n,2}^{(2,0)}(y,x)}}\\
& \leq \sup_{x\in S_X} \{ \hat r_{n,k}^{}(x)\}  \sup_{|x-\tilde x|<h_{n,1},\, y\in \R} \abs{ \partial _y^2 F(y|\tilde x) - \partial _y^2 F(y|x) - (\tilde x -x)\, \partial _y^2 \partial _{x} F(y|x) }  \\
 & = O_{\mathbb P} \left(h_{n,1}^{1+\delta} \sqrt {\frac{\abs{\log h_{n,1}}}{n h_{n,1}^{} }}+h_{n,1}^{1+\delta} \right)\\
  & = O_{\mathbb P} \left(h_{n,1}^{1+\delta}  \right).
 \end{align*}
Finally, by~\eqref{rate:pnk}, we have $ \hat B_{n,3}^{(1,1)}(y) = O_{\mathbb P} (h_{n,1}\sqrt {{\abs{\log h_{n,1}}}/(n h_{n,1})}+h_{n,1})=O_{\mathbb P} (h_{n,1} )$. Putting all this together yields 
\begin{align*}
&\sup_{x\in S_X,\, y\in \R} \abs{ \partial _{y}^{}  \hat b_{n,k}(y,x) } =O_{\mathbb P} \left(h_{n,1}^{2} +h_{n,2}^2\right), \\
&\sup_{x\in S_X,\, y\in \R} \abs{ \partial _{y}^{2}  \hat b_{n,k}(y,x)} =O_{\mathbb P} \left(h_{n,1}^{1+\delta} + h_{n,2}^{1+\delta}\right),\\
&\sup_{x\in S_X,\, y\in \R} \abs{ \partial _{y}^{}\partial _{x}^{}  \hat b_{n,k}(y,x) } =O_{\mathbb P} \left(h_{n,2}^2h_{n,1}^{-1}+h_{n,1}\right).
\end{align*}

\paragraph{Back to equation (\ref{equation:decompuniform}).}
So far, we have obtained uniform convergence rates in probability for $\hat \beta_{n,k}$, $\hat b_{n,k}$ and $\hat s_{n,k}$ (with $k \in \{0, 1\}$ for $\hat \beta_{n,k}$ and $\hat b_{n,k}$ and $k \in \{1, 2\}$ for $\hat s_{n,k}$), and their first and second-order partial derivatives. In combination with equation~\eqref{equation:decompuniform} and hypothesis~(H3), these yield the convergence rates for $\hat F_n(y|x)  -  F(y|x)$ stated in the proposition. For the benefit of the reader, we provide here the details. For $l \in \{0, 1, 2\}$, we have, since $\abs{ \log h_{n,1} } / (nh_{n,1}^3) = o(1)$ by (G\ref{cond:bandwidth}),
\begin{align*}
  \sup_{x \in S_X} \abs{ \partial_x^l \hat{s}_{n,k}(x) } 
  &= 
  \begin{cases}
    O_{\prob}(1) & \text{if $l \in \{0, 1\}$,} \\
    O_{\prob}\left( \sqrt{ \frac{\abs{\log h_{n_1}}}{n h_{n_1}^5} } + 1 \right) & \text{if $l = 2$,}
  \end{cases}
  \\
 \sup_{x \in S_X, y \in \R} \abs{ \partial_x^l \hat{\beta}_{n,k}(y, x) }
  &=
  O_{\prob} \left( \sqrt{ \frac{ \abs{ \log h_{n_1}} }{ n h_{n_1}^{1+2l} } } \right), \\
  \sup_{x \in S_X, y \in \R} \abs{ \partial_x^l \hat{b}_{n,k}(y, x) }
  &=  O_{\mathbb P} \left(h_{n,1}^{2-l} +h_{n,1}^{-l}h_{n,2}^2 \right).
\end{align*}
Moreover,
\begin{align*}
  \sup_{x \in S_X,\ y \in \R} \abs{ \partial_y \hat{b}_{n,k}(x, y) }
  &=
  O_{\prob}(h_{n,1}^2+h_{n,2}^2), \\
  \sup_{x \in S_X,\ y \in \R} \abs{ \partial^2_y \hat{b}_{n,k}(x, y) }
  &=
  O_{\mathbb P} \left(h_{n,1}^{1+\delta} + h_{n,2}^{1+\delta}\right), \\
  \sup_{x \in S_X,\ y \in \R} \abs{ \partial_y \partial_x \hat{b}_{n,k}(x, y) }
  &=
O_{\mathbb P} \left(h_{n,2}^2h_{n,1}^{-1}+h_{n,1}\right).
\end{align*}
Finally, for $(l_1, l_2)$ equal to $(1, 0)$, $(1, 1)$ or $(2, 0)$,
\begin{align*}
  \sup_{x \in S_X,\ y \in \R} \abs{ \partial_y^{l_1} \partial_x^{l_2} \hat{\beta}_{n,k}(y, x) }
  &=
  O_{\mathbb P}\left(\sqrt {\frac{\abs{\log h_{n,1}h_{n,2}}}{n h_{n,1}^{1+2 l_2 }h_{n,2}^{1+2(l_1-1) } }} \right).
\end{align*}
Using equation \eqref{equation:decompuniform}, we find the following uniform convergence rates for $\hat F_n(y|x) - F(y|x)$ and its derivatives. First, differentiating both sides of equation~\eqref{equation:decompuniform} $l \in \{0, 1, 2\}$ times with respect to $x$, we have, using $\abs{ \log h_{n,1} } / (nh_{n,1}^3) = o(1)$,
\[
  \sup_{x \in S_X,\ y \in \R} \abs{ \partial_x^l \hat F_n(y|x) - \partial_x^l F(y|x) }
  =
  O_{\prob} \left( \sqrt{ \frac{ \abs{ \log h_{n,1} } }{ n h_{n,1}^{1+2l} } } + h_{n,1}^{2-l} +h_{n,1}^{-l}h_{n,2}^2\right).
\]
Second, differentiating once or twice with respect to $y$, we find
\begin{align*}
  \sup_{x \in S_X,\ y \in \R} \abs{ \partial_y \hat F_n(y|x) - \partial_y F(y|x) }
  &=
  O_{\prob}  \left( \sqrt {\frac{\abs{\log h_{n,1}h_{n,2}}}{n h_{n,1}h_{n,2} }} +h_{n,1}^{2}+h_{n,2}^{2} \right), \\
  \sup_{x \in S_X,\ y \in \R} \abs{ \partial_y^2 \hat F_n(y|x) - \partial_y^2 F(y|x) }
  &=
  O_{\prob}  \left( \sqrt {\frac{\abs{\log h_{n,1}h_{n,2}}}{n h_{n,1}h_{n,2}^3 }} +h_{n,1}^{1+\delta} + h_{n,2}^{1+\delta}\right).
\end{align*}
Finally, the rate for the mixed second-order partial derivative is
\[
  \sup_{x \in S_X,\ y \in \R} \abs{ \partial_y \partial_x \hat F_n(y|x) - \partial_y \partial_x F(y|x) }
  =
  O_{\prob}  \left( \sqrt {\frac{\abs{\log h_{n,1}h_{n,2}}}{n h_{n,1}^3h_{n,2} }} +h_{n,2}^2h_{n,1}^{-1}+h_{n,1} \right).
\]
This completes the proof of the proposition. \qed
}

\subsubsection{Proof of Proposition \ref{prop:quantiletransformunif}}

The first statement is a direct consequence of Fact \ref{prel:consistency_inverse}. On the event corresponding to $E_{n,j,1}$, by the mean-value theorem, we find
\begin{align*}
 \abs{\hat  F_{n} ^- ( u|x )  -  F^-(u|x) }
  &= \abs{  F^-\bigl( F \bigl( \hat  F_{n} ^- ( u|x )|x \bigr) \mid x \bigr)  -  F^-(u|x) }\nonumber\\
  & \leq \frac 1 b  \abs{  F\bigl(\hat  F_{n} ^- ( u|x )|x\bigr)  - u }.
\end{align*}
Conclude using the first statement of the proposition to obtain a uniform rate of convergence  of $ \sqrt{{\abs{\log h_{n,1}}}/{n {h_{n,1}} }} +{h_{n,1}^{2}}+{h_{n,2}^{2}}$. \qed

\subsubsection{Proof of Proposition \ref{prop:regularity}}

Recall that $ \delta_1= \min(\alpha/2, \delta)$. We establish the occurrence with high probability of the following uniform bounds: for certain constants $M_1'$ and $M_1''$, to be determined later, 
\begin{align}
&\sup_{x\in S_X,\, u\in [\gamma,1-\gamma] } \abs{\partial_x \hat F_{n}^- (u  |x) } \leq M_1', \label{ineq:reg1} \\
&\sup_{x\neq x',\, u\in [\gamma,1-\gamma] } \frac{\abs{ \partial_x \hat F^-_n (u  |x)-\partial_x \hat F^-_n (u  |x')} }{|x-x'|^{\delta_1}} \leq M_1''  \label{ineq:reg2}.
\end{align}
The constant $M_1$ of the statement will be taken equal to the maximum between $M_1'$ and $M_1''$. 

By Lemma \ref{lemma:FoF^-}, we have, almost surely, $\hat F_n\big(\hat F_n^-(u|x)|x\big)=u$ for $u\in[\gamma,1-\gamma]$. Differentiating both sides of this equality with respect to $x$ produces the identity
\begin{align*}
\partial_x \hat F^-_n (u  |x) = -\left.\frac{\partial _x \hat F_n (y |x) }{\hat f _n (y|x) }\right|_{y = \hat F^- _n(u  |x) }.
\end{align*}
Using (G\ref{cond:smoothnessdensity1}), we have
\begin{align*}
&\abs{\hat f _n \big( \hat F_n^- (u  |x)|x\big)  - f \big( F^- (u  |x)|x\big)}\\
& \leq  \abs{\hat f_n  \big( \hat F_n^- (u  |x)|x\big)  - f \big( \hat F_n^- (u  |x)|x\big)}+ M \abs{   \hat F_n^- (u  |x)  -   F^- (u  |x)}\\
&\leq \sup_{x'\in S_X,\, y\in \R} \abs{ \hat f_n  ( y|x')  - f (y|x')}+ M \sup_{x'\in S_X,\, u\in [\gamma,1-\gamma]}\abs{  \hat F_n^- (u  |x')  -   F^- (u  |x')}.
\end{align*}
Invoking Proposition~\ref{prop:uniformCVderivatives} (using  $n h_{n,1}^2 / \abs{\log h_{n,1}} \to \infty$) and Proposition~\ref{prop:quantiletransformunif}, we get
\begin{align}\label{cv:denom}
&\sup_{x\in S_X,\, u\in [\gamma,1-\gamma] } \abs{\hat f _n \big( \hat F_n^- (u  |x)|x\big)  - f \big( F^- (u  |x)|x\big)}  =o_{\mathbb P}(1).
\end{align}
Now we introduce  $\hat g_n(y,x) = \partial_x \hat F_n(y|x)$ and $g(y,x) = \partial_x  F(y|x)$. In a similar way, invoking Proposition~\ref{prop:uniformCVderivatives} (using  $n h_{n,1}^{3} / \abs{\log h_{n,1}} \to \infty$ and $h_{n,1}^{-1}h_{n,2}^2\to 0$) and Proposition~\ref{prop:quantiletransformunif}, we find
\begin{align}\label{cv:num}
&\sup_{x\in S_X,\, u\in [\gamma,1-\gamma] } \abs{  \hat g_n \big(\hat F_n^- (u  |x) , x\big)  - g \big(F^- (u  |x),x\big)   } = o_{\mathbb P}(1).
\end{align}
It follows that, with probability going to $1$, 
\begin{align*}
\sup_{x\in S_X,\, u\in [\gamma,1-\gamma] } \abs{ \partial_x \hat F_{n}^- (u  |x) } \leq \frac{2M}{b_\gamma/2}.
\end{align*}
Hence we have shown \eqref{ineq:reg1} with $M_1' = 2M / (b_\gamma/2)$. To show \eqref{ineq:reg2}, we start by proving that \eqref{ineq:reg2} holds true whenever the maps $(y ,x)\mapsto \partial_x \hat F_{n} (y  |x)$ and $(y ,x)\mapsto \hat f_{n} (y  |x)$ are both $\delta_1$-H\"older with constant $2M$, i.e., when
\begin{align}
&\sup_{x\neq x',\, y\neq y' } \frac{\abs{  \partial_x \hat F_n (y  |x)-\partial_x \hat F_n (y'  |x')}}{|(y-y',x-x')|^{\delta_1}} \leq 2 M , \label{ineq:holder1}\\
&\sup_{x\neq x',\, y\neq y' }\frac{ \abs{   \hat f_n (y  |x)-\hat f_n (y'  |x')} }{|(y-y',x-x')|^{\delta_1}} \leq 2 M \label{ineq:holder2} .
\end{align}
 Abbreviate $\hat F^-_n (u  |x) $  to $ \xi_{u,x}$, and write
\begin{align*}
&\partial_x \hat F^-_n (u  |x)-\partial_x \hat F^-_n (u  |x') \\
&=\frac{g( \xi_{u,x'},x') - g ( \xi_{u,x} ,x)  }{\hat f _n ( \xi_{u,x'} |x') } + \frac{ g ( \xi_{u,x} ,x)  }{\hat f _n ( \xi_{u,x'}|x') \hat f _n ( \xi_{u,x}|x) } \bigl( \hat f _n ( \xi_{u,x}|x) -\hat f _n (\xi_{u,x'}|x') \bigr).
\end{align*}
Use \eqref{cv:denom} and \eqref{cv:num} to obtain, with probability going to $1$,
\begin{align*}
\abs{ \partial_x \hat F^-_n (u  |x)-\partial_x \hat F^-_n (u  |x')} \leq \frac{1}{b_\gamma/2} \abs{ g( \xi_{u,x'},x') - g ( \xi_{u,x} ,x) } + \frac{ 2 M }{(b_\gamma/2)^2 }\abs{ \hat f _n ( \xi_{u,x}|x) -\hat f _n (\xi_{u,x'}|x')}.
\end{align*}
Use the H\"older properties \eqref{ineq:holder1} and \eqref{ineq:holder2}, then the equivalence between $l_p$ norms, $p>0$, and finally \eqref{ineq:reg1}, to get
\begin{align*}
\abs{ \partial_x \hat F^-_n (u  |x)-\partial_x \hat F^-_n (u  |x')} &\leq \frac{2M}{b_\gamma/2} |( \xi_{u,x'}-\xi_{u,x} ,x' -x) |^{\delta_1} + \frac{ (2 M)^2 }{(b_\gamma/2)^2 }|^{} ( \xi_{u,x}-\xi_{u,x'} ,x -x')|^{\delta_1}\\
&= \frac{2M}{b_\gamma/2}\left(1+\frac{2M}{b_\gamma/2}\right) |( \xi_{u,x'}-\xi_{u,x} ,x' -x) |^{\delta_1}\\
&\leq \frac{2M}{b_\gamma/2}\left(1+\frac{2M}{b_\gamma/2}\right) M_{\delta_1} \left( | \xi_{u,x'}-\xi_{u,x}|^{\delta_1}  + |  x' -x |^{\delta_1}\right)\\
&\leq  \frac{2M}{b_\gamma/2}\left(1+\frac{2M}{b_\gamma/2}\right)^2 M_{\delta_1}   |  x' -x |^{\delta_1} ,
\end{align*}
for some positive constant $M_{\delta_1}$.

To terminate the proof, we just need to show \eqref{ineq:holder1} and \eqref{ineq:holder2}. We only focus on the former because the treatment of the latter results in the same approach involving slightly weaker conditions. Because the function spaces $\mathcal C_{s,M}(S_X)$, $s\in \mathbb R$, are decreasing sets in $s$, (G\ref{cond:smoothnessdensity1}) still holds with  $\delta_1= \min(\alpha/2, \delta)$ in place of $\delta$. We shall apply Proposition \ref{prop:uniformCVderivatives} with $\delta_1$ in place of $\delta$. Write $\hat m_n(y ,x) =\partial_x \{\hat F_{n} (y  |x)- F (y  |x)\}$. We distinguish between the case that $|(y - y', x- x')|$ is smaller than $h_{n,1}$ or not.

\begin{itemize}
\item  First, suppose $| (y -y ',x-x')|\leq h_{n,1}$. By the mean-value theorem and because of the rates associated to $\partial_{x}^2$ and $\partial_x\partial_y$ in Proposition \ref{prop:uniformCVderivatives}, we have
\begin{align*}
&\frac{\abs{ \hat m_n (y  ,x) -\hat m_n (y ' ,x') }}{| (y-y ',x-x')|^{\delta_1 }}\\
&\leq \sup_{x\in S_X,\, y\in \R} \abs{ \nabla \{\hat m_n (y ,x)\} } \, | (y-y ',x-x')|^{1-\delta_1}\\
&  \leq \sup_{x\in S_X,\, y\in \R}  \abs{ \nabla \{\hat m_n (y ,x)\}  } \, h_{n,1}^{1-\delta_1}  \\
&= O_{\mathbb P}\left( \sqrt{\frac{|\log h_{n,1}|}{nh_{n,1}^{3+2\delta_1}}}+ h_{n,1}^{1-\delta_1}+ h_{n,1}^{-1-\delta_1}h_{n,2}^2 +\sqrt{\frac{|\log h_{n,1}h_{n,2}|}{nh_{n,1}^{1+2\delta_1}h_{n,2}}} +h_{n,1}^{2-\delta_1}+h_{n,1}^{1-\delta_1} h_{n,2}^2\right).
\end{align*}
Because $\delta_1= \min(\alpha/2, \delta)$, (G\ref{cond:bandwidth2}'') still holds with $2\delta_1$ in place of $\alpha$. This implies that the previous bound goes to $0$.
\item  Second, suppose $| (y -y ',x-x')|>h_{n,1}$. Using the rates associated to $\partial_{x}$ in Proposition~\ref{prop:uniformCVderivatives}, we find
\begin{align*}
\frac{\abs{  \hat m_n(y ,x) -\hat m_n(y ',x') } }{| (y -y ',x-x')|^{\delta_1}}&\leq 2\sup_{x\in S_X,\, y\in \R} \abs{  \hat m_n(y ,x) }  h_{n,1}^{-\delta_1} \\
&= O_{\mathbb P}\left( \sqrt{\frac{|\log h_{n,1}|}{nh_{n,1}^{3+2\delta_1}}} + h_{n,1}^{1-\delta_1}+ h_{n,1}^{-1-\delta_1}h_{n,2}^2\right),
\end{align*}
which also goes to $0$ by (G\ref{cond:bandwidth2}'').
\end{itemize}

As a consequence of (G\ref{cond:bandwidth1bis}'), the previous bounds convergence to $0$. Now because $\partial _x F $ belongs to $\mathcal C_{\delta_1,M}(\R\times S_X)$, we have, with probability going to $1$,
\begin{align*}
 \frac{\abs{ \partial_x \hat F_n (y  |x)-\partial_x \hat F_n (y'  |x') } }{| (y -y ',x-x')|^{\delta_1 }} &\leq \frac{ \abs{  \hat m_n (y  |x)- \hat m_n (y'  |x')} }{| (y -y ',x-x')|^{\delta_1 }}+ \frac{\abs{ \partial_x F (y  |x)-\partial_x  F(y'  |x')}}{| (y -y ',x-x')|^{\delta_1 }} \\
 &\leq 2M,
\end{align*}
as required. \qed

\section{Analytical results}
\begin{lemma}
\label{lem:ddotC}
If $C$ is a bivariate copula satisfying~(G\ref{cond:copula_smoothness}), then for all $\gamma \in (0, 1/2)$, all $\bu \in [\gamma, 1-\gamma]^2$ and all $\bv \in [0, 1]^2$,
\[
  \abs{ C( \bv ) - C( \bu ) - \sum_{j=1}^2 (v_j - u_j) \, \dot{C}_j( \bu ) }
  \le  \frac{4\kappa}{\gamma} \sum_{j=1}^2 (v_j - u_j)^2.
\]
\end{lemma}

In the lemma, the point $\bv$ is allowed to lie anywhere in $[0, 1]^2$, even on the boundary or at the corner. The `anchor point' $\bu$, however, must be at a distance at least $\gamma$ away from the boundary.

\begin{proof}
Fix $\gamma \in (0, 1/2)$, $\bu \in [\gamma, 1-\gamma]^2$ and $\bv \in [0, 1]^2$. For $t \in [0, 1]$, put 
\[ 
  \bw(t) = \bu + t ( \bv - \bu ). 
\]
Note that $\bw(t) \in (0, 1)^2$ for $t \in [0, 1)$. The function $t \mapsto C( \bw(t) )$ is continuous on $[0, 1]$ and is continuously differentiable on $(0, 1)$. By the fundamental theorem of calculus,
\begin{align*}
  C( \bv ) - C( \bu )
  = C( \bw(1) ) - C( \bw(0) ) 
  = \int_0^1 \frac{\diff C( \bw(t) )}{\diff t}  \, \diff t 
  = \int_0^1 \sum_{j=1}^2 (v_j - u_j) \, \dot{C}_j( \bw(t) ) \, \diff t.
\end{align*}
It follows that
\[
  C( \bv ) - C( \bu ) - \sum_{j=1}^2 (v_j - u_j) \, \dot{C}_j( \bu )
  = \sum_{j=1}^2 (v_j - u_j) \int_0^1 \{ \dot{C}_j( \bw(t) ) - \dot{C}_j( \bu ) \} \, \diff t.
\]
Fix $t \in [0, 1)$ and $j \in \{1, 2\}$. Note that
\[
  \bu + s ( \bw(t) - \bu ) = \bu + st ( \bv - \bu ) = \bw(st).
\]
The function $s \mapsto \dot{C}_j( \bw(st) )$ is continuous on $[0, 1]$ and continuously differentiable on $(0, 1)$. By the fundamental theorem of calculus,
\begin{align*}
  \dot{C}_j( \bw(t) ) - \dot{C}_j( \bu )
  = \int_0^1 \frac{\diff \dot{C}_j( \bw(st) )}{\diff s} \, \diff s
  = \sum_{k=1}^2 t(v_k - u_k) \int_0^1 \ddot{C}_{jk}( \bw(st) ) \, \diff s.
\end{align*}
We obtain
\begin{multline*}
  C( \bv ) - C( \bu ) - \sum_{j=1}^2 (v_j - u_j) \, \dot{C}_j( \bu )   = \sum_{j=1}^2 \sum_{k=1}^2 (v_j - u_j) (v_k - u_k) \int_0^1 \int_0^1 \ddot{C}_{jk}( \bw(st) ) \, t \, \diff t \, \diff s.
\end{multline*}
Suppose we find a positive constant $L$ such that, for all $\bu \in [\gamma, 1-\gamma]^2$, $\bv \in [0, 1]^2$ and $j, k \in \{1, 2\}$,
\begin{equation}
\label{eq:iintddot}
  \abs{ \int_0^1 \int_0^1 \ddot{C}_{jk}( \bw(st) ) \, t \, \diff t \, \diff s } \le L.
\end{equation}
Then
\begin{align*}
  \abs{ C( \bv ) - C( \bu ) - \sum_{j=1}^2 (v_j - u_j) \, \dot{C}_j( \bu ) }
  &\le L \sum_{j=1}^2 \sum_{k=1}^2 \abs{v_j - u_j} \abs{v_k - u_k} \\
  &= L \, \bigl( \abs{v_1 - u_1} + \abs{v_2 - u_2} \bigr)^2 \\
  &\le 2L \, \bigl( (v_1 - u_1)^2 + (v_2 - u_2)^2 \bigr),
\end{align*}
which is the inequality stated in the lemma with $2L$ in place of $4\kappa/\gamma$. It remains to show \eqref{eq:iintddot}. By Condition~(G\ref{cond:copula_smoothness}),
\[
  \abs{ \ddot{C}_{jk}( \bw(st) ) }
  \le \kappa \, \{ w_j \, (1 - w_j) \, w_k \, (1-w_k) \}^{-1/2},
\]
where $w_j = w_j(st)$ is the $j$-th coordinate of $w(st)$. 
Now
\begin{align*}
  w_j (1-w_j)
  &= \{ u_j + st( v_j - u_j ) \} \, \{ 1 - u_j - st (v_j - u_j ) \}.
\end{align*}
For fixed $st \in [0, 1)$, this expression is concave as a function of $(u_j, v_j) \in [\gamma, 1-\gamma] \times [0, 1]$. Hence it attains its minimum for $(u_j, v_j) \in \{ \gamma, 1-\gamma \} \times \{0, 1\}$. In each of the four possible cases, we find
\[
  w_j (1-w_j) \ge \frac{1}{2} (1-st) \gamma.
\]
We obtain
\[
  \abs{ \ddot{C}_{jk}( \bw(st) ) }
  \le \frac{2\kappa}{(1-st) \gamma}.
\]
As a consequence,
\[
  \int_0^1 \int_0^1 \abs{ \ddot{C}_{jk}( \bw(st) ) } \, t \, \diff t \, \diff s
  \le \frac{2\kappa}{\gamma} \int_0^1 \int_0^1 \frac{t}{1 - st} \, \diff t \, \diff s
  = \frac{2\kappa}{\gamma},
\]
by direct calculation of the double integral.
\end{proof}

We quote Lemma~4.3 in \cite{segers:15}, which is a variant of ``Vervaat's lemma'', i.e., the functional delta method for the mapping sending a monotone function to its inverse:

\begin{lemma}
\label{lem:Vervaat:random}
Let $F : \R \to [0, 1]$ be a continuous distribution function. Let $0 < r_n \to \infty$ and let $\hat{F}_n$ be a sequence of random distribution functions such that, in $\ell^\infty(\R)$,
\begin{equation*}
  r_n ( \hat{F}_n - F ) \dto \beta \circ F, \qquad n \to \infty,
\end{equation*}
where $\beta$ is a random element of $\ell^\infty([0, 1])$ with continuous trajectories. Then $\beta(0) = \beta(1) = 0$ almost surely and
\begin{equation*}
  \sup_{u \in [0, 1]} \abs{ r_n \{ F(\hat{F}_n\inv(u)) - u \} + r_n \{ \hat{F}_n(F\inv(u)) - u \} } = o_{\mathbb P}(1).
\end{equation*}
As a consequence, in $\ell^\infty([0, 1])$,
\begin{equation*}
  \bigl( r_n \{ F( \hat{F}_n\inv(u) ) - u \} \bigr)_{u \in [0, 1]}
  \dto - \beta, \qquad n \to \infty.
\end{equation*}
\end{lemma}

\begin{lemma}
\label{lemma:FoF^-}
Let $F$ be a continuous function such that $\lim_{y\rightarrow -\infty} F(y) =0$ and $\lim_{y\rightarrow +\infty} F(y) =1$, then for any $u\in (0,1)$, we have $F(F^-(u)) = u$.
\end{lemma}

\begin{proof}
Let $y_0=F^-(u)$. By assumption, the set $\{y \in \reals \;:\; F(y) \geq u\} $ is non empty and therefore $-\infty <y_0<+\infty $. By definition of the quantile transformation, for any $y<y_0$, it holds $F(y)<u$. Now using the continuity of $F$  gives $F(y_0)\leq u$. Conclude by noting that we always have $F(y_0)\geq u$.
\end{proof}

\begin{lemma}
\label{lemma:bracketingquantile}
Let $F$ and $G$ be cumulative distribution functions. If there exists $\eps > 0$ such that $\abs{ F(y) - G(y) } \le \epsilon$ for every $y\in \R$, then
\begin{align}
\label{eq:bracketingquantile:1}
  G^-( (u-\epsilon) \vee 0) &\leq F^-(u), &&u \in [0, 1], \\
\label{eq:bracketingquantile:2}
  F^-(u) &\leq G^-(u+\eps), &&u \in [0, 1-\eps].
\end{align}
\end{lemma}

\begin{proof}
We first show \eqref{eq:bracketingquantile:1}. If $F^-(u) = \infty$, there is nothing to show, while if $F^-(u) = -\infty$, then $u \le F(F^-(u)) = F(-\infty) = 0$, so $G^-((u-\eps) \vee 0) = G^-(0) = -\infty$ too. Hence we can suppose that $F^-(u)$ is finite. Since $G(y) \ge F(y) - \epsilon$ for all $y \in \reals$, we have $G(F^-(u)) \ge F(F^-(u)) - \epsilon \ge u - \epsilon$. Trivially, also $G(F^-(u)) \ge 0$. Together, we find $G(F^-(u)) \ge (u - \epsilon) \vee 0$. As a consequence, $F^-(u) \ge G^-((u - \epsilon) \vee 0)$.

Next we show \eqref{eq:bracketingquantile:2}. Let $u \in [0, 1-\eps]$. By \eqref{eq:bracketingquantile:1} with the roles of $F$ and $G$ interchanged and applied to $u+\eps$ rather than to $u$, we find $F^-(u) = F^-(((u+\eps) - \eps) \vee 0) \le G^-(u+\eps)$.
\end{proof}

\end{appendices}

\section*{Acknowledgments}

The research by F. Portier was supported by the Fonds de la Recherche Scientifique Belgium (F.R.S.-FNRS) A4/5 FC 2779/2014-2017 No.\ 22342320, and the research by J. Segers was supported by the contract ``Projet d'Act\-ions de Re\-cher\-che Concert\'ees'' No.\ 12/17-045 of the ``Communaut\'e fran\c{c}aise de Belgique'' and by IAP research network Grant P7/06 of the Belgian government (Belgian Science Policy).

The authors greatfully acknowledge valuable comments by the reviewers and the editors, in particular the helpful suggestions related to the structure of the theory and the presentation of the paper.


\setlength{\bibsep}{0.3ex plus 0.3ex}

\bibliographystyle{chicago}

\end{document}